\documentclass[reqno,a4paper]{amsart}
\usepackage{color}
\usepackage[colorlinks=true,allcolors=blue,backref=page]{hyperref}
\usepackage{amsmath, amssymb, amsthm}
\usepackage{mathrsfs}
\usepackage{mathtools}
\usepackage[noabbrev,capitalize,nameinlink]{cleveref}

\crefname{equation}{}{}
\usepackage{fullpage}
\usepackage[noadjust]{cite}
\usepackage{graphics}
\usepackage{graphicx}
\usepackage[labelfont=bf,labelsep=period]{caption}
\usepackage{pifont}
\usepackage{tikz, pgf}
\usetikzlibrary{math}
\usepackage{bbm}
\usepackage[T1]{fontenc}

\usetikzlibrary{arrows.meta}

\usepackage{environ}
\usepackage{framed}
\usepackage{url}
\usepackage[linesnumbered,ruled,vlined]{algorithm2e}
\usepackage[noend]{algpseudocode}
\usepackage{cite}
\usepackage{framed}
\usepackage[framemethod=tikz]{mdframed}
\usepackage{appendix}
\usepackage{graphicx}
\usepackage[textsize=tiny]{todonotes}
\usepackage{tcolorbox}
\usepackage{enumerate}
\allowdisplaybreaks[1]
\usepackage{enumerate}
\usepackage[shortlabels]{enumitem}

\crefformat{enumi}{#2#1#3}
\crefrangeformat{enumi}{#3#1#4 to~#5#2#6}
\crefmultiformat{enumi}{#2#1#3}
{ and~#2#1#3}{, #2#1#3}{ and~#2#1#3}

\crefname{algocf}{Algorithm}{Algorithms}

\crefname{equation}{}{} 
\AtBeginEnvironment{appendices}{\crefalias{section}{appendix}} 

\usepackage[color,final]{showkeys} 

\colorlet{refkey}{orange!20}
\colorlet{labelkey}{blue!30}

\crefname{algocf}{Algorithm}{Algorithms}

\numberwithin{equation}{section}
\newtheorem{theorem}{Theorem}[section]

\newtheorem{lemma}[theorem]{Lemma}
\newtheorem{claim}[theorem]{Claim}
\crefname{claim}{Claim}{Claims}

\newtheorem{corollary}[theorem]{Corollary}

\newtheorem*{question*}{Question}
\newtheorem{fact}[theorem]{Fact}

\theoremstyle{definition}
\newtheorem{definition}[theorem]{Definition}

\newtheorem{question}[theorem]{Question}
\newtheorem*{definition*}{Definition}

\theoremstyle{remark}
\newtheorem{remark}[theorem]{Remark}
\newtheorem*{remark*}{Remark}

\newcommand{\mb}{\mathbb}
\newcommand{\mbf}{\mathbf}
\newcommand{\mbm}{\mathbbm}
\newcommand{\mc}{\mathcal}

\newcommand{\mr}{\mathrm}
\newcommand{\ol}{\overline}
\newcommand{\on}{\operatorname}

\newcommand{\eps}{\varepsilon}

\let\originalleft\left
\let\originalright\right
\renewcommand{\left}{\mathopen{}\mathclose\bgroup\originalleft}
\renewcommand{\right}{\aftergroup\egroup\originalright}

\allowdisplaybreaks

\title{The Exact Rank of Sparse Random Graphs}

\author[A1]{Margalit Glasgow}

\address{Department of Computer Science, Stanford University, Stanford, CA.}

\email{mglasgow@stanford.edu}

\author[A2]{Matthew Kwan}

\address{Institute of Science and Technology Austria (ISTA).}

\email{matthew.kwan@ist.ac.at}

\author[A3]{Ashwin Sah}
\author[A4]{Mehtaab Sawhney}
\address{Department of Mathematics, Massachusetts Institute of Technology, Cambridge, MA 02139, USA}
\email{\{asah,msawhney\}@mit.edu}

\begin{document}

\begin{abstract}
Two landmark results in combinatorial random matrix theory, due to
Koml\'os and Costello--Tao--Vu, show that discrete random matrices
and symmetric discrete random matrices are typically nonsingular.
In particular, in the language of graph theory, when $p$ is a fixed
constant, the biadjacency matrix of a random Erd\H{o}s--R\'enyi bipartite
graph $\mathbb{G}(n,n,p)$ and the adjacency matrix of an Erd\H{o}s--R\'enyi
random graph $\mathbb{G}(n,p)$ are both nonsingular with high probability.
However, very sparse random graphs (i.e., where $p$ is allowed to
decay rapidly with $n$) are typically singular, due to the presence
of ``local'' dependencies such as isolated vertices and pairs of
degree-1 vertices with the same neighbour.

In this paper we give a combinatorial description of the rank of a
sparse random graph $\mathbb{G}(n,n,c/n)$ or $\mathbb{G}(n,c/n)$
in terms of such local dependencies, for all constants $c\ne e$ (and
we present some evidence that the situation is very different for
$c=e$). 
This gives an essentially complete answer to a question raised
by Vu at the 2014 International Congress of Mathematicians.

As applications of our main theorem and its proof, we also determine the asymptotic singularity probability of the 2-core of a sparse random graph, we show that the
rank of a sparse random graph is extremely well-approximated by its matching
number, and we deduce a central limit theorem for the rank of $\mathbb{G}(n,c/n)$. 
\end{abstract}

\maketitle

\section{Introduction} \label{sec:introduction}
\global\long\def\core{\operatorname{core}}
\global\long\def\corank{\operatorname{corank}}
\global\long\def\rank{\operatorname{rank}}
\global\long\def\supp{\operatorname{supp}}
\global\long\def\lam2{\lambda_{\mathrm{2}}}
\global\long\def\lamKS{\lambda_{\mathrm{KS}}}

A foundational theorem in combinatorial random matrix theory,
due to Koml\'os~\cite{Kom67,Kom68}, says that discrete random
matrices with i.i.d.~entries are typically nonsingular (over $\mb{R}$). In particular, let $B$ be an
$n\times n$ random matrix with i.i.d.~$\operatorname{Bernoulli}(p)$
entries (meaning that each entry $B_{ij}$ satisfies $\Pr[B_{ij}=1]=p$
and $\Pr[B_{ij}=0]=1-p$). For any constant $p\in(0,1)$, if we take $n\to\infty$ then
such a random matrix is nonsingular \emph{with high probability} (``whp''
for short): that is, $\lim_{n\to\infty}\Pr[B\text{ is singular}]=0$.

A huge number of strengthenings and variations of Koml\'os' theorem
have been considered over the years. Two particular highlights include a result
of Tikhomirov~\cite{Tik20} that for constant $0<p\le1/2$, the
singularity probability is $(1-p+o(1))^{n}$, and a result of Costello,
Tao, and Vu~\cite{CTV06} that \emph{symmetric} discrete random matrices
are also nonsingular whp. A symmetric binary matrix can be interpreted
as the adjacency matrix of a graph, so the Costello--Tao--Vu theorem
has an interpretation in terms of random graphs: for constant $p\in(0,1)$,
an Erd\H{o}s--R\'enyi random graph $G\sim\mb{G}(n,p)$ has nonsingular
adjacency matrix whp\footnote{There is a slight difference between a random symmetric Bernoulli
matrix and the adjacency matrix of a random graph: namely, the adjacency
matrix of any graph has zeroes on the diagonal. However, the same
techniques usually apply to both settings, and we will not further
concern ourselves with this detail.}. Actually, Koml\'os' theorem can be interpreted in graph-theoretic
terms as well: the random matrix $B$ described above can be interpreted
as the \emph{biadjacency} matrix of a \emph{bipartite} Erd\H{o}s--R\'enyi
random graph $G\sim\mb G(n,n,p)$ (where one of the parts corresponds
to the rows of the matrix, and the other part corresponds
to the columns).

If $p$ decays too rapidly with $n$ (in particular, if $p\le(1-\varepsilon)\log n/n$
for some constant $\varepsilon>0$), then for reasons related to the
coupon collector problem, a typical outcome of $G\sim\mb{G}(n,p)$ (respectively,
$G\sim\mb{G}(n,n,p)$) has isolated vertices, meaning that its adjacency
matrix (respectively, biadjacency matrix) has all-zero rows and is
therefore singular. In fact, $\log n/n$ is a \emph{sharp
threshold} for singularity, in the sense that if $p\ge(1+\varepsilon)\log n/n$
(and $p$ is bounded away from $1$) then a typical $G\sim\mb G(n,p)$
(respectively, $G\sim\mb G(n,n,p)$) has nonsingular adjacency matrix
(respectively, nonsingular biadjacency matrix). This seems to have been first
observed by Costello and Vu~\cite{CV08}\footnote{The Costello--Vu proof was only written for $\mb G(n,p)$, but it can be easily adapted to $\mb G(n,n,p)$; alternatively, see \cite{FKS21} for a very simple proof in the $\mb G(n,n,p)$ case.}, and refinements
and generalisations were proved by Basak and Rudelson~\cite{BR18}
and Addario-Berry and Eslava~\cite{AE14}. In particular, the latter
authors proved a sharp \emph{hitting time} type
result: if we consider
the random graph \emph{process} where we start with the empty graph
on $n$ vertices (or the empty bipartite graph with $n+n$ vertices)
and add random edges one-by-one (respecting our bipartition, in the
bipartite case), then whp at the very same moment where the last isolated
vertex disappears our graph becomes nonsingular.

Na\"ively, it is quite surprising that the property of being singular (which is intrinsically an algebraic property of a matrix) can be predicted so effectively by the simple combinatorial property of having an isolated vertex. It is natural to ask whether this theme continues below the singularity threshold $\log n/n$: even when a random matrix is likely to be singular, can we describe the \emph{rank} via ``local combinatorial dependencies'' such as
isolated vertices? In their aforementioned paper, Costello and Vu~\cite{CV08} actually proved the first result along these lines: for $p\ge\left(1/2+\varepsilon\right)\log n/n$,
whp the rank of $\mb G(n,p)$ is precisely $n$ minus the number of isolated vertices. In follow-up work, Costello and Vu~\cite{CV10} considered the more general
regime where $p\ge c\log n/n$ for a constant $c>0$. They found that in this regime it is still possible to give a combinatorial description of the rank, though one must consider more sophisticated types of ``local dependencies'' than isolated vertices. For example, the next simplest type of dependency is a \emph{cherry}:
a pair of degree-$1$ vertices with the same neighbour. More recently,
DeMichele, the first author, and Moreira~\cite{DGM22} gave a combinatorial description
of the rank of $G\sim\mb G(n,p)$ and $G\sim\mb G(n,n,p)$, in terms of
a procedure that iteratively deletes local dependencies, which holds
whp whenever $\lim_{n\to\infty}np=\infty$ (i.e., when $p$ asymptotically
dominates $1/n$).

The most challenging regime is where $p=c/n$ for constant $c$. An asymptotic for the typical rank of $\mb G(n,c/n)$ was conjectured by Bauer and Golinelli~\cite{BG01} (motivated by statistical physics considerations), and this asymptotic was later proved by Bordenave, Lelarge, and Salez~\cite{BLS11} via analytic techniques.
In his lecture at the 2014 International Congress of Mathematicians~\cite{VuVideo} (also in \cite{VuSlides}),
Vu asked whether one can also give a precise combinatorial characterisation
of the rank in this regime. 

The main purpose of this paper is to provide an answer to Vu's question, and the analogous question for $\mb G(n,n,c/n)$, exactly characterising the rank of sparse random graphs (and in the process, providing a linear-time algorithm to compute the rank).

At a high level, we show that whp all linear dependencies arise from two different types of combinatorial structures. First, we need to account for ``tree-like'' structures generalising isolated vertices and cherries (previously identified in the work of Costello and Vu~\cite{CV10} and DeMichele, the first author, and Moreira~\cite{DGM22}), which can be ``peeled off'' by an iterative process called \emph{Karp--Sipser leaf removal} (defined in \cref{def:ks}). Second, we need to account for certain short cycles (which we call ``special cycles'', defined in \cref{def:special-cycles}), which cause linear dependencies for a different reason.

Our proof of this characterisation involves a wide range of tools and ideas, both original and adapted from existing work. This includes analysis of degree-constrained random graphs and of the \emph{Karp--Sipser leaf-removal algorithm}, robust analysis of random walks, spectral convergence machinery for locally convergent graphs, a ``rank-boosting'' technique, and some special-purpose notions of matrix pseudoinverses and ``minimal kernel vectors'' (all of which we describe further in \cref{sec:outline}). To try to give a rough impression of the most fundamental difficulty compared to previous work: note that the rank of a matrix can be interpreted as the size of its largest nonsingular submatrix. In the setting of most previous work, maximum nonsingular submatrices are in some sense ``robustly'' nonsingular (in particular, the corresponding subgraphs have good \emph{expansion} properties), which makes it possible to rule out certain types of kernel vectors via lossy union bounds. However, in our situation the largest nonsingular submatrices are in some sense ``only barely nonsingular'', with essentially the weakest possible expansion a nonsingular submatrix can have, and there is almost no room to make any kind of lossy approximation.

In any case, once one has a characterisation of the rank in terms of explicit combinatorial structures, it becomes possible to prove further results about the rank via combinatorial tools. Indeed, as corollaries of our main theorem and its proof, we are able to show a number of additional theorems: we compute the asymptotic singularity probability of the 2-core, we obtain a very strong bound on the difference between the matching number and the rank, and we prove a central limit theorem for the rank of $\mb G(n, c/n)$. Since the statement of our main result (\cref{thm:rank-characterisation}) is somewhat technical, we take a moment to discuss these corollaries before presenting the precise statement of our main result.

\subsection{Nonsingularity of the 2-core}\label{sub:nonsingularity-2-core}
The \emph{$k$-core} $\core_{k}(G)$ of a graph $G$ is
the subgraph obtained by iteratively deleting vertices with degree
less than $k$ (in any order). Equivalently, it is the largest induced subgraph with minimum degree at least $k$. This notion was first introduced in 1984 by Bollob\'as~\cite{Bol84}, and $k$-cores have since become fundamental objects of study in random graph theory.

In the context of combinatorial random matrix theory, an important reason to study $k$-cores is that all of the most obvious types of ``local dependencies'' involve vertices of low degree. For example, recall that isolated vertices and cherries are abundant types of local dependencies, and it turns out that all of the ``tree-like'' local dependencies mentioned earlier in this introduction contain a vertex of degree 1. Another example of a local dependency, which has non-negligible probability of appearing in the regime $p=c/n$, is a pair of vertices of degree
2 with the same neighbourhood (i.e., a cycle of length 4, in which a pair of opposite vertices have degree exactly 2).

Resolving a conjecture
of Vu, it was recently proved by Ferber and the last three authors~\cite{FKSS21} (see also \cite{DGM22}) that
for constants $k\ge3$ and $c>0$, the $k$-core of $\mb G(n,c/n)$
is nonsingular whp. 
That is to say, trimming low-degree vertices typically removes any singularity present in the graph (foreshadowing the main result of this paper, that whp the \emph{only} dependencies are ``tree-like''
or ``cycle-like'').

While the assumption $k\ge3$ is necessary for a ``whp'' result due to the possibility of ``cycle-like'' dependencies, these types of dependencies seem to be rare (e.g., the expected number of 4-cycles is only $O(1)$, and with non-negligible probability there are no 4-cycles at all). So, it is natural to ask whether one can still prove meaningful theorems about nonsingularity in the case $k=2$. We prove such a theorem: roughly speaking, the 2-core is ``right on the edge of singularity'', being neither singular whp nor nonsingular whp. (In retrospect, one can see that in the case $k\ge 3$, the $k$-core is actually ``quite far from being singular'' with respect to natural local dependencies, and this ``wiggle room'' played a crucial role in the proofs in \cite{FKSS21,DGM22}).

\begin{theorem}\label{thm:2-core}
Fix a constant $c>1$, let $G\sim\mathbb{G}(n,c/n)$, and let $A$ be the adjacency matrix of the 2-core of the largest component of $G$ (which is unique whp). Then
\[\lim_{n\to \infty}\Pr[A\emph{ is nonsingular}] =\left(\frac{\displaystyle{1-\left(\frac{\lam2}{e^{\lam2/2}-e^{-\lam2/2}}\right)^{4}}}{\displaystyle{1-\left(\frac{\lam2}{e^{\lam2}-1}\right)^{4}}}\right)^{1/4}>0,\]
where $\lam2=\lam2(c)>0$ is the unique solution to $\lam2/(1-e^{-\lam2})=c$. Moreover, the corank\footnote{The \emph{corank} of a matrix is the dimension of its kernel.} of $A$ has an asymptotic $\on{Poisson}(\mu)$ distribution, where $\mu$ is chosen such that the above probability is asymptotic to $e^{-\mu}$ (and in particular, the corank is bounded in probability\footnote{A sequence of random variables $(X_n)_{n=1}^\infty$ is said to be \emph{bounded in probability} if for all $\varepsilon>0$, there are $N,M$ such that $\Pr[X_n\ge M]\le \varepsilon$ for all $n\ge N$.}).
\end{theorem}

Note that the assumption $c>1$ corresponds to the celebrated \emph{phase transition} of the Erd\H{o}s--R\'enyi random graph. Indeed, if $c<1$ (the ``subcritical'' regime), whp all the connected components of $\mb G(n,c/n)$ have
size $O(\log n)$ (and each of them is a tree or is \emph{unicyclic}, having exactly one cycle); thus, in this regime the 2-core is rather trivial, consisting only of a very small number of isolated cycles. 
On the other hand, if $c>1$ (the ``supercritical'' regime), then whp $\mb{G}(n,c/n)$ has a unique ``giant component'' with nontrivial structure (whose number of vertices is of order $n$), in addition to some trees and unicyclic components of size $O(\log n)$. See for example the monographs \cite{JLR00,FK16,Bol01} for more details about the component phase transition of the Erd\H{o}s--R\'enyi random graph, and see \cite{DLP14} for a precise description of the ``anatomy'' of a supercritical random graph in terms of its 2-core.

\begin{remark*}
The statement of \cref{thm:2-core} is only about the giant component in the supercritical regime, but one may also wish to consider the entire 2-core (including any small-cycle components), in which case it makes sense to consider all $c>0$ (not just $c>1$). With the methods in this paper (and some results about critical random graphs~\cite{ABG10}) it is possible to show that for $G\sim\mb{G}(n,c/n)$ we have $\lim_{n\to\infty}\Pr[\text{the }2\text{-core of }G\text{ is nonsingular}]>0$ if and only if $c\ne 1$ (when $c=1$ there are too many nontrivial components, each of which is reasonably likely to be singular). 
We omit the details. (Also, note that the asymptotic nonsingularity probability in \cref{thm:2-core} tends to zero as $c\to 1$.)
\end{remark*}

\subsection{Comparing the rank and the matching number}\label{sub:rank-vs-matching-number}
In a graph $G$, a \emph{matching} is a collection of disjoint edges. The \emph{matching number} $\nu(G)$ is the maximum number of edges in a matching in $G$. If $G$ is bipartite, then $\nu(G)$ can be interpreted as the size of the largest permutation matrix ``contained'' in the biadjacency matrix $B(G)$ of $G$, where our notion of matrix containment allows deleting rows and columns, and changing 1-entries to 0-entries. Recalling the permutation definition of the determinant, $\nu(G)$ is a trivial upper bound for $\rank B(G)$. Confirming a conjecture of Lelarge~\cite{Lel13} motivated by statistical physics considerations, it was proved by Coja-Oghlan, Erg\"{u}r, Gao, Hetterich, and Rolvien~\cite{CEGHR20} that this trivial bound is nearly best-possible\footnote{Actually, they proved this for a much more general class of random matrices and for rank over any field.} for sparse random bipartite graphs: for $G\sim \mb G(n,n,c/n)$ we have $\nu(G)-\rank B(G)=o(n)$ whp.

In the non-bipartite case, there is no general inequality relating the rank of the adjacency matrix $\rank A(G)$ of a graph $G$ with its matching number $\nu(G)$, but a theorem of Bordenave, Lelarge, and Salez~\cite{BLS11} (mentioned earlier in the introduction) shows that nonetheless for a sparse random graph $G$ we have $\rank A(G)=2\nu(G)+o(n)$ whp (see also \cite{BLS13}). It is also natural to consider an alternative parameter $\sigma(G)$, defined to be the size of the largest permutation matrix ``contained'' in $A(G)$. This parameter $\sigma(G)$ has a combinatorial interpretation as the maximum number of vertices in a union of vertex-disjoint cycles and edges in $G$. Note that $2\nu(G)$ and $\rank A(G)$ are both at most $\sigma(G)$.

As our second result (a corollary of our main result \cref{thm:rank-characterisation}, to come), we dramatically improve the $o(n)$ error terms in the results described above, showing that $\nu(G)$ is an \emph{extremely} good approximation for $\rank B(G)$, and $2\nu(G)$ and $\sigma(G)$ are extremely good approximations for $\rank A(G)$, away from a ``critical point'' $p=e/n$. (The significance of this rather mysterious-sounding critical point will be explained later in this introduction; for now we just remark that this point also happens to be critical for several other spectral phenomena in Erd\H{o}s--R\'enyi random graphs~\cite{CCKLRb,CS21}).

\begin{theorem}\label{thm:matching-vs-rank}
Fix a constant $c\ne e$.
\begin{enumerate}
\item[(A)] Let $G\sim\mathbb{G}(n,c/n)$.
\begin{enumerate}
    \item[(1)] $|\rank A(G)-2\nu(G)|$ is bounded in probability.
    \item[(2)] $|\rank A(G)-\sigma(G)|$ is bounded in probability.
\end{enumerate}
\item[(B)] For $G\sim\mathbb{G}(n,n,c/n)$, $|\rank B(G)-\nu(G)|$ is bounded in
probability.
\end{enumerate}
\end{theorem}
\begin{remark*}
Given \cref{thm:matching-vs-rank}, one may wonder whether $\nu(G)$ (in the setting of (B), and $2\nu(G),\sigma(G)$ in the setting of (A)) in fact \emph{perfectly} describe the rank. For example, could it be true that in the setting of (B) we have $\rank B(G)=\nu(G)$ whp? As will become clear when we discuss our main theorem, this is too much to hope for. We believe that in the setting of (B), the asymptotic distribution of $\nu(G)-\rank B(G)$ is Poisson (with a certain explicit parameter), and in the setting of (A), both $2\nu(G)-\rank A(G)$ and $\sigma(G)-\rank A(G)$ have somewhat more complicated ``Poisson-like'' distributions. However, rigorous proofs of these facts would require adaptations of certain highly nontrivial graph-theoretic results (to characterise $\nu(G)$ and $\sigma(G)$). We believe that these adaptations are possible, but pursuing this direction would be outside the scope of the present paper. See \cref{sec:distributions} for details.
\end{remark*}

\subsection{The asymptotic distribution of the rank}\label{sub:CLT}
Aronson, Frieze, and Pittel~\cite{AFP98} conjectured that for a constant $c$, the matching number $\nu(G)$ of a random graph $G\sim\mathbb{G}(n,c/n)$ satisfies a central limit theorem. This was proved for $c<1$ by Pittel~\cite{Pit90}, and for $c>e$ by Krea\v{c}i\'c~\cite[Theorem~19]{Kre17}. Since \cref{thm:matching-vs-rank}(A1) tells us that $\rank A(G)$ is extremely well approximated by $2\nu(G)$, it is easy to deduce a corresponding central limit theorem for the rank.

\begin{corollary}\label{cor:CLT}
Let $G\sim \mb G(n,c/n)$ for a constant $c<1$ or $c>e$, let $G\sim \mb G(n,c/n)$, and let $X$ be the rank of the adjacency matrix of $G$. Then $(X-\mb EX)/\sqrt{\on{Var} X}\overset{d}{\to}\mathcal{N}(0,1)$.
\end{corollary}
Actually, in an upcoming paper together with Goldschmidt and Krea\v{c}i\'c~\cite{matching-CLT}, we are able to close the gap between $1$ and $e$ in \cref{cor:CLT}. Specifically, the regime $c\le e$ is rather different in nature than the regime $c>e$, and when $G\sim \mb G(n,c/n)$ for $c\le e$, we are able to give a unified proof that the rank and matching number of $G$ both satisfy a central limit theorem (without going through \cref{thm:matching-vs-rank}(A1)).

\begin{remark*}
\cite{Pit90} and \cite{Kre17} provide explicit formulas for the asymptotic values of $\mb EX$ and $\on{Var} X$, though these are a bit too complicated to describe here. It is worth remarking that the asymptotic formula for $\on{Var} X$ is the single place in this paper where there is a material difference between the ``binomial'' model of Erd\H{o}s--R\'enyi random graphs (where each edge is present with probability $p$ independently) and the ``uniform'' model of Erd\H{o}s--R\'enyi random graphs (where we choose a random subset of exactly $m$ edges, for say $m=\lfloor p\binom n2\rfloor$). Indeed, the variance of the matching number (and therefore the variance of the rank) differs by a constant factor between these two settings; see \cite{Pit90,Kre17} for details. For all the other results in the paper (which are all stated for the binomial model), one can make trivial changes to the proofs to obtain exactly the same result in the uniform model.
\end{remark*}
\begin{remark*}
We believe that a central limit theorem does \emph{not} hold for the rank of $\mb G(n,n,c/n)$; see \cref{sub:further}.
\end{remark*}

\subsection{Exactly characterising the rank}
In this subsection we finally state our main theorem, giving an exact combinatorial characterisation of the rank of a sparse random matrix. First, we need to introduce the \emph{Karp--Sipser leaf removal algorithm}, which was introduced in 1981 by Karp and Sipser~\cite{KS81} as a tool to study matchings in random graphs (in a paper which kickstarted the \emph{differential equations method} for random graph processes; see \cite{Wor99DE}), but is now also of great importance in statistical physics, theoretical computer science, and random matrix theory (see for example \cite{BG01,BG01b,MRZ03,BLS11,CCKLRb}). 

\begin{definition}[Karp--Sipser leaf removal] \label{def:ks}
Starting from a graph $G$, choose an arbitrary degree-1 vertex and delete it together with its neighbour. Repeat this ``leaf-deletion'' until no further degree-1 vertices remain. Let $i(G)$ be the number of isolated
vertices in the resulting graph. If $G$ is bipartite, let $i_{1}(G)$
and $i_{2}(G)$ be the number of isolated vertices on the two sides of the bipartition
$V_{1}\cup V_{2}$. Let $\core_{\mathrm{KS}}(G)$ be the graph of
remaining non-isolated vertices (the \emph{Karp--Sipser core}). One can check that $i(G)$
and $\core_{\mathrm{KS}}(G)$ (and $i_{1}(G),i_{2}(G)$, if $G$ is
bipartite) do not depend on the order that the leaf-deletions are performed (see for example the appendix of \cite{BG01b}).
\end{definition}

It is easy to check (see \cref{lem:rank-decrement}) that a single step of leaf-removal decreases $\rank A(G)$ by exactly 2, and if $G$ is bipartite, decreases $\rank B(G)$ by exactly 1. It is then easy to deduce (see \cref{cor:KS-bounds}) that $\rank A(G)\le n-i(G)$ for any $n$-vertex graph $G$ (i.e., $\corank A(G)\ge i(G)$), and $\rank B(G)\le n-\max(i_1(G),i_2(G))$ for any $(n+n)$-vertex bipartite graph $G$ (i.e., $\corank B(G)\ge\max(i_1(G),i_2(G))$). We will refer to these two bounds as the \emph{Karp--Sipser bounds} for the rank of $A(G)$ and $B(G)$, respectively. We remark that there is a one-sided version of the Karp--Sipser bound for $B(G)$ (where leaves are only removed from one of the two sides of our bipartite graph), sometimes called the \emph{2-core bound} in the computer science and statistical physics literature~\cite{CEGHR20,AS08,DM08} (here ``2-core'' refers to a certain hypergraph notion of a 2-core, not to be confused with the notion in \cref{thm:2-core}).

The Karp--Sipser process takes care of ``tree-like'' local dependencies. In random graphs $\mb G(n,p)$ or $\mb G(n,n,p)$ with $np\to \infty$, these are whp the only types of dependencies that exist (see \cite{DGM22,CV10}); that is, the Karp--Sipser core is nonsingular, so the Karp--Sipser bound is sharp. 
However, in the case $p=O(1/n)$, there may be ``cycle-like'' local dependencies in the Karp--Sipser core, such as pairs of degree-2 vertices with the same neighbourhood. We capture dependencies of this type in the following definition, depicted in \cref{fig:cycles}.

\begin{figure}
\begin{minipage}{0.4\textwidth}
\centering
\begin{tikzpicture}
\def \n {8}
\def \radius {2 cm}
\def \margin {12.5} 
\foreach \s in {1,...,\n}
{
    \tikzmath{\curr =  sin(90 * \s);}
    \node[draw, circle, minimum size = 25] at ({360/\n * (\s - 1)}:\radius) {\pgfmathprintnumber{\curr}};
  \draw[-, >=latex] ({360/\n * (\s - 1)+\margin}:\radius) 
    arc ({360/\n * (\s - 1)+\margin}:{360/\n * (\s)-\margin}:\radius);
}
\end{tikzpicture}
\caption*{Isolated special cycle}

\end{minipage}
\hspace{.7cm}
\begin{minipage}{0.35\textwidth}
\begin{tikzpicture}
\def \n {4}
\def \radius {2 cm}
\def \margin {12.5} 
\foreach \s in {1,...,\n}
{
    \tikzmath{\curr =  sin(90 * (\s + 1) );}
    \node[draw, circle, minimum size = 25] at ({360/\n * (\s - 1)+45}:\radius) {\pgfmathprintnumber{\curr}};
  \draw[-, >=latex] ({360/\n * (\s - 1) + 45+\margin}:\radius) 
    arc ({360/\n * (\s - 1)+45+\margin}:{360/\n * (\s)+45-\margin}:\radius);
}
\phantom{\foreach \s in {1,...,\n}
{
    \tikzmath{\curr =  sin(90 * \s);}
    \node[draw, circle, minimum size = 25] at ({360/\n * (\s - 1)}:\radius) {\pgfmathprintnumber{\curr}};
  \draw[-, >=latex] ({360/\n * (\s - 1)+\margin}:\radius) 
    arc ({360/\n * (\s - 1)+\margin}:{360/\n * (\s)-\margin}:\radius);
}}
\node[draw, circle, minimum size = 25, dashed] at (3.8,1.4){0};
\draw[dashed] (1.87,1.4) -- (3.35,1.4);
\node[draw, circle, minimum size = 25, dashed] at (3.8,0){0};
\draw[dashed] (1.84,1.267) -- (3.39,.2);
\end{tikzpicture}
\caption*{\hspace{-2em}Non-isolated special cycle}
\end{minipage}
\caption{On the left is an isolated special cycle with 8 vertices. On the right is a non-isolated special cycle with 4 vertices. In both pictures, we depict the entries of a kernel vector of the corresponding adjacency matrix. Note that for the isolated cycle, one can obtain an additional linearly independent kernel vector by shifting each entry one edge clockwise around the cycle. }
    \label{fig:cycles}
\end{figure}

\begin{definition}[Special cycles]\label{def:special-cycles}
Say an induced cycle in a graph $G$ is \emph{special} if
its length is divisible by 4, and if every second vertex has degree
2 in $G$. In particular, an \emph{isolated cycle} is a cycle in which every vertex has degree exactly 2 (i.e., it is its own connected component), so isolated cycles with length divisible by 4 are special ``in two different ways''. Let $s(G)$ be the number of special cycles in $G$, where
we count each isolated cycle twice. 

If $G$ is bipartite, say an induced cycle in $G$ is 1-special
(respectively, 2-special) if its length is divisible by 4, and every
vertex in $V_{1}$ (respectively, every vertex in $V_{2}$) has degree
2. Let $s_{1}(G)$ and $s_{2}(G)$ be the numbers of 1-special and
2-special cycles in $G$, respectively.
\end{definition}

To see that a special cycle indeed constitutes a dependency, note that we can construct a kernel vector by ``alternating $\pm 1$ entries around a special cycle'', as follows.
\begin{fact}\label{fact:special-kernel}
Let $G$ be a graph on the vertex set $V$. Let $u_1,\ldots,u_{4k}$ (in order) be the vertices of a special cycle, where $u_2,u_4,\ldots,u_{4k}$ have degree 2. Define $\mbf v\in \{-1,0,1\}^V$ by setting the entries indexed by $u_2,u_6,\ldots,u_{4k-2}$ to $1$ and setting the entries indexed by $u_4,u_8,\ldots,u_{4k}$ to $-1$, and setting all other entries to zero. (That is to say, we go around the cycle, alternating $1$ and $-1$ on our degree-2 vertices). Then, $\mbf v$ is a kernel vector of $A(G)$.

When $G$ is a bipartite graph with bipartition $V_1\cup V_2$, an analogous construction gives a left kernel vector of $B(G)$ if $u_2,u_4,\ldots,u_{4k}\in V_1$, and a right kernel vector of $B(G)$ if $u_2,u_4,\ldots,u_{4k}\in V_2$.
\end{fact}

Our main theorem says that for $c\ne e$, the rank of a sparse random graph $\mb G(n,c/n)$ or $\mb G(n,n,c/n)$ can be described in terms of the Karp--Sipser bound and the special cycles within the Karp--Sipser core.

\begin{theorem}\label{thm:rank-characterisation}
Fix a constant $c\ne e$.
\begin{enumerate}
\item[(A)] Let $G\sim\mathbb{G}(n,c/n)$. Then whp $\corank A(G)=i(G)+s(\core_{\mathrm{KS}}(G))$.
\item[(B)] Let $G\sim\mathbb{G}(n,n,c/n)$. Then whp
\[\corank B(G)=\max\big(i_{1}(G)+s_{1}(\core_{\mathrm{KS}}(G)),\;i_{2}(G)+s_{2}(\core_{\mathrm{KS}}(G))\big).\]
\end{enumerate}
\end{theorem}
\begin{remark*}
If we fix a vertex and consider an exploration process to find a special cycle containing that vertex, it is not hard to show that this process is \emph{subcritical} and explores only $O(1)$ vertices in expectation. Via a standard concentration inequality, it follows that whp we can find all the special cycles in the Karp--Sipser core in time $O(n)$. The Karp--Sipser leaf removal process also completes in time $O(n)$, so \cref{thm:rank-characterisation} actually gives a linear-time algorithm for computing the rank of a sparse random graph\footnote{To be precise, we obtain a linear-time algorithm to compute a quantity that agrees with the rank whp.}.
\end{remark*}

We can also describe the asymptotic distribution of the ``defect'' in the Karp--Sipser bound; for this we need to define some Poisson parameters.

\begin{definition}[Poisson parameters]\label{def:poisson-parameters}
For $0\le c<e$, let $\eta=\eta(c)\in[0,1]$ be the unique solution to
the equation $c=\eta e^{\eta}$. For $c\ge0$, define $\Phi_c\colon[0,1]\to[0,1]$
by $\alpha\mapsto1-\exp\left(-c\exp(-c(1-\alpha))\right)$. If $c>e$
then $\Phi_c$ has multiple fixed points (see for example \cite{CCKLRb});
let $\alpha_\ast=\alpha_\ast(c)$ and $\alpha^\ast=\alpha^\ast(c)$ be
the smallest and largest of these fixed points, respectively, and
let $\lamKS(c)=c(\alpha^\ast-\alpha_\ast)$. For $\lambda\ge 0$ let
\[\gamma(\lambda)=-\frac{1}{4}\log\left(1-\left(\frac{\lambda}{e^{\lambda/2}-e^{-\lambda/2}}\right)^{4}\right),\quad \gamma^\dagger(\lambda)=-\frac{1}{8}\log\left(1-\left(\frac{\lambda}{e^{\lambda}-1}\right)^{4}\right)\]
Then, for $c\in [0,e)\cup (e,\infty)$ let
\[
\gamma_{\mr B}=\gamma_{\mr B}(c)=\begin{cases}
-\frac{1}{4}\log\left(1-\eta^{4}\right) & \text{if }c<e,\\
\gamma(\lamKS(c)) & \text{if }c>e,\end{cases}
\]
\[
\gamma_{\mr A}^{\dagger}=\gamma_{\mr A}^{\dagger}(c)=\begin{cases}
\gamma_\mr{B}/2 & \text{if }c<e,\\
\gamma^\dagger(\lamKS(c)) & \text{if }c>e,\end{cases},\qquad\qquad 
\gamma_{\mr A}=\gamma_{\mr A}(c)=\begin{cases}
0 & \text{if }c<e,\\
\gamma_\mr{B}-2\gamma_\mr{A}^{\dagger} & \text{if }c>e.\end{cases}
\]
\end{definition}

\begin{theorem}\label{thm:rank-characterisation-distribution}
Fix a constant $c\ne e$.
\begin{enumerate}
\item[(A)] Let $G\sim\mathbb{G}(n,c/n)$. Then
\[\corank A(G)-i(G)\overset d\to Y+2Y^{\dagger},\]
where $Y,Y^{\dagger}$ are independent Poisson random variables with means $\gamma_\mr{A}(c)$ and $\gamma_\mr{A}^{\dagger}(c)$, respectively.
\item[(B)] Let $G\sim\mathbb{G}(n,n,c/n)$. Then
\[\corank B(G)-\max(i_1(G),i_2(G))\overset d\to Y,\]
where $Y$ is Poisson with mean $\gamma_\mr{B}(c)$.
\end{enumerate}
\end{theorem}
\begin{remark*}
As written, our proof is not strong enough to estimate the \emph{expected} defect in the Karp--Sipser bound, but it does seem to be possible to prove such estimates by taking more care with quantitative aspects (which we do not pursue in this paper, in the interests of keeping our proofs as simple as possible). Specifically, one expects $\lim_{n\to \infty }\mb E[\corank A(G)-i(G)]= \gamma_{\mr A}(c)+2\gamma_{\mr A}^{\dagger}(c)$ in the setting of (A) and $\lim_{n\to \infty }\mb E[\corank B(G)-\max(i_1(G),i_2(G))]= \gamma_{\mr B}(c)$ in the setting of (B).
\end{remark*}

The reader is overdue an explanation for the significance of the ``critical point'' $c=e$. It turns out that this point amounts to a ``phase transition'' for the Karp--Sipser process. Namely (in the settings of both $\mb G(n,n,c/n)$ and $\mb G(n,c/n)$), for $c<e$, the Karp--Sipser core whp consists of a tiny number of vertex-disjoint cycles, whereas for $c>e$ the Karp--Sipser core whp has a single giant component with nontrivial structure, in addition to a tiny number of vertex-disjoint cycles. This situation parallels the phase transition (at $c=1$) of the components of a random graph, and suggests that when $c=e$, the Karp--Sipser core may have similar structure to the 2-core of a critical random graph $\mb G(n,1/n)$ or $\mb G(n,n,1/n)$. Unfortunately, it is very challenging to study the Karp--Sipser core in this critical regime, and essentially nothing has been rigorously proved (though see the very recent work of Budzinski, Contat, and Curien~\cite{BCC22} on a simpler model of random graphs, and the numerical simulations of Bauer and Golinelli~\cite{BG01b}).

Although our understanding of the critical Karp--Sipser process is not sufficient to prove or disprove \cref{thm:rank-characterisation} at the critical point $c=e$, we are at least able to show (as a consequence of \cref{thm:rank-characterisation-distribution}) that the defect in the Karp--Sipser bound is unbounded in probability for $p$ near $e/n$, strongly suggesting that the situation is rather different at the critical point.
\begin{theorem}\label{thm:critical}
There is a sequence $(p_{n})_{n=1}^{\infty}$ with $np_{n}\to e$,
such that:
\begin{enumerate}
\item[(A)] For $G\sim\mathbb{G}(n,p_n)$, we have $\corank A(G)-i(G) \overset{p}{\to}\infty$.
\item[(B)] For $G\sim\mathbb{G}(n,n,p_n)$, we have
$\corank B(G)-\max(i_1(G),i_2(G)) \overset{p}{\to}\infty$.
\end{enumerate}
\end{theorem}
\begin{remark*}
With some more work, it seems that it would be possible to prove that for $p_n=e/n$ (or any $(p_{n})_{n=1}^{\infty}$ for which $np_n$ converges sufficiently rapidly to $e$), in the settings of both (A) and (B), the defect in the Karp--Sipser bound is whp at least of order $\log n$. See \cref{rem:big-defect}.
\end{remark*}

We discuss the critical regime $c=e$ further in \cref{sub:further}.

\subsection{Degree-constrained random graphs}
Both the Karp--Sipser core and the 2-core have minimum degree at least 2. In fact, more is true: for each of these types of cores, if we condition on the vertex set of the core, and its number of edges, then it is a \emph{uniformly random} graph on the conditioned vertex set, with the conditioned number of edges, subject to the constraint of having minimum degree at least 2 (as we will see in \cref{sec:cores}).
\begin{definition}\label{def:min-degree-model}
For a set $V$ and a positive integer $m\ge|V|$, let $\mathcal{K}(V,m,2)$
be the uniform distribution on graphs with vertex set $V$, exactly
$m$ edges, and minimum degree at least $2$. For a pair of sets $V_{1},V_{2}$
and a positive integer $m\ge2\max(|V_{1}|,|V_{2}|)$, let $\mathcal{K}(V_{1},V_{2},m,2)$
be the uniform distribution on bipartite graphs with vertex set $V_{1}\cup V_{2}$,
exactly $m$ edges, and minimum degree at least $2$. We write $\mathcal{K}(n,m,2)=\mathcal{K}(\{1,\ldots,n\},m,2)$
and $\mathcal{K}(n_{1},n_{2},m,2)=\mathcal{K}(\{1,\ldots,n_{1}\},\{n_{1}+1,\ldots,n_{1}+n_{2}\},m,2)$.
\end{definition}

The main engine driving the proofs of \cref{thm:rank-characterisation,thm:2-core} is the following theorem on the rank of $\mathcal{K}(n,m,2)$ and $\mathcal{K}(n_{1},n_{2},m,2)$, which may be of independent interest.

\begin{theorem}\label{thm:main-RMT}
Fix $\varepsilon>0$. Recall the definitions of $\gamma(\lambda),\gamma^\dagger(\lambda)$ from \cref{def:poisson-parameters}.
\begin{enumerate}
\item[(A)] Suppose $(1+\varepsilon)n\le m\le n/\varepsilon$ and let $G\sim\mathcal{K}(n,m,2)$.
\begin{enumerate}
    \item[(1)] Whp $\rank A(G)=n-s(G)$.
    \item[(2)] Suppose $2m/n$ converges to a constant $\alpha>2$. Choose $\lambda>0$ such that if $Z\sim \on{Poisson}(\lambda)$, then $\alpha=\mb E[Z|Z\ge 2]$. Then $s(G)\overset d \to Y+2Y^{\dagger}$, where $Y,Y^{\dagger}$ are independent Poisson with means $\gamma(\lambda)-2\gamma^\dagger(\lambda)$ and $\gamma^\dagger(\lambda)$ respectively. (Here $Y^{\dagger}$ captures the isolated cycles with length divisible by 4, and $Y$ captures the other special cycles.)
\end{enumerate}
\item[(B)] Suppose $n_{1}-n_{2}\to\infty$, $n_{1}/n_{2}\to1$ and $(1+\varepsilon)(n_1+n_2)\le m\le (n_1+n_2)/\varepsilon$,
and let $G\sim\mathcal{K}(n_{1},n_{2},m,2)$.
\begin{enumerate}
    \item[(1)] Whp $\rank B(G)=n_{2}-s_{2}(G)$.
    \item[(2)] Suppose $2m/(n_1+n_2)$ converges to a constant $\alpha>2$. Choose $\lambda>0$ such that if $Z\sim \on{Poisson}(\lambda)$, then $\alpha=\mb E[Z|Z\ge 2]$. Then $s_1(G)\overset d \to Y$ and $s_2(G)\overset d \to Y$, where $Y$ is Poisson with mean $\gamma(\lambda)$.
\end{enumerate}
\end{enumerate}
\end{theorem}

\subsection{Further directions}\label{sub:further}
The theory of random Bernoulli matrices  (i.e., adjacency matrices of $\mb G(n,p)$, biadjacency matrices of $\mb G(n,n,p)$, and closely related random matrix models) is very rich, and there are a large number of conjectures and open problems. See for example the surveys of Guionnet~\cite{Gui} and Vu~\cite{Vu08,Vu20}. Below we mention some directions which are especially closely related to the present paper.

\subsubsection{The critical regime}
Perhaps the most obvious direction for further research is to improve our understanding in the critical case $c=e$. Unfortunately, our understanding of the critical Karp--Sipser core is very poor; even its typical number of vertices is unknown (though conjectures motivated by numerical simulations have been made by Bauer and Golinelli~\cite{BG01b}, and a rigorous result was recently obtained by Budzinski, Contat, and Curien~\cite{BCC22} for a simpler model of random graphs). Also, we suspect that in the critical case, the Karp--Sipser core typically has nearly as many vertices as edges (i.e., the average degree is very close to 2), so \cref{thm:main-RMT} does not apply, motivating the following question.

\begin{question}\label{ques:uniform-corank}
What can we say about the typical rank of the adjacency matrix of $G\sim \mc K(n,n+t,2)$, for $t=o(n)$? What can we say about the typical rank of the biadjacency matrix of $G\sim \mc K(n_1,n_2,2n_1+t,2)$, for $n_1\ge n_2$ with $n_1=(1+o(1))n_2$ and $t=o(n_1)$?
\end{question}

We remark that if $t=O(1)$ then the special cycles may intersect each other, and the combinatorial description of the rank in \cref{thm:main-RMT}(1) no longer holds whp. In this case there may simply not exist a description of the rank that holds whp and which can be reasonably described as ``combinatorial''.

\subsubsection{The asymptotic distribution of the rank}
In \cref{cor:CLT} we proved a central limit theorem for the rank of $\mb G(n,c/n)$ for $c<1$ or $c>e$, complemented by upcoming work with Goldschmidt and Krea\v{c}i\'c~\cite{matching-CLT} in which we handle the regime $c\le e$. We wonder whether it may also be possible to prove a \emph{local} central limit theorem for the rank. Indeed, it seems plausible that (at least in the regime $c>e$) the techniques in \cite{CCKS19} might be helpful to prove a local central limit theorem for the Karp--Sipser bound $n-i(G)$; we suspect that it would then be possible to adapt the methods in this paper to deduce a local central limit theorem for the rank of $\mb G(n,c/n)$.

However, we do \emph{not} believe that even a coarse central limit theorem holds for the rank of $G\sim \mb G(n,n,c/n)$. Recall that $\corank B(G)$ is approximately $\max(i_1(G),i_2(G))$; we believe that the asymptotic joint distribution of $i_1(G)$ and $i_2(G)$ is a nontrivial bivariate Gaussian, in which case the limiting distribution of $\rank B(G)$ would be expressible in terms of the maximum of two Gaussians.

\subsubsection{Other sparse random matrix distributions}
One may wish to study more general types of sparse random matrices than $\mb G(n,p)$ and $\mb G(n,n,p)$. For example, we could fix a distribution $\mathcal L$ for the nonzero entries (instead of having every nonzero entry be exactly 1). The methods in this paper are quite robust, and we believe it should be possible to handle random matrices of this type, though the notion of ``special cycle'' would have to be adapted accordingly (the defect in the Karp--Sipser bound would still be controlled by short cycles, but it would be more complicated to describe exactly which short cycles are relevant).

However, the methods in this paper do have some limitations: they are only suitable when an approximate rank result is available (e.g., recall that Bordenave, Lelarge and Salez~\cite{BLS11} found a formula for the rank of $\mb G(n,p)$ up to $o(n)$ additive error). Our methods also do not apply to graphs with bounded degree (e.g.~random regular graphs, which were recently shown to have full rank whp by Huang~\cite{Hua18} and M\'{e}sz\'{a}ros~\cite{Mes20}, in breakthrough works using completely different methods to the present paper).

\subsubsection{Rank over other fields}
One may wish to study rank over fields other than $\mb R$ (e.g., rank over $\mb F_2$).  We do not believe that an exact combinatorial characterisation of the rank is actually possible over finite fields, because in general dependencies need not be ``local'' (even a dense random matrix has a nontrivial probability of being singular over $\mb F_2$). However, we do believe that there are typically very few ``non-local dependencies'', and in particular it should still be true that the defect in the Karp--Sipser bound (for both $\mb G(n,n,c/n)$ and $\mb G(n,c/n)$, with $c\ne e$) is bounded in probability.

To prove this would require a number of modifications to our proof (for example, one should incorporate some of the techniques in \cite{FJSS21}, which build on ideas introduced in \cite{MapNote}). The most significant obstacle is that our proof uses spectral convergence machinery due to Bordenave, Lelarge and Salez~\cite{BLS11} which is fundamentally only suitable for real rank. In the bipartite setting (i.e., for $\mb G(n,n,p)$) one can substitute machinery due to Coja-Oghlan, Erg\"{u}r, Gao, Hetterich, and Rolvien~\cite{CEGHR20}, which provides asymptotic formulas for the rank of a broad class of random matrices over arbitrary fields. In the non-bipartite setting, such machinery is not yet available in appropriate generality, but an exciting first step in this direction was very recently made by van der Hofstad, M\"uller, and  Zhu~\cite{HMZ}.

\subsection{Notation}\label{sub:notation}
We use the notation $\delta\ll \varepsilon$ to indicate that $\delta$ is sufficiently small in terms of $\varepsilon$ (so $1/M\ll \varepsilon$ means that $M$ is sufficiently \emph{large} in terms of $\varepsilon$, and $\varepsilon\ll 1$ means that $\varepsilon$ is sufficiently small in absolute terms).

We use standard asymptotic notation throughout, as follows. For functions $f=f(n)$ and $g=g(n)$, we write $f=O(g)$ or $f \lesssim g$ to mean that there is a constant $C$ such that $|f(n)|\le C|g(n)|$ for sufficiently large $n$. Similarly, we write $f=\Omega(g)$ or $f \gtrsim g$ to mean that there is a constant $c>0$ such that $f(n)\ge c|g(n)|$ for sufficiently large $n$. Finally, we write $f\asymp g$ or $f=\Theta(g)$ to mean that $f\lesssim g$ and $g\lesssim f$, and we write $f=o(g)$ or $g=\omega(f)$ to mean that $f(n)/g(n)\to0$ as $n\to\infty$. Subscripts on asymptotic notation indicate quantities that should be treated as constants.

We also use standard graph-theoretic notation. In particular, $V(G)$ and $E(G)$ denote the vertex set of a graph $G$, and $v(G)=|E(G)|$ and $e(G)=|E(G)|$ denote the numbers of vertices and edges. We write $G[U]$ to denote the subgraph \emph{induced} by a set of vertices $U\subseteq V(G)$. For a vertex $v\in V(G)$, its neighborhood (i.e., the set of vertices adjacent to $v$) is denoted by $N_G(v)$, and its degree is denoted $\deg_G(v)=|N_G(v)|$ (the subscript $G$ will be omitted when it is clear from context). 
We also write $N_U(v)=U\cap N(v)$ and $\deg_U(v)=|N_U(v)|$ to denote the degree of $v$ into a vertex set $U$.

Somewhat less standardly, in this paper all bipartite graphs will have parts indexed by 1 and 2. We write $V_1(G),V_2(G)$ for the two parts of a bipartite graph $G$, and write $v_1(G),v_2(G)$ for the number of vertices in each part. For a set of vectors $S$, we write $\dim S$ for the dimension of the span of $S$, and we write $\supp(S)$ for the union of supports of vectors in $S$.

We define the \emph{double factorial} $n!!$ to be the product of all integers from $1$ to $n$ which have the same parity as $n$, and we define the \emph{falling factorial} $(n)_k=n!/(n-k)!$. For a real number $x$, the floor and ceiling functions are denoted $\lfloor x\rfloor=\max(i\in \mb Z:i\le x)$ and $\lceil x\rceil =\min(i\in\mb Z:i\ge x)$. We will however sometimes omit floor and ceiling symbols and assume large numbers are integers, wherever divisibility considerations are not important. 
All logarithms in this paper without an explicit base are to base $e$, and the set of natural numbers $\mb N$ includes zero.

\subsection{Acknowledgments} We would like to thank Noga Alon for suggesting that our main result gives a linear-time algorithm for computing the rank.

\section{Overview of the paper and proofs}\label{sec:outline}
Most of the paper (all of \cref{sec:random-walk,sec:boost-corank,sec:corank-init,sec:RMT-overview,sec:stalks-computation,sec:bipartite-sketch}) is devoted to \cref{thm:main-RMT}(1), characterising the corank of a degree-constrained random graph. Before discussing its proof, we briefly outline the reductions for the other theorems:
\begin{itemize}
    \item For the asymptotic distribution of the corank (\cref{thm:main-RMT}(2)): we simply need to understand the asymptotic distribution of the number of special cycles in a degree-constrained random graph ($\mathcal{K}(n,m,2)$ or $\mathcal{K}(n_{1},n_{2},m,2)$). This can be done with standard techniques (namely, we perform a method-of-moments calculation in the so-called \emph{configuration model} for random graphs with a given degree sequence, after using standard Poisson approximation techniques to study the typical degree sequence of $\mathcal{K}(n,m,2)$ and $\mathcal{K}(n_{1},n_{2},m,2)$). The details appear in \cref{sec:distributions}.
    \item Regarding the 2-core (\cref{thm:2-core}): for $G\sim\mb G(n,c/n)$ with $c>1$, it is easy to estimate the typical number of vertices and edges in the 2-core of $G$ (in particular, there are whp $\Omega(n)$ vertices and the average degree is $2+\Omega(1)$). So, \cref{thm:2-core} follows directly from \cref{thm:main-RMT}. The details appear in \cref{sec:cores}.
    \item Regarding our main rank characterisation theorems (\cref{thm:rank-characterisation,thm:rank-characterisation-distribution}): for $G\sim\mb G(n,c/n)$ with $c>e$, the typical number of vertices and edges in the Karp--Sipser core of $G$ were already studied in the seminal work of Karp and Sipser (again, there are whp $\Omega(n)$ vertices and the average degree is $2+\Omega(1)$). So, in this case the conclusions of \cref{thm:rank-characterisation}(A) and \cref{thm:rank-characterisation-distribution}(A) again follow directly from \cref{thm:main-RMT}. The case $c<e$ is actually much simpler, and does not require \cref{thm:main-RMT}: it was shown by Aronson, Frieze and Pittel~\cite{AFP98} that for $G\sim\mb G(n,c/n)$ with $c<e$, the Karp--Sipser core of $G$ whp consists purely of vertex-disjoint cycles (\cref{lem:subcritical-cycles}), so the conclusions of \cref{thm:rank-characterisation}(A) and \cref{thm:rank-characterisation-distribution}(A) then follow simply by reasoning about the rank of adjacency matrices of cycles. In all cases, the bipartite setting (for \cref{thm:rank-characterisation}(B) and \cref{thm:rank-characterisation-distribution}(B)) can be handled similarly.
    \item \cref{thm:critical} (regarding the critical regime $np_n\to e$) follows from \cref{thm:rank-characterisation-distribution}, and the observation that the Poisson parameters defined in \cref{def:poisson-parameters} blow up as $c\to e$.
    \item For the comparison between rank and matching number (\cref{thm:matching-vs-rank}): It turns out that \cref{thm:matching-vs-rank}(A2) and (B) follow directly from \cref{thm:rank-characterisation}, via certain trivial inequalities concerning $\nu(G)$ and $\sigma(G)$. For \cref{thm:matching-vs-rank}(A1), we combine \cref{thm:rank-characterisation} with an exact description of the matching number of $\mathcal{K}(n,m,2)$ due to Frieze and Pittel~\cite{FP04}.
    \item \cref{cor:CLT} (the central limit theorem for the rank) is an essentially immediate deduction from central limit theorems for the matching number due to Pittel~\cite{Pit90} and Krea\v{c}i\'c~\cite{Kre17}. The details appear in \cref{sec:distributions}.
\end{itemize}
The deductions of \cref{thm:rank-characterisation,thm:critical,thm:rank-characterisation-distribution,thm:matching-vs-rank} all appear in \cref{sec:karp-sipser}, together with various facts about the Karp--Sipser process. We also remark that \cref{sec:preliminaries} contains a few basic preliminary facts that will be used throughout the paper, \cref{sec:cores} contains some basic facts about 2-cores and Karp--Sipser cores, and \cref{sec:core-estimates} contains some general facts about degree-constrained random graphs.

Now we discuss the tools and ideas in the proof of \cref{thm:main-RMT}(1) (restricting our attention to (A1), which is the slightly more difficult of the two settings; (B1) is proved in essentially the same way, but certain minor simplifications are possible, sketched in \cref{sec:bipartite-sketch}).

\subsection{Spectral convergence}
Qualitatively, \cref{thm:main-RMT}(1) says that degree-constrained random graphs are very nearly nonsingular (the only obstructions to singularity are a small number of special cycles). One can obtain a much weaker result in a similar spirit using spectral convergence machinery of Bordenave, Lelarge, and Salez~\cite{BLS11}. Specifically, there is a notion of \emph{local weak convergence} of graphs, introduced independently by Benjamini and Schramm~\cite{BS01} and by Aldous and Steele~\cite{AS04}. In \cite{BLS11}, it is shown that when a sequence of graphs locally weakly converges to a Galton--Watson tree, then the spectrum also converges, and one can estimate the limiting rank via a generating function associated with the Galton--Watson tree. (The fact that spectral information can be deduced from a local limit is not surprising, in light of the fact that the $t$-th moment of the empirical spectral distribution is precisely equal to the number of closed walks of length $t$.)
It can be shown that if $\lim_{n\to \infty}n m_n$ converges to a limit, then the local weak limit of $\mathcal{K}(n,m_n,2)$ is a Galton--Watson tree whose offspring distribution has an explicit (``truncated Poisson'') distribution. With a simple calculation concerning generating functions associated with truncated Poisson distributions, and a compactness argument, one can use the machinery of \cite{BLS11} to show that in the setting of \cref{thm:main-RMT}(A1), we have $\corank(G)=o(n)$ whp.

Of course, the above result is far weaker than the statement of \cref{thm:main-RMT} (we hope to prove that $\corank(G)$ is bounded in probability, not just that $\corank(G)=o(n)$). We will make up the difference using a ``rank-boosting'' strategy, using tools that are traditionally used to study singularity of random matrices (in particular, tools related to the \emph{Littlewood--Offord problem}).

\subsection{The evolving rank, and the Littlewood--Offord problem}
In this subsection we very briefly explain the techniques in the seminal paper of Costello, Tao, and Vu~\cite{CTV06} (building on the original ideas of Koml\'os~\cite{Kom67,Kom68}), who proved that \emph{dense} random graphs are nonsingular. Roughly speaking, their approach was to reveal a random graph (say $\mb G(n,1/2)$) in a vertex-by-vertex fashion, at each step studying how the addition of a new vertex affects the rank. They proved that if, at a given step, the corank is nonzero, then at the next step the corank will typically decrease by one. On the other hand, if the corank is already zero, then at the next step the corank will typically stay at zero. In this way, they could view the evolution of the corank as a random walk that heavily trends towards zero, and show that such random walks almost always end at zero.

In order to implement this strategy, it is necessary to understand how the rank changes when we add a new vertex. For example, if we add a new vertex $v$ to a graph $H$ to obtain a graph $H+v$ (and let $\mbf x$ be a random zero-one vector describing the presence of edges between $v$ and the vertices of $H$), then the determinant of $A(H+v)$ can be expressed as a quadratic polynomial in $\mbf x$ (with coefficients depending on $H$). So, showing that $H+v$ is full-rank is tantamount to showing that a certain quadratic polynomial is nonzero. Correspondingly, an important ingredient in \cite{CTV06} was the fact that certain quadratic polynomials of independent random variables are unlikely to be zero.

The \emph{Littlewood--Offord problem} studies the point probabilities of sums of independent discrete random variables. In particular, the fundamental theorem in this field is the \emph{Erd\H{o}s--Littlewood--Offord theorem}, which was used in Koml\'os' foundational papers~\cite{Kom67,Kom68} on discrete random matrices. To study the evolving rank of a random graph, Costello, Tao, and Vu initiated the study of the \emph{quadratic} Littlewood--Offord problem: specifically, they proved that if an $N$-variable real quadratic polynomial $f$ has $\Omega(N^2)$ nonzero coefficients, and $\mbf x\in\{0,1\}^N$ is a uniform random binary vector, then $\Pr[f(\mbf x)=0]\le N^{-1/8}$. It turns out that in order to fully understand the evolution of the rank one needs Littlewood--Offord-type theorems of both linear and quadratic type: in order to show that the corank typically decreases when it is nonzero, one considers a linear Littlewood--Offord problem, and in order to show that the corank typically stays zero when it is zero, one considers a quadratic Littlewood--Offord problem.

A key reason for the difficulty of studying sparse random matrices is that Littlewood--Offord theorems break down in very sparse settings: if $\mbf x\in\{0,1\}^N$ is a random binary vector in which every entry is $1$ with probability only $c/N$, it is simply not in general true that the event $f(\mbf x)=0$ is unlikely. For example, if $f(\mbf x)=x_1+\cdots+x_N$ or $f(\mbf x)=(x_1+\cdots+x_N)^2$ (in the linear and quadratic cases, respectively), then the asymptotic distribution of $f(\mbf x)$ is $\on{Poisson}(c)$, or the square of a $\on{Poisson}(c)$ distribution, and the point probabilities of $f(\mbf x)$ are of the form $\Omega(1)$. Roughly speaking, the problem is that in this very sparse regime there is ``not enough randomness'' in $\mbf x$.

\subsection{Rank-boosting}\label{sub:rank-boosting}

The key insight to overcome this issue is as follows. In the setting of \cref{thm:main-RMT}, while the \emph{average} degree of $G$ is typically only $O(1)$, whp there are at least a few vertices with much higher degree. Indeed, the maximum of $n$ independent Poisson random variables is typically about $\log n/\log \log n$, and correspondingly it turns out that $G$ typically has at least a few vertices of that degree. More qualitatively, for any $\beta=o(1)$, the $\beta n$ highest-degree vertices all have degree $\omega(1)$.

In \cite{FKSS21}, Ferber and the last three authors leveraged this observation together with the techniques discussed in the last two subsections, to prove that the $k$-core of a random graph (for $k\ge 3$) is nonsingular whp. Specifically, for a random $n$-vertex graph constrained to have minimum degree at least $k$, they designed a procedure to identify $\beta n$ vertices of high degree \emph{without actually revealing the neighbours of these vertices}. They then showed that the graph induced by the remaining $(1-\beta)n$ vertices locally weakly converges to a Galton--Watson tree, and used the machinery in \cite{BLS11} to prove that the corank of this graph is at most say $\beta n/2$ whp. Now, adding back the $\beta n$ high-degree vertices one-by-one, and studying the evolution of the rank, at each step there is quite a lot of randomness, because each of these vertices has high degree and its neighbourhood has not yet been revealed. So, with a random walk argument together with a quadratic Littlewood--Offord theorem, they could show that at the end of this vertex-adding process the corank has decreased from $\beta n/2$ to zero whp.

At a high level, the approach in this paper is to apply the same rank-boosting strategy to $\mathcal{K}(n,m,2)$ to prove \cref{thm:main-RMT}. However, the situation is far more delicate, for reasons we discuss in the following subsections. 

\subsection{The small-support kernel, minimal kernel vectors, and stalks}
The above rank-boosting strategy cannot succeed as written, because it is simply not true that $\mathcal{K}(n,m,2)$ is nonsingular whp (due to the possible existence of special cycles). This is due to an issue we have so far neglected to mention: for any Littlewood--Offord-type approach (in which we study the rank via events of the form $f(\mbf x)=0$), it is necessary to establish ``non-degeneracy'' conditions for $f$. For instance, we need to ensure that $f$ has many nonzero coefficients (to see that something like this is necessary, note that if $f$ were the zero polynomial, we would have $f(\mbf x)=0$ with probability 1, no matter how dense of a random vector $\mbf x$ is).

The polynomials $f$ that we need to consider are defined in terms of the evolving random graph $G$ (as we add vertices one-by-one). It turns out that if, at some point in the process, $f$ has few nonzero coefficients, this essentially corresponds to $A(G)$ having a kernel vector with small \emph{support} (i.e., with few nonzero entries)\footnote{To be precise, recall that we may need to consider either linear or quadratic $f$, depending on the situation. In the linear case, the coefficients of $f$ correspond precisely to a kernel vector, and in the quadratic case there is a correspondence between coefficients of $f$ and ``almost kernel vectors'' of $A(G)$ (i.e., vectors $\mbf v$ such that $A(G)\mbf v$ has only two nonzero entries). So, in much of what follows, we really need to consider almost kernel vectors as well as kernel vectors.}. Therefore, an essential part of the Littlewood--Offord-based proofs mentioned so far~\cite{Kom67,Kom68,CTV06,FKSS21} is to prove that $A(G)$ has no small-support kernel vectors.

Crucially, this can be accomplished by purely \emph{combinatorial} means: for example, if $\mbf v$ is a kernel vector of an adjacency matrix $A(G)$ (such that the nonzero entries of $\mbf v$ correspond to a set of vertices $R$, say), then when a vertex has a neighbour in $R$, it must in fact have at least two neighbours in $R$ (in order for there to be a cancellation yielding zero in the corresponding entry of $A(G)\mbf v$). In the settings of \cite{Kom67,Kom68,CTV06,FKSS21}, one can simply use a crude combinatorial union bound calculation to show that whp there are no small sets $R$ with this property (for example, in \cite{FKSS21}, one can simply use that $R$ and its neighbours would comprise an atypically dense set, which is unlikely to appear in a sparse random graph). Specifically, union bounds of this type can be made to work as long as $R\le \eta n$ for some small constant $\eta$. We remark that when studying the $k$-core (for $k\ge 3$) in \cite{FKSS21}, it was  possible to engineer the high-degree-vertex extraction in such a way that (crude union bounds show that) whp no short kernel vectors ever appear during the entire vertex-adding process.

Unfortunately, in the setting of \cref{thm:main-RMT}, small-support kernel vectors seem to be unavoidable: special cycles may exist in $\mc K(n,m,2)$ itself, and since we are no longer assuming $k\ge 3$ it seems to be impossible to engineer our high-degree vertex extraction to avoid the emergence of many small-support kernel vectors during our vertex-adding process. Instead, we need to perform a very delicate calculation to upper-bound the numbers of various types of small-support kernel vectors (and in particular to show that at the end of the vertex-adding process, whp the \emph{only} short kernel vectors are those corresponding to special cycles, and linear combinations thereof). It turns out that a na\"ive union bound does not suffice here, and we need to consider a notion of \emph{minimal} kernel vectors (essentially, kernel vectors which cannot be broken down into kernel vectors with smaller support). This notion was first considered by DeMichele, the first author, and Moreira in \cite{DGM22}. A large part of the paper (\cref{sec:stalks-computation}) is spent very carefully studying the expected number of combinatorial configurations corresponding to minimal kernel vectors (called \emph{stalks}), with support size at most $\eta n$, in degree-constrained random graphs. (Unlike in \cite{FKSS21}, we cannot merely consider the density of a stalk; we need to very carefully consider its structure.)

\subsection{Boosting the large-support kernel, and a special-purpose pseudoinverse}
Due to the existence of small-support kernel vectors, the evolution of the rank no longer has such a simple description as in \cite{CTV06,FKSS21}. Instead of showing that the corank drifts towards zero, we show that the corank ``drifts towards the dimension of the small-support kernel''. Specifically, we prove that if there is a kernel vector with large support, then the corank decreases whp, and in any case the corank whp does not increase. It turns out that it is still possible to control the probabilities of these events via events of the form $f(\mbf x)=0$ for some linear or quadratic polynomials $f$, but unlike in \cite{CTV06,FKSS21}, we cannot define $f$ in terms of a determinant (because if there is any kernel vector the determinant is always zero). Instead, our polynomial $f$ is defined in terms of a special-purpose ``pseudoinverse'', first (implicitly) considered in \cite{DGM22}. We state and prove a general purpose rank-boosting lemma, summarising a one-step application of a linear and quadratic Littlewood--Offord theorem,  in \cref{sec:boost-corank}.

\subsection{Robust analysis of a random walk}
Summing up, our approach is as follows. After extracting high-degree vertices and showing that the resulting graph has small corank (executed in \cref{sec:corank-init}), we add back the high-degree vertices one-by-one, and consider the evolution of the rank of this random graph process. Letting $\dim K^{(\eta)}_t$ be the dimension of the span of small-support kernel vectors at time $t$ (which is a lower bound for the corank at time $t$), we prove an upper bound on $\mb E\dim K^{(\eta)}_t$ (in terms of $t$) via direct combinatorial means, and prove using Littlewood--Offord theorems that the corank trends towards $\dim K^{(\eta)}_t$. We wish to combine all these ingredients to prove that at the end of the process, whp the corank is exactly equal to $\dim K^{(\eta)}_t$ (which we then show is equal to $s(G)$).

In order to execute this plan, we need a more robust random walk analysis than in \cite{CTV06,FKSS21}. The main issue is that because we no longer have the ``wiggle room'' afforded by the assumption $k\ge 3$, it is much harder to prove bounds that hold whp for all steps $t$ (e.g., our bounds on $\mb E\dim K^{(\eta)}_t$, together with Markov's inequality, do not provide strong enough probabilistic bounds for a union bound over all $t$). Instead, we have estimates that hold for \emph{most} steps $t$, and we need a more robust analysis of random walks that can tolerate a small number of ``bad steps'' (as long as they are not clustered near the end of the process). We present a general lemma along these lines in \cref{sec:random-walk}, which we hope will be useful for other applications. In \cref{sec:RMT-overview} we put everything together, completing the proof of \cref{thm:main-RMT}(A1).

\section{Preliminaries}\label{sec:preliminaries}
In this section we collect some basic facts that will be used throughout the paper. First, to unify the proofs for the bipartite and nonbipartite cases to the greatest extent possible, we observe that for a bipartite graph $G$, the rank of its biadjacency matrix is related to the rank of its adjacency matrix.
\begin{fact}\label{fact:biadjacency-rank}
If $G$ is bipartite, then $\rank A(G)=2\rank B(G)$
\end{fact}
\begin{proof}
This follows immediately from the fact that (given an appropriate ordering of the vertices) $A(G)$ has the block representation
\[\begin{pmatrix}0 & B(G)^\intercal \\
B(G) & 0
\end{pmatrix}.\qedhere\]
\end{proof}

We will also need a Chernoff bound for binomial and hypergeometric distributions (see for example \cite[Theorems~2.1 and~2.10]{JLR00}). Recall that the hypergeometric distribution $\on{Hyp}(N,K,n)$ is the distribution of $|Z\cap U|$, for fixed sets $U\subseteq S$ with $|S|=N$ and $|U|=K$ and a uniformly random size-$n$ subset $Z\subseteq S$.
\begin{lemma}[Chernoff bound]\label{lem:chernoff}
Let $X$ be either:
\begin{itemize}
    \item a sum of independent random variables, each of which take values in $\{0,1\}$, or
    \item hypergeometrically distributed (with any parameters).
\end{itemize}
Then for any $\delta>0$ we have
\[\Pr[X\le (1-\delta)\mb{E}X]\le\exp(-\delta^2\mb{E}X/2),\qquad\Pr[X\ge (1+\delta)\mb{E}X]\le\exp(-\delta^2\mb{E}X/(2+\delta)).\]
\end{lemma}

Finally, we will need a consequence of the Azuma--Hoeffding inequality (see \cite[Theorem~2.25]{JLR00}).
\begin{lemma}\label{lem:azuma}
Let $Z$ be a random variable defined in terms of a sequence of random variables $X_0, \ldots, X_n$, such that modifying any individual $X_k$ changes $Z$ by at most $c_k$. Then 
\[\mb{P}[|Z-\mb{E} Z|\ge t]\le 2\exp\bigg(-\frac{t^2}{2\sum_{k}c_k^2}\bigg).\]
\end{lemma}

\section{Structure of cores}\label{sec:cores}
In this section we collect some standard results on the 2-core and the Karp--Sipser core of a sparse random graph.

\subsection{The 2-core}\label{sub:2-core}
First, the following description of the component structure of the supercritical 2-core follows immediately from, e.g., \cite[Theorem~5.12]{JLR00} and the main result of \cite{DLP14}. 
\begin{lemma}\label{thm:2-core-components}
Fix a constant $c>1$ and let $G\sim\mathbb{G}(n,c/n)$. Then whp the following holds. The giant component of $G$ has a 2-core which is connected, has $\Omega(n)$ vertices, and has average degree $2+\Omega(1)$. Also, all components outside the giant either have empty 2-core or their 2-core is a cycle.
\end{lemma}

Second, the following lemma concerns the edge and vertex statistics of the supercritical 2-core (there are whp $\Omega(n)$ vertices, and the average degree is whp $2+\Omega(1)$). It follows from, e.g., \cite[Lemma~2.16]{FK16}.

\begin{lemma}\label{thm:2-core-statistics}
Fix $c>1$, and let $\lam2=\lam2(c)>0$ be the unique solution to $\lam2/(1-e^{-\lam2})=c$. There is $\beta=\beta(c)>0$ such that the following holds. Let $Z$ be a Poisson random variable with mean $\lam2$, and let $G\sim\mathbb{G}(n,c/n)$. Then 
\[
\frac{v(\core_{2}(G))}{n}\overset{p}{\to}\beta,\quad\frac{2e(\core_{2}(G))}{\beta n}\overset{p}{\to}\mb E[Z|Z\ge2]>2.
\]
\end{lemma}

Note that \cref{thm:2-core} directly follows from \cref{thm:main-RMT}(A) given \cref{thm:2-core-components,thm:2-core-statistics}, as follows.

\begin{proof}[Proof of \cref{thm:2-core}]
For $G\sim \mb G(n,c/n)$, it is well-known (see for example \cite[Lemma~6.1]{FKSS21}) that given the vertex set $W$ of the 2-core and its number of edges $m$, we have $\on{core}_2(G)\sim \mc K(W,m,2)$. So, by \cref{thm:main-RMT}(A) and \cref{thm:2-core-statistics}, after deleting isolated special cycles, the asymptotic distribution of the corank is Poisson with mean $\gamma(\lam2)-2\gamma^\dagger(\lam2)$. Considering the probability that such a Poisson random variable is equal to zero, and recalling the structural description in \cref{thm:2-core-components} (whp the giant component is obtained precisely by deleting isolated cycles), the desired result follows.
\end{proof}

\subsection{The Karp--Sipser core}
For the Karp--Sipser core, we need some results for both the supercritical ($c>e$) and subcritical ($c<e$) cases. First, for the supercritical Karp--Sipser core, we need a counterpart of \cref{thm:2-core-statistics}, and in the bipartite setting, we need the fact that the two sides of the Karp--Sipser core have quite different sizes.

\begin{lemma}\label{thm:WP-outputs}
Fix a constant $c>e$. Let $\lamKS=\lamKS(c)$ be
as in \cref{def:poisson-parameters}, and let $Z$ be a Poisson random variable with mean $\lamKS$.
There is $\beta=\beta(c)>0$ such that the following holds.
\begin{enumerate}
\item[(A)] If $G\sim\mathbb{G}(n,c/n)$ then 
\[
\frac{v(\core_{\mathrm{KS}}(G))}{n}\overset{p}{\to}\beta,\quad\frac{2e(\core_{\mathrm{KS}}(G))}{\beta n}\overset{p}{\to}\mb E[Z|Z\ge2]>2.
\]
\item[(B)] Let $G\sim\mathbb{G}(n,n,c/n)$.
\begin{enumerate}
\item[(1)] For each $i\in\{1,2\}$,
\[
\frac{v_{i}(\core_{\mathrm{KS}}(G))}{n}\overset{p}{\to}\beta,\quad\frac{e(\core_{\mathrm{KS}}(G))}{\beta n}\overset{p}{\to}\mb E[Z|Z\ge2]>2.
\]
\item[(2)] $|v_{1}(\core_{\mathrm{KS}}(G))-v_{2}(\core_{\mathrm{KS}}(G))|\overset{p}{\to}\infty$.
\end{enumerate}
\end{enumerate}
\end{lemma}
\cref{thm:WP-outputs}(B2) appears as \cite[Lemma~7.1]{CCKLRb}. There are multiple ways to prove \cref{thm:WP-outputs}(A)
and \cref{thm:WP-outputs}(B1). One classical way is to use the so-called \emph{differential
equations method} to study the likely trajectories of certain statistics
associated with the Karp--Sipser process. In the setting of (A), this was done by Karp and Sipser~\cite{KS81} in one of the first applications of the differential equations method. Their analysis was later refined by Aronson,
Frieze, and Pittel~\cite{AFP98} (see also the discussion in \cite{Kre17}). A more modern approach (which arguably yields simpler proofs, though often with worse quantitative aspects)
is to express the relevant statistics in terms of fixed points of
a certain \emph{warning propagation} operator. This was done by Coja--Oghlan,
Cooley, Kang, Lee, and Ravelomanana \cite{CCKLRb} in the setting of (B) (specifically, \cref{thm:WP-outputs}(B1) appears as \cite[Proposition~2.6]{CCKLRb}). Both approaches work equally well in the settings of (A) and (B), with very minor alterations to the proofs.

Second, in the subcritical case ($c<e$), we need the fact that the Karp--Sipser core consists only of vertex-disjoint cycles, and the numbers of cycles of each length are asymptotically jointly Poisson distributed.
\begin{lemma}\label{lem:subcritical-cycles}
Fix a constant $c<e$. Let $\eta\in [0,1]$ be the unique solution to $c=\eta e^\eta$.
\begin{enumerate}
\item[(A)] If $G\sim\mathbb{G}(n,c/n)$ then whp $\core_{\mathrm{KS}}(G)$ is a
collection of vertex-disjoint cycles. Let $N_{\ell}$ be the number
of such cycles of length $\ell$, and let $(Z_\ell)_{\ell=3}^\infty$
be a sequence of independent Poisson random variables with $\mb EZ_{\ell}=\eta^{\ell}/(2\ell)$.
Then $v(\core_{\mathrm{KS}}(G))$ is bounded in probability and
\[
(N_{\ell})_{\ell=3}^{\infty}\overset{d}{\to}(Z_{\ell})_{\ell=3}^{\infty}.
\]
\item[(B)] If $G\sim\mathbb{G}(n,n,c/n)$ then whp $\core_{\mathrm{KS}}(G)$ is
a collection of vertex-disjoint cycles. Let $N_{\ell}$ be the number
of such cycles of length $\ell$, and let $(Z_{2k})_{k=2}^\infty$
be a sequence of independent Poisson random variables with $\mb EZ_{2k}=\eta^{2k}/(2k)$.
Then $v(\core_{\mathrm{KS}}(G))$ is bounded in probability and
\[
(N_{2k})_{k=2}^{\infty}\overset{d}{\to}(Z_{2k})_{k=2}^{\infty}.
\]
\end{enumerate}
\end{lemma}

\cref{lem:subcritical-cycles}(A) is implicit in the proof of \cite[Theorem~2]{AFP98}, and \cref{lem:subcritical-cycles}(B) can be proved in essentially the same way (as the bipartite case of \cref{lem:subcritical-cycles} has not explicitly appeared in the literature before, we provide a brief sketch in \cref{sec:AFP}). 

It turns out that the Karp--Sipser core enjoys the same symmetry property as the 2-core: if we condition on the vertex set $W$ of the Karp--Sipser core, and its number of edges $m$, then $\core_{\mathrm{KS}}(G)$ is distributed as $\mathcal{K}(W,m,2)$ (or $\mathcal{K}(W_{1},W_{2},m,2)$, in the
bipartite case, where $W=W_1\cup W_2$).
\begin{lemma}\label{lem:rotate-core}
Consider any $0\le p\le1$.
\begin{enumerate}
\item[(A)] Let $G\sim\mathbb{G}(n,p)$, let $W$ be the vertex set of $\core_{\mathrm{KS}}(G)$
and let $m$ be the number of edges in $\core_{\mathrm{KS}}(G)$.
Then the conditional distribution of $\core_{\mathrm{KS}}(G)$ is
$\mathcal{K}(W,m,2)$.
\item[(B)] Let $G\sim\mathbb{G}(n_{1},n_{2},p)$, let $W_{1}\cup W_{2}$ be the
vertex set of $\core_{\mathrm{KS}}(G)$ and let $m$ be the number
of edges in $\core_{\mathrm{KS}}(G)$. Then the conditional distribution
of $\core_{\mathrm{KS}}(G)$ is $\mathcal{K}(W_{1},W_{2},m,2)$.
\end{enumerate}
\end{lemma}

\begin{proof}
We prove (A); the proof of (B) is similar. Consider any two graphs
$H,H'$ on the vertex set $W$ with $m$ edges and minimum degree
at least 2. For any outcome of $G$ yielding $\core_{\mathrm{KS}}(G)=G[W]=H$,
we can simply replace $G[W]$ with $H'$ to obtain an outcome of $G$
yielding $\core_{\mathrm{KS}}(G)=H'$ (iterated leaf removal yields
$G[W]$ in both cases). This implies that $H$ and $H'$ are equally
likely to occur as $\core_{\mathrm{KS}}(G)$.
\end{proof}

In much the same way that we were able to deduce \cref{thm:2-core} (on the 2-core) from the lemmas in \cref{sub:2-core} together with \cref{thm:main-RMT}, we will be able to deduce \cref{thm:matching-characterisation,thm:matching-vs-rank}
from the lemmas in this subsection together with \cref{thm:main-RMT}. However, the deductions are not quite as immediate, so we save them for the next section.

\section{Karp--Sipser leaf removal}\label{sec:karp-sipser}
In this section we make some basic observations about the Karp--Sipser leaf-removal process, and show how to deduce \cref{thm:rank-characterisation,thm:critical,thm:rank-characterisation-distribution,thm:matching-vs-rank} from these observations together with \cref{thm:main-RMT}.

It is a simple fact (first observed by Karp and Sipser \cite{KS81}) that in
any graph $G$, removing a degree-1 vertex and its neighbour reduces
the matching number by exactly 1. This leaf removal also has a predictable
effect on $\rank A(G)$ and $\sigma(G)$, and on $\rank B(G)$, if $G$ is bipartite. (Recall that $\sigma(G)$ is the size of the largest permutation matrix ``contained'' in $A(G)$, where our notion of matrix containment allows deleting rows and columns, and changing 1-entries to 0-entries.)
\begin{lemma}\label{lem:rank-decrement}
Fix any graph $G$, and delete a leaf $v$ and its neighbour $w$
to obtain a graph $G'$.
\begin{enumerate}
\item[(A)] $\rank A(G')=\rank A(G)-2$ and $\sigma (G')=\sigma(G)-2$.
\item[(B)] If $G$ is bipartite then $\rank B(G')=\rank B(G)-1$.
\end{enumerate}
\end{lemma}

\begin{proof}
For (B), without loss of generality we can assume that $v$ corresponds
to the first row and $w$ corresponds to the first column. Then, observe
that 
\[
B(G)=\begin{pmatrix}1 & \begin{matrix}0 & \cdots & 0\end{matrix}\\
x_{2}\\
\vdots & B(G')\\
x_{n}
\end{pmatrix}
\]
for some $x_{2},\ldots,x_{n}\in\{0,1\}$. Since $B(G)$ comes from adding a zero row to the top of $B(G')$ and then adding a column, clearly $\rank B(G)\le\rank B(G')+1$. Furthermore, for every full-rank submatrix
of $B(G')$, we can add the first row and column of $B(G)$ to obtain
a full-rank submatrix of $B(G)$, so $\rank B(G)\ge\rank\ensuremath{B(G')}+1$. The result follows.

Similarly, for (A), without loss of generality we can assume that
$v$ corresponds to the first row and column, and $w$ corresponds
to the second row and column. Then,
\[
A(G)=\begin{pmatrix}\begin{matrix}0\\
1
\end{matrix} & \begin{matrix}1\\
0
\end{matrix} & \begin{matrix}0 & \cdots & 0\\
y_{3} & \cdots & y_{n}
\end{matrix}\\
0 & y_{3}\\
\vdots & \vdots & A(G')\\
0 & y_{n}
\end{pmatrix}
\]
for some $y_{3},\ldots,y_{n}\in\{0,1\}$. For every full-rank submatrix of $A(G')$,
we can add the first two rows and columns of $A(G)$ to obtain a full-rank
submatrix of $A(G)$, and we similarly deduce $\rank A(G)=\rank A(G')+2$. For $\sigma$, any permutation submatrix contained in $A(G')$ gives rise to a permutation submatrix in $A(G)$ with two more rows and columns, so we deduce $\sigma(G)=\sigma(G')+2$.
\end{proof}

Now we formally state the Karp--Sipser bounds on rank and matching number that were mentioned in the introduction.
\begin{corollary}\label{cor:KS-bounds}
Fix any graph $G$.
\begin{enumerate}
    \item[(A)] $\max(\rank A(G),2\nu(G))\le \sigma(G)\le v(G)-i(G)$.
    \item[(B)] If $G$ bipartite with vertex set $V_1\cup V_2$ then
    \[\rank B(G)\le \nu(G)\le \min(v_1(G)-i_1(G),v_2(G)-i_2(G)).\]
\end{enumerate}
\end{corollary}
\begin{proof}
First we prove (A). Recall from the introduction (\cref{sub:rank-vs-matching-number}) that $\rank A(G)\le \sigma(G)$ and $2\nu(G)\le \sigma(G)$. So, it suffices to prove $\sigma(G)\le n-i(G)$. Let $v^{\mr c}=v(\core_{\mr{KS}}(G))$. The number of leaf-removal steps in the Karp--Sipser process is $(v(G)-v^{\mr c}-i(G))/2$, so by \cref{lem:rank-decrement}(A), we have $\sigma(G)=v(G)-v^{\mr c}-i(G)+\sigma(\core_{\mr{KS}}(G))$. The desired result follows from the fact that $\sigma(\core_{\mr{KS}}(G))\le v^{\mr c}$.

Now we prove (B). Recall from the introduction that $\rank B(G)\le \nu(G)$, so it suffices to prove that $\nu(G)\le \min(v_1(G)-i_1(G),v_2(G)-i_2(G))$. For $i\in \{1,2\}$, let $v_i^\mr{c}=v_i(\core_{\mr{KS}}(G))$. The number of leaf-removal steps is $v_1(G)-v_1^\mr{c}-i_1(G)=v_2(G)-v_2^\mr{c}-i_2(G)$; since each leaf-removal reduces the matching number by exactly 1 we have
\[\nu(G)=v_1(G)-v_1^\mr{c}-i_1(G)+\nu(\core_{\mr{KS}}(G))=v_2(G)-v_2^\mr{c}-i_2(G)+\nu(\core_{\mr{KS}}(G)).\]
The desired result then follows from the fact that $\nu(\core_{\mr{KS}}(G))\le \min(v_1^\mr{c},v_2^\mr{c})$. 
\end{proof}

\subsection{Deductions}
We now show how to deduce \cref{thm:rank-characterisation,thm:critical,thm:matching-vs-rank,thm:rank-characterisation-distribution}.
\begin{proof}[Proof of \cref{thm:rank-characterisation,thm:rank-characterisation-distribution}]
First we prove (A). Let $v=v(\core_{\mr{KS}}(G))$ and $m=e(\core_{\mr{KS}}(G))$. The number of leaf-removal steps is $(n-v-i(G))/2$, so
\[\rank A(G)=n-v-i(G)+\rank A(\core_{\mr{KS}}(G))\]
by \cref{lem:rank-decrement}(A). For \cref{thm:rank-characterisation} we need to prove that $\rank A(\core_{\mr{KS}}(G))=v-s(\core_{\mr{KS}}(G))$ whp.
\begin{itemize}
\item If $c<e$ then by \cref{lem:subcritical-cycles}(A), whp $\core_{\mr{KS}}(G)$ is a vertex-disjoint union of cycles (and the number of cycles of length $\ell$ is asymptotically Poisson, with parameter $\eta^\ell/(2\ell)$, where $\eta$ is the unique solution to
the equation $c=\eta e^{\eta}$). It is easy to compute (see for example \cite[Example~7.8]{Bog18}) that for a length-$\ell$ cycle $C_\ell$ we have
\[\rank A(C_\ell)=\begin{cases}
  \ell-2&\text{if }\ell\text{ is divisible by 4},\\
  \ell&\text{otherwise.}
\end{cases}\]
So, whp $\rank A(\core_{\mr{KS}}(G))=v-s(\core_{\mr{KS}}(G))$, proving \cref{thm:rank-characterisation}(A). For \cref{thm:rank-characterisation-distribution}(A), note that the defect in the Karp--Sipser bound is exactly twice the number of 4-divisible cycles $\sum_{k=1}^\infty N_{4k}$ (with notation as in \cref{lem:subcritical-cycles}). Recall that any sum of independent Poisson random variables is itself Poisson, and that $v$ is bounded whp. Hence $\sum_{k=1}^\infty N_{4k}$ is Poisson with parameter
\[\sum_{k=1}^\infty\frac{\eta^{4k}}{8k}=-\frac{1}{8}\log(1 - \eta^4),\]
and the result follows.
\item If $c>e$ then by \cref{thm:WP-outputs}(A) we have $v=\Omega(n)$ and $m/v=1+\Omega(1)$ and $m/v=O(1)$ whp. Conditioning on such an outcome of $v,m$, by \cref{lem:rotate-core} and \cref{thm:main-RMT}(A) we have $\rank A(\core_{\mr{KS}}(G))=v-s(\core_{\mr{KS}}(G))$ whp, and the defect $s(\core_{\mr{KS}}(G))$ in the Karp--Sipser bound has the required asymptotic distribution.
\end{itemize}

Next we prove (B). For $i\in \{1,2\}$, let $v_i=v_i(\core_{\mr{KS}}(G))$ and let $m=v(\core_{\mr{KS}}(G))$. The number of leaf-removal steps is $n-v_1-i_1(G)=n-v_2-i_2(G)$, so
\[\rank B(G)=n-v_1-i_1(G)+\rank B(\core_{\mr{KS}}(G))=n-v_2-i_2(G)+\rank B(\core_{\mr{KS}}(G))\]
by \cref{lem:rank-decrement}(B). For \cref{thm:rank-characterisation} we need to prove that
\[\rank B(\core_{\mr{KS}}(G))=\min(v_1-s_1(\core_{\mr{KS}}(G)),v_2-s_2(\core_{\mr{KS}}(G)))\]
whp.
\begin{itemize}
\item If $c<e$ then by \cref{lem:subcritical-cycles}(B), whp $\core_{\mr{KS}}(G)$ is a vertex-disjoint union of even cycles (and therefore $v_1=v_2$). Using \cref{fact:biadjacency-rank} for $G=C_{2\ell}$ 
we see that
\[\rank B(C_{2\ell})=\begin{cases}
  \ell-1&\text{if }\ell\text{ is divisible by 2},\\
  \ell&\text{otherwise.}
\end{cases}\]
So, whp $\rank B(\core_{\mr{KS}}(G))=v_1-s_1(\core_{\mr{KS}}(G))=v_2-s_2(\core_{\mr{KS}}(G))$, as desired. Then, \cref{thm:rank-characterisation-distribution}(B) follows in essentially the same way as for \cref{thm:rank-characterisation-distribution}(A), using \cref{lem:subcritical-cycles}(B) for the joint cycle count distribution.
\item If $c>e$ then by \cref{thm:WP-outputs}(B) we have $v_1,v_2=\Omega(n)$ and $v_1=v_2+o(n)$ and $|v_1-v_2|=\omega(1)$ and $1+\Omega(1)\le m/(v_1+v_2)\le O(1)$ whp. Conditioning on such an outcome of $v_1,v_2,m$, the desired result follows from \cref{lem:rotate-core} and \cref{thm:main-RMT}(B) (for the $i$ minimising $v_i$, the defect in the Karp--Sipser bound is exactly $s_i(\core_{\mr{KS}}(G))$).\qedhere
\end{itemize}
\end{proof}

For \cref{thm:matching-vs-rank}(A1) we also need a counterpart of \cref{thm:rank-characterisation} for the matching number. Specifically, we need to know that $n-i(G)-2\nu(G)$ is bounded in probability; this follows from an exact characterisation of $\nu(G)$ essentially due to Frieze and Pittel~\cite{FP04}, as follows. For a graph $G$, let $q(G)$ be its number of isolated odd cycles.

\begin{theorem}\label{thm:matching-characterisation}
Fix a constant $c\ne e$. For $G\sim \mb G(n,c/n)$, whp
\[\nu(G)=\left\lfloor \frac{n-i(G)-q(\core_{\mr{KS}}(G))}2\right\rfloor.\]
Moreover, $q(\core_{\mr{KS}}(G))$ is bounded in probability.
\end{theorem}
\begin{proof}
Let $v=v(\core_{\mr{KS}}(G))$ and $m=e(\core_{\mr{KS}}(G))$. The number of leaf-removal steps is $(n-v-i(G))/2$, so
\[\nu(G)=(n-v-i(G))/2+\nu(\core_{\mr{KS}}(G)).\]
We need to prove that $\nu(\core_{\mr{KS}}(G))=\lfloor (v-q(\core_{\mr{KS}}(G)))/2\rfloor$ whp.
\begin{itemize}
    \item If $c<e$ then by \cref{lem:subcritical-cycles}(A), whp $\core_{\mr{KS}}(G)$ is a vertex-disjoint union of cycles; let $N_\ell$ be the number of such cycles of length $\ell$. Note that $\nu(C_\ell)=\lfloor\ell/2\rfloor$ so whp
\[\nu(\core_{\mr{KS}}(G))=\sum_{\ell=3}^\infty N_\ell\lfloor \ell/2\rfloor,\]
    from which the desired result follows.
    \item If $c>e$ then by \cref{thm:WP-outputs}(A) we have $v=\Omega(n)$ and $m/v=1+\Omega(1)$ and $m/v=O(1)$ whp. Condition on such an outcome of $v,m$, and note that by \cref{lem:rotate-core} we have $\core_{\mr{KS}}(G)\sim \mc K(v,m,2)$. A result of Frieze and Pittel~\cite[Theorem~2]{FP04}, which characterises the matching number of such random graphs whp, then implies the desired result.\qedhere
\end{itemize}
\end{proof}
\begin{proof}[Proof of \cref{thm:matching-vs-rank}]
\cref{thm:matching-vs-rank}(B) and (A2) follow directly from \cref{thm:rank-characterisation-distribution}, given \cref{cor:KS-bounds}. For (A1), we simply compare the formulas in \cref{thm:matching-characterisation,thm:rank-characterisation}.
\end{proof}
Finally, we deduce \cref{thm:critical} from \cref{thm:rank-characterisation-distribution} (more or less, we just need to observe that $\gamma_{\mr B}(c),\gamma_{\mr A}(c)\to \infty$ as $c\to e$).

\begin{proof}[Proof of \cref{thm:critical}]
A direct computation shows that $\gamma_{\mr B}(c),\gamma_{\mr A}(c)\to \infty$ as $c\to e$ (specifically, $\eta\to 1$ as $c\to e$ from below, and $\lamKS(c)\to 0$ as $c\to e$ from above). Also, by Chebyshev's inequality, for $Z\sim \mr{Poisson}(\gamma)$ we have $\Pr[Z\le\gamma/2]\le 4/\gamma$. So, for each $c$, if $n$ is sufficiently large (say $n\ge n_c$), in the setting of (A) we have
\[\Pr[n-i(G)-\rank A(G)<\gamma_{\mr A}(c)/2]<\frac{5}{\gamma_{\mr A}(c)}\]
and in the setting of (B) we have
\[\Pr[n-\max(i_1(G),i_2(G))-\rank B(G)<\gamma_{\mr B}(c)/2]<\frac{5}{\gamma_{\mr B}(c)},\]
by \cref{thm:rank-characterisation-distribution}. Letting $c_k=c+1/k$ for $k\ge 1$ and $k_n=\max(k\colon n\ge n_{c_k})$ for all $n\ge n_{c_1}$, the desired result follows by taking $p_n=c_{k_n}/n$ for $n\ge n_{c_1}$ (and say $p_n=0$ for $n<n_{c_1}$).
\end{proof}
\begin{remark}\label{rem:big-defect}
With more work, it seems to be possible to give an alternative (and more direct) proof of \cref{thm:critical} with stronger quantitative aspects. Indeed, if the leaf-removals in the Karp--Sipser process are performed one-by-one in a random order, then we obtain a randomly evolving ``partial Karp--Sipser core'' (which gradually shrinks over time until the final Karp--Sipser core is reached). If $np_n$ converges sufficiently rapidly to $e$, then using the differential equations method as in \cite{AFP98}, we believe that one can track the evolution of the partial Karp--Sipser core until a point where almost all vertices in the partial core have degree 2 (in the strong sense that the sum of degrees different from 2 is an $n^{-\Omega(1)}$-fraction of the total degree sum). Then, it is not hard to see that the partial core is uniform over all graphs with its degree sequence, and it should follow from a standard configuration-model calculation that there are $\Omega(\log n)$ isolated special cycles (which will end up as isolated special cycles in the final Karp--Sipser core, and will therefore each contribute to the defect in the Karp--Sipser bound).
\end{remark}

\section{Degree-constrained random graphs}\label{sec:core-estimates}
Most of the rest of the paper will be spent proving \cref{thm:main-RMT}, on the rank of degree-constrained random graphs of the form $\mc K(V,m,2)$ and $\mc K(V_1,V_2,m,2)$. In this section we first prove some basic properties about the degree sequence and edge distribution of such graphs.

First, a key observation is that both $\mc K(V,m,2)$ and $\mc K(V_1,V_2,m,2)$ are uniform given their degree sequence.

\begin{lemma}\label{lem:rotate-sequence-simple}
$\phantom.$
\begin{enumerate}
    \item[(A)] Consider $G\sim \mc K(V,m,2)$ for any $V,m$. If we condition on an outcome of $(\deg_G(v))_{v\in V}$, then conditionally $G$ is a uniformly random graph with this degree sequence.
    \item[(B)] Consider $G\sim \mc K(V_1,V_2,m,2)$ for any $V_1,V_2,m$. If we condition on outcomes of $(\deg_G(v_1))_{v_1\in V_1}$ and $(\deg_G(v_2))_{v_2\in V_2}$, then conditionally $G$ is a uniformly random bipartite graph with this pair of degree sequences.
\end{enumerate}
\end{lemma}
\begin{proof}
For (A), recall that $K(V,m,2)$ is a uniform distribution on graphs satisfying certain constraints on their degrees. So, if we condition on a particular degree sequence, the resulting distribution is uniform over all graphs with that degree sequence. Similar reasoning yields (B).
\end{proof}

With \cref{lem:rotate-sequence-simple} in hand, we can prove certain properties about $\mc K(V,m,2)$ and $\mc K(V_1,V_2,m,2)$ by first studying their degree sequence, then studying random graphs with given degree sequences. First, we can obtain a precise statistical understanding of the degree sequence using methods due to Cain and Wormald~\cite{CW06}: roughly speaking, the degree statistics can be approximated in terms of \emph{truncated Poisson} random variables, where we take a Poisson random variable $Z$ and condition on the event $Z\ge 2$.

\begin{lemma}\label{lem:degree-sequence-poisson}
Fix a constant $\varepsilon>0$.
\begin{enumerate}
    \item[(A)] For some $m,n$ satisfying $1+\varepsilon\le m/n\le 1/\eps$, let $G\sim \mc K(n,m,2)$, and choose $\lambda>0$ such that if $Z\sim \on{Poisson}(\lambda)$, then $2m/n=\mb E[Z|Z\ge 2]$. Then the following hold with probability at least $1-n^{-\omega(1)}$.
    \begin{enumerate}
    
    \item[(1)] For all $t\ge 2$, the number of vertices $v$ with $\deg(v)=t$ is $\rho_t n+O_\eps(\sqrt n\log n)$, where $\rho_t=\Pr[Z=t|Z\ge2]$.
    \item[(2)]$\displaystyle\sum_{v}\;\binom{\deg(v)}2=(E_2+o_{\eps}(1))n$, where $E_2=\mb E\left[\binom{Z}2\middle|Z\ge2\right]$ for $Z\sim \on{Poisson}(\lambda)$.
    \item[(3)] \vspace{3pt}For any $j\ge 3$,
$\displaystyle\sum_{v}\binom{\deg(v)}{j}\le e^{O_\varepsilon(j)} n.$

    \item[(4)] 
    For any set $S$ of $s$ vertices we have $\sum_{v\in S}\deg(v)\lesssim_{\varepsilon} s\log(2n/s)$.
    \end{enumerate}
    \item[(B)]
    For some $m,n_1,n_2$ satisfying $2+\varepsilon\le m/{n_1},m/{n_2}\le 1/\eps$, let $G\sim \mc K(n_1,n_2,m,2)$ (with parts $V_1,V_2$), and choose $\lambda_1,\lambda_2>0$ such that, for $i\in \{1,2\}$, if $Z_i\sim \on{Poisson}(\lambda_i)$, then $m/n_i=\mb E[Z_i|Z_i\ge 2]$. Then, writing $n=n_1+n_2$, the following hold with probability at least $1-n^{-\omega(1)}$.
    \begin{enumerate}
    
    \item[(1)] For all $i\in \{1,2\}$ and $t\ge 2$, the number of vertices $v\in V_i$ with $\deg(v)=t$ is $\rho_t^{(i)} n_i+O_\eps(\sqrt n\log n)$, where $\rho_t^{(i)}=\Pr[Z_i=t|Z_i\ge2]$.
    \item[(2)] For $i\in \{1,2\}$ we have $\displaystyle\sum_{v\in V_i}\;\binom{\deg(v)}2=(E_{2;i}+o_{\eps}(1))n_i$, where $E_{2;i}=\mb E\left[\binom{Z_i}2\middle|Z_i\ge2\right]$ for $Z_i\sim\on{Poisson}(\lambda_i)$.
    \item[(3)] \vspace{3pt}For any $j\ge 3$,
$\displaystyle\sum_{v}\binom{\deg(v)}{j}\le e^{O_\varepsilon(j)} n.$
    \item[(4)] 
    For any set $S$ of $s$ vertices we have $\sum_{v\in S}\deg(v)\lesssim_{\varepsilon} s\log(2n/s)$.
    \end{enumerate}
\end{enumerate}
\end{lemma}

To prove \cref{lem:degree-sequence-poisson} we need the following observation (due to Bollob\'as, Cooper, Fenner, and Frieze~\cite{BCFF00} in the non-bipartite case), that the degree sequences in the setting of \cref{lem:degree-sequence-poisson} can be effectively approximated by a sequence of independent truncated Poisson random variables. (We write $\on{Poisson}_{\ge 2}(\lambda)$ for the conditional distribution of $Z\sim\on{Poisson}(\lambda)$ given $Z\ge 2$.) 
\begin{lemma}\label{lem:CW-comparison}
Fix $\varepsilon>0$.
\begin{enumerate}
    \item[(A)] Consider $m,n$ satisfying $1+\varepsilon\le m/n\le 1/\varepsilon$, and choose $\lambda>0$ such that if $Z\sim \on{Poisson}(\lambda)$ then $2m/n=\mb E[Z|Z\ge 2]$. For $V=\{1,\ldots,n\}$, let $G=\mc K(V,m,2)$ and let $(T_{v})_{v\in V}$ be a sequence of independent $\on{Poisson}_{\ge 2}(\lambda)$ random variables. Then for any $\mbf t\in \mb N^V$ we have
\[\Pr[(\deg_G(v))_{v\in V}=\mbf t]\lesssim_\varepsilon \sqrt n \Pr[(T_v)_{v\in V}=\mbf t].\]
    \item[(B)] Consider $m,n_1,n_2$ satisfying $2+\varepsilon\le m/n_1,m/n_2\le 1/\varepsilon$, and choose $\lambda_1,\lambda_2>0$ such that, for $i\in \{1,2\}$, if $Z_i\sim \on{Poisson}(\lambda_i)$ then $m/n_i=\mb E[Z_i|Z_i\ge 2]$. For a partition $V=V_1\cup V_2$ into two parts of sizes $n_1,n_2$, let $G\sim \mc K(V_1,V_2,m,2)$, and let $(T_{v})_{v\in V}$ be a sequence of independent truncated Poisson random variables with $T_v\in \on{Poisson}_{\ge 2}(\lambda_i)$ whenever $v\in V_i$. Then for any $\mbf t\in \mb N^V$ we have
\[\Pr[(\deg_G(v))_{v\in V}=\mbf t]\lesssim_\varepsilon n \Pr[(T_v)_{v\in V}=\mbf t].\] 
\end{enumerate}
\end{lemma}

\begin{proof}[Proof sketch]
Part (A) appears as \cite[Lemma~1]{BCFF00}. It is proved by considering a random multigraph distribution (called $\mc M(n,2m,2)$ in \cite{BCFF00}; the edges are just a sequence of $m$ independent random pairs of vertices sampled with replacement, conditioned on all degrees being at least 2) and observing that the following hold.
\begin{itemize}
    \item If one conditions on the event that this random multigraph is simple (which occurs with probability $\Omega_\eps(1)$), then one obtains the graph distribution $\mc K(V,m,2)$.
    \item The degree sequence $(\deg_G(v))_{v\in V}$ of this random multigraph has precisely the conditional distribution of $(T_v)_{v\in V}$ given $\sum_v T_v=2m$ (in \cite{BCFF00} this conditional distribution is called $\mc O(n,\lambda,2)$, and the unconditional distribution is called $\mc P(n,\lambda,2)$). Moreover, the event $\sum_v T_v=2m$ occurs with probability $\Omega_\eps(1/\sqrt n)$. (Roughly speaking, this is because $\sum_v T_v$ has standard deviation $O_\eps(\sqrt n)$, and is more-or-less uniform over integers within standard-deviation-range of the mean.)
\end{itemize}

For part (B), we can consider an analogous \emph{bipartite} random multigraph distribution on the vertex set $V_1\cup V_2$: consider $m$ independent random edges between $V_1$ and $V_2$, conditioned on all degrees being at least 2. Then, we analogously observe that the following hold.
\begin{itemize}
    \item If one conditions on the event that this random multigraph is simple (which occurs with probability $\Omega_\eps(1)$), then one obtains the graph distribution $\mc K(V_1,V_2,m,2)$.
    \item The degree sequence $(\deg_G(v))_{v\in V_1\cup V_2}$ of this random multigraph has precisely the conditional distribution of $(T_v)_{v\in V_1\cup V_2}$ given $\sum_{v\in V_1} T_v=\sum_{v\in V_2} T_v=m$. Moreover, the event $\sum_{v\in V_1}T_v=\sum_{v\in V_2} T_v=m$ occurs with probability $\Omega_\eps((1/\sqrt n)\cdot(1/\sqrt n) )=\Omega_\eps(1/n)$.\qedhere
\end{itemize}
\end{proof}

Now we prove \cref{lem:degree-sequence-poisson}.

\begin{proof}[Proof of \cref{lem:degree-sequence-poisson}]
We just prove (A); the proof of (B) is essentially identical. Let $V=\{1,\ldots,n\}$ and let $(T_v)_{v\in V}$ be a sequence of independent $\on{Poisson}_{\ge 2}(\lambda)$ random variables. By \cref{lem:CW-comparison}, to prove that a property of the degree sequence $(\deg_G(v))_{v\in V}$ holds with probability $1-n^{-\omega(1)}$, it suffices to prove that $(T_v)_{v\in V}$ satisfies this property with probability $1-n^{-\omega(1)}$. So, we work only with $(T_v)_{v\in V}$.

First, for each $t$, we have $\Pr[T_v=t]=\rho_t$ for each $v$ independently. So, (1) holds with the desired probability, by a Chernoff bound and a union bound. Also, by a Chernoff bound for the Poisson distribution (see for example \cite[Theorem~5.4]{MU17}), for each $v$ and $t\ge \lambda$ we have
\begin{equation}\label{eq:Poisson-chernoff}
\Pr[T_v\ge t]\le (\Omega_\eps(t))^{-t}.
\end{equation}
This (together with the union bound) implies that with probability $1-n^{-\omega(1)}$ we have say
\begin{equation}\label{eq:degree-bound}
T_v\le \log n\text{ for all }v\in V.
\end{equation}

If (1) and \cref{eq:degree-bound} hold, then trivially (4) holds whenever say $s\le n^{0.9}$. For the case $s>n^{0.9}$, note that (when (1) and \cref{eq:degree-bound} hold) for any $t\ge \lambda$ we have
\[\sum_{d_v \ge t}d_v\le \sum_{j=t}^{\log n} j\left((\Omega_\eps(j))^{-j} n + \log n\cdot O_\eps(\sqrt{n})\right)\le (\Omega_\eps(t))^{-t}n + \sqrt{n}(\log n)^4.\]
Taking $t = C\log(2n/s)$ for $1/C\ll \varepsilon$, we have that
\[
 \sum_{v \in S}d_v\le ts + \sum_{d_v \ge t} d_v
\le ts+ (s/n)n + \sqrt{n}(\log n)^4\le O_{\eps}(s\log(n/s))
\]
so (4) holds in this case too.
Next, note that
\[\mb E\left[\sum_{v}\;\binom{T_{v}}2\right]=nE_2,\]
but (2) does not immediately follow from an off-the-shelf concentration inequality, since $\on{Poisson}_{\ge 2}(\lambda)$ is a distribution with unbounded support. Let
\[P=\sum_{v}\;\binom{\min(\log n,T_{v})}2.\]
Recalling \cref{eq:degree-bound}, it suffices to prove that (with probability $1-n^{-\omega(1)}$) $P$ satisfies the estimate in (2). To this end, note that
\begin{align*}
\mb E\left[\binom{T_{v}}2-\binom{\min(\log n,T_{v})}2\right]&\le \sum_{t= \lceil\log n\rceil}^\infty O(t^2)\Pr[T_v\ge t]\\
&\le \sum_{t=\lceil\log n\rceil} O(t^2)(\Omega_\eps(t))^{-t}=o_\eps(1),
\end{align*}
by \cref{eq:Poisson-chernoff}, so $\mb E P=(E_2+o(1))n$.
Also, note that changing some $T_{v}$ changes $P$ by at most $(\log n)^2$, so by the Azuma--Hoeffding inequality (see \cref{lem:azuma}), we have $|P-\mb E P|\le n^{1/2+o(1)}$ with probability $1-n^{-\omega(1)}$, from which (2) follows. 

Finally, we prove (3). By \cref{eq:degree-bound}, it suffices to consider the case $3\le j\le\log n$. Then, when (1) and \cref{eq:degree-bound} hold, using \cref{eq:Poisson-chernoff}, we have
\begin{align*}
\sum_v T_v^j=\int_{0}^\infty |\{v\colon T_v^j\ge t\}|dt&= \int_{0}^\infty |\{v\colon T_v\ge s\}|js^{j-1}ds\\
&=\int_{0}^{\log n} \left((\Omega_\eps(s))^{-s}n+\sqrt n(\log n)^3\right)j s^{j-1}ds+O_\eps(n)\\
&=n\int_{0}^\infty js^{j-1}(\Omega_\eps(s))^{-s} ds+O_{\eps}\left(n + \sqrt{n}(\log n)^3\cdot (\log n)^j\right).
\end{align*}
Now, for any $c>0$ we have
\[\int_0^\infty js^{j-1}(cs)^{-s}ds=(O_c(j))^{j}.\]
Also note that for any $x>0$ we have $(x/j)^j\le \exp(x/e)$. Taking $x=\log n/100$, it follows that $(\log n)^j\le n^{1/3}(O(j))^{j}$. We deduce $\sum_v T_v^j\le n(O_\eps(j))^{j}$. Since $j!\ge (\Omega(j))^j$ by Stirling's approximation, we then deduce $\sum_v T_v^j/j!\le e^{O_\eps(j)}n$, from which (3) follows.
\end{proof}

For disjoint sets $V_1,V_2$ and sequences $\mbf d^1\in \mb N^{V_1},\mbf d^2\in \mb N^{V_2}$, write $\mb G(\mbf d^1,\mbf d^2)$ to denote the uniform distribution on bipartite graphs with degree sequence specified by $(\mbf d^1,\mbf d^2)$. For a set $V$ and a sequence $\mbf d\in \mb N^V$, write $\mb G(\mbf d)$ to denote the uniform distribution on graphs with degree sequence $\mbf d$. Now, to work with random graphs of the form $\mb G(\mbf d),\mb G(\mbf d^1,\mbf d^2)$, we use an auxiliary random graph model called the \emph{configuration model}. This model was first explicitly considered in 1980 by Bollob\'as~\cite{Bol80} (though similar ideas were considered earlier by various authors~\cite{BC78,BBK72,WorThesis}), and has since become an indispensable tool in random graph theory.

\begin{definition}\label{def:config-model}
For a degree sequence $\mbf d=(d_1,\ldots,d_n)$, consider a set of $r=d_1+\cdots+d_n$ ``stubs'', grouped into $n$ labelled ``buckets'' of sizes $d_1,\ldots,d_n$. A \emph{configuration} is a perfect matching on the $r$ stubs, consisting of $r/2$ disjoint edges. Given a configuration, contracting each of the buckets to a single vertex gives rise to a multigraph with degree sequence $d_1,\ldots,d_n$ (where we use the convention that loops contribute 2 to the degree of a vertex).
\begin{enumerate}
    \item[(A)] For a set $V$ and a degree sequence $\mbf d\in \mb N^V$, let $\mb G^\ast(\mbf d)$ be the random multigraph distribution obtained by contracting a uniformly random configuration.
    \item[(B)] For disjoint sets $V_1,V_2$ and a pair of sequences $\mbf d^{1}\in \mb N^{V_1},\mbf d^2\in \mb N^{V_2}$, let $\mb G^\ast(\mbf d^1,\mbf d^2)$ be the random bipartite multigraph distribution obtained by contracting a uniformly random configuration in which we only allow edges between the buckets corresponding to $V_1$ and the buckets corresponding to $V_2$.
\end{enumerate}
\end{definition}

The uniform models $\mb G(\mbf d),\mb G(\mbf d^1,\mbf d^2)$ can be closely compared with their configuration models, as follows.

\begin{lemma}\label{lem:simplicity}
Fix $C>0$.
\begin{enumerate}
    \item[(A)]
    \begin{enumerate}
        \item[(1)] For any set $V$ and sequence $\mbf d\in \mb N^V$, if we consider $G^\ast\in \mb G^\ast(\mbf d)$ and condition on $G^\ast$ being a simple graph, then we recover the distribution $\mb G(\mbf d)$.
        \item[(2)] If the sum of squares of entries of $\mbf d$ is at most $Cn$, then the probability that $G^\ast$ is simple is $\Omega_C(1)$.
    \end{enumerate}
    \item[(B)]
    \begin{enumerate}
        \item[(1)] For any disjoint sets $V_1,V_2$ and sequences $\mbf d^1\in \mb N^{V_1},\mbf d^2\in \mb N^{V_2}$, if we consider $G^\ast\in \mb G^\ast(\mbf d^1,\mbf d^2)$ and condition on $G^\ast$ being a simple graph, then we recover the distribution $\mb G(\mbf d^1,\mbf d^2)$.
        \item[(2)] If the sum of squares of entries of $\mbf d$ is at most $Cn$, then the probability that $G^\ast$ is simple is $\Omega_C(1)$.
    \end{enumerate}
\end{enumerate}
\end{lemma}

Parts (A1) and (B1) follow from the (easy) fact that each simple graph corresponds to the same number of configurations. Parts (A2) and (B2) of \cref{lem:simplicity} were first proved by Janson~\cite{Jan09} and Blanchet and Stauffer~\cite{BS13}, respectively. We remark that Janson~\cite{Jan14} later gave a simplified proof for both (A2) and (B2), and that many authors previously proved various special cases (see for example \cite{BC78,Bol80,Bol01,McK85,MW91,BBK72,GMW06,McK85}). Several of these special cases are sufficient for the applications in this paper.

The advantage of the configuration model is that it has much more independence than a uniformly random graph with a given degree sequence, and is therefore much easier to study.

We finish this section with a simple expansion estimate for random graphs with given degree sequences.

\begin{lemma}\label{lem:dense-subset}
Fix $\varepsilon>0$, and consider one of the following two situations.
\begin{enumerate}
    \item[(A)] Suppose $m,n$ satisfy $1+\varepsilon\le m/n\le 1/\varepsilon$, and let $G\sim \mc K(n,m,2)$.
    \item[(B)] Suppose $m,n,n_1,n_2$ satisfy $2+\varepsilon\le m/{n_1},m/{n_2}\le 1/\eps$, let $G\sim \mc K(n_1,n_2,m,2)$, and let $n=n_1+n_2$.
\end{enumerate}
In both situations, the following properties hold for $n$ large.
\begin{enumerate}
    \item There exists $C =C_{\ref{lem:dense-subset}}(\varepsilon)>0$ such that the following holds (for large enough $n$). With probability at least $1-O(1/\sqrt n)$: for all $s$, every subgraph of $G$ with $s$ vertices has at most $s+\lfloor Cs/\sqrt{\log(2n/s)}\rfloor$ edges.
    \item Whp, the number of cycles of length less than $\log \log n$ is at most $\exp((\log \log n)^3)$.
\end{enumerate}
\end{lemma}

We did not attempt to prove the absolute best bounds possible; for example, with more care, in the setting of (1) it seems one can prove an upper bound of roughly $s+s/\log (n/s)$ when $s$ is not too large.

\begin{remark}\label{rem:f}
In practice, we will apply (1) in the case where $s/n$ is small with respect to $\varepsilon$. So, $\lfloor Cs/\sqrt{\log (n/s)}\rfloor$ can be thought of as a ``lower order term'' relative to $s$. In particular, when (say) $s< \sqrt{\log n}/(2C)$, we have $\lfloor Cs/\sqrt{\log (n/s)}\rfloor=0$, meaning that no set of $s$ vertices has more than $s$ edges (i.e., we cannot have anything denser than a cycle).
\end{remark}

\begin{proof}
We only prove (A); the proof of (B) is essentially identical. We handle (1) and (2) together, considering what happens more generally for a set of size $s$ with at least $s+t$ edges.

Let $\mbf d=(d_1,\ldots,d_n)$ be the degree sequence of $G$, and condition on any outcome of $\mbf d$ satisfying the conclusion of \cref{lem:degree-sequence-poisson}(A). By \cref{lem:simplicity,lem:rotate-sequence-simple}, it suffices to prove the desired result for $G\sim \mb G^\ast(\mbf d)$ (note that \cref{lem:degree-sequence-poisson}(A2) implies that $d_1^2+\cdots+d_n^2=O_\eps(n)$).

Consider any $t\le 2s$ with say $s\le n/6$. For a set $S$ of $s$ vertices, the probability that $G[S]$ contains at least $s+t$ edges is at most
\[\binom{\sum_{v\in S}d_v}{2s + 2t}\bigg(\frac{6s}{n}\bigg)^{s+t}.\]
Indeed, we are considering the probability of the event that there is some set of $2s+2t$ stubs from the buckets corresponding to vertices in $S$, which all pair among themselves (in our random configuration). For any set of $2s+2t$ stubs from $S$, the probability that they all pair among themselves is at most $(6s/n)^{s+t}$, since $2s+2t\le 6s$ and $2m-(2s+2t)\ge 2n-6s\ge n$.

So, the expected number of sets of $s$ vertices with at least $s+t$ edges is at most
\begin{align}
\sum_{\substack{S\subseteq \{1,\ldots,n\}:\\|S|=s}} \binom{\sum_{v\in S}d_v}{2s + 2t}\bigg(\frac{6s}{n}\bigg)^{\!s+t}
\!\!&\le \bigg(\frac{n}{s}\bigg)^{\!s}\bigg(\frac{s\log(2n/s)}{s+t}\bigg)^{\!2s+2t}\bigg(\frac{s}{n}\bigg)^{\!s+t}\!e^{O_\varepsilon(s)}\notag\\
&\le\left(O_\eps\left(\log \left(\frac{n}{s}\right)\right)\right)^{\!O(s)}\bigg(\frac{s}{n}\bigg)^{\!t},\label{eq:expansion-estimate}
\end{align}
where in the first inequality we used that $\binom{n}{s}\le (en/s)^s$ and we used \cref{lem:degree-sequence-poisson}(A4) (which says that $\sum_{v\in S} d_v\lesssim_\varepsilon s\log(2n/s)$). We immediately deduce (2), taking $t=0$ and summing over $s< \log \log n$.

For (1), let $t=\lfloor s/\sqrt{\log(2n/s)}\rfloor+1$; we will prove that for sufficiently small $c=c(\varepsilon)>0$, every subgraph  with $s\le cn$ vertices has fewer than $s+t$ edges. The desired result will then follow, noting that (1) trivially holds for subgraphs with $s>c n$ vertices (taking $C$ large in terms of $c$).

So, we sum the estimate in \cref{eq:expansion-estimate} over all $s\le cn$. The contribution from say $s\le (\log n)^{2/3}$ is $n^{-1+o_{\eps}(1)}\le 1/(2\sqrt n)$, and the contribution from $s>2C\sqrt{\log n}$ is at most
\begin{align*}
&\sum_{s=\lfloor(\log n)^{2/3}\rfloor+1}^{\lfloor cn\rfloor}\left(O_\eps\left(\log \left(\frac{n}{s}\right)\right)\right)^{s}\left(\frac{s}{n}\right)^{\Omega(s/\sqrt{\log(n/s)})}\\
&\qquad\qquad\le n\!\!\max_{cn\ge s\ge (\log n)^{2/3}} \exp\left(s\left(O_\varepsilon(1)+O\left(\log \log \left(\frac{n}{s}\right)\right)-\Omega\left(\sqrt{\log \left(\frac{n}{s}\right)}\right)\right)\right)\\
&\qquad\qquad\le n\!\!\max_{cn\ge s\ge (\log n)^{2/3}} \exp\left(-\Omega_\varepsilon(s\sqrt{\log(n/s)})\right)\le \frac{1}{2\sqrt n}
\end{align*}
for small enough $c$.
\end{proof}

\section{Random walk analysis}\label{sec:random-walk}
The following lemma is a slight adaptation of \cite[Lemma~5.2]{FKSS21} (which is itself a variation on \cite[Lemma~2.9]{CTV06}). Roughly speaking, it says that certain negatively biased random walks typically end up at a nonpositive value. 

\begin{lemma}\label{thm:random-walk-cheap}
Fix $C,\delta,\varepsilon>0$. Let
$X_{N},\ldots,X_{0}$ be a sequence of real random variables satisfying
the following conditions for some $p\in(0,1)$.
\begin{enumerate}[{\bfseries{W\arabic{enumi}'}}]
    \item\label{W1'} $X_{N}\le(1-\varepsilon)\delta N$ (with probability 1)
    \item\label{W2'} $X_{t}\le X_{t+1}+C$ for all $t\le N$ (with probability 1).
    \item\label{W3'} For any
$t\le N-1$ and any $x_{N},\ldots,x_{t+1}$:
\begin{enumerate}
\item if $x_{t+1}>0$ then $\Pr[X_{t}\le x_{t+1}-\delta\,|\,X_{N}=x_{N},\ldots,X_{t+1}=x_{t+1}]\ge 1-p$.
\item if $x_{t+1}\le0$ then $\Pr[X_{t}\le0\,|\,X_{N}=x_{N},\ldots,X_{t+1}=x_{t+1}]\ge1-p$.
\end{enumerate}
\end{enumerate}
Then 
\[
\Pr[X_{0}>0]\le O_{C,\delta,\varepsilon}(p).
\]
\end{lemma}
Informally, condition \cref{W3'} says  that our random walk ``wants to be nonpositive'': when we are positive we tend to go down at the next step, and when we are nonpositive we tend to stay nonpositive at the next step.
\begin{proof}
First, note that the statement is trivial if say $N\le 10$. Then, note that we can reduce to the case where $\delta=1$ and each $X_t$ is a nonnegative integer. Indeed, define $X_t'=\max(0,\lceil X_t/\delta\rceil)$. Note that conditions \cref{W1',W2',W3'} 
are still satisfied for $X_N',\ldots,X_0'$ (with ``1'' in place of ``$\delta$'', with ``$\lceil C/\delta+1\rceil$'' in place of ``$C$'', and with say $\varepsilon/2$ in place of $\varepsilon$), and note that $X_0>0$ implies $X_0'>0$. After this reduction, the lemma statement is a slight variant of \cite[Lemma~5.2]{FKSS21} (and can be proved in the same way).
\end{proof}

We will need a generalisation of \cref{thm:random-walk-cheap} permitting a small number of ``bad steps'' in
which we have no control over the behaviour of our random walk. Crucially,
the set of bad steps is allowed to depend on the trajectory of the
random walk; we only assume that the bad steps are unlikely to concentrate near the end of the walk.

\begin{theorem}\label{thm:random-walk-master}
Fix $C,\delta,\varepsilon>0$ with $C/\delta \geq 1$. Let
$X_{N},\ldots,X_{0}$ be a sequence of real random variables, and let $R\subseteq\{0,\ldots,N\}$ be a random set of ``bad steps'', satisfying
the following conditions for some $p\in(0,1)$.
\begin{enumerate}[{\bfseries{W\arabic{enumi}}}]
\setcounter{enumi}{-1}
\item\label{W0} There is an underlying sequence of random elements $G_N,\ldots,G_0$, such that $X_t$ and the event $\{t\in R\}$ are both determined by $G_t$ (for all $t\le N$).
\item\label{W1} $X_{N}\le(1-\varepsilon)\delta N$ (with probability 1)
\item\label{W2} $X_{t}\le X_{t+1}+C$ for all $t\le N$ (with probability
1).
\item\label{W3} For any
$t\le N-1$ and any $x_{t+1}\in \mb R$, and any outcomes $g_N,\ldots,g_{t+1}$ of $G_N,\ldots,G_{t+1}$ satisfying $X_{t+1}=x_{t+1}$:
\begin{enumerate}
\item if $x_{t+1}>0$ then $\Pr[X_{t}\le x_{t+1}-\delta\emph{ or }t\in R\,|\,G_N=g_N,\ldots,G_{t+1}=g_{t+1}]\ge 1-p$.
\item if $x_{t+1}\le0$ then $\Pr[X_{t}\le0\emph{ or }t\in R\,|\,G_N=g_N,\ldots,G_{t+1}=g_{t+1}]\ge1-p$.
\end{enumerate}
\item\label{W4} With probability at least $1-p$, for each $t\le N$ we have $|R\cap \{0,\ldots,t\}|\le \delta t/(32C)$.
\end{enumerate}
Then 
\[
\Pr[X_{0}>0]\le O_{C,\delta,\varepsilon}(p).
\]
\end{theorem}

\begin{proof}
As with \cref{thm:random-walk-cheap}, we can assume that $\delta=1$, each $X_t$ is a nonnegative integer, and $C$ is a nonnegative integer, if we prove the result under a slight weakening of condition \cref{W4} that $|R\cap \{0,\ldots,t\}|\le \delta t/(8C)=t/(8C)$. To see this, take $X_t'=\max(0,\lceil X_t/\delta)\rceil)$ and then note that conditions \cref{W0,W1,W2,W3} and the weaker version of \cref{W4} are still satisfied for $X_N',\ldots,X_0'$ (with ``1'' in place of ``$\delta$'', with ``$\lceil C/\delta+1\rceil$'' in place of ``$C$'', and with say $\varepsilon/2$ in place of $\varepsilon$). We will deduce \cref{thm:random-walk-master} from \cref{thm:random-walk-cheap}.

Inductively define a sequence $Y_N,\ldots,Y_0$ by taking $Y_N=0$ and
\[Y_t=Y_{t+1}-(C+1){\mbm 1}_{t\in R}+(1/2)\mbm 1_{Y_{t+1}<0}\]
for $t<N$.
The reader may wish to imagine a ``cost'' of $C+1$ being incurred at every bad step, and that this is repaid over the future of the process (specifically, $1/2$ is repaid per step, until all debts are repaid). Define a modified sequence $X_N',\ldots,X_1'$ by $X_t'=X_t+Y_t$. The idea is that $Y_t$ ``compensates'' if $t$ is a bad step. Notice that $Y$ is half-integral, and non-positive.

Note that $X_N'=X_N$ and $X_t'\le X_{t+1}'+C+1/2$ for all $t\le N$ (so our modified sequence still satisfies a version of properties \cref{W1} and \cref{W2} in \cref{thm:random-walk-cheap}). We now verify that $X_t'$ satisfies \cref{W3} in \cref{thm:random-walk-cheap} with $\delta=1/2$. Note that if $t \in R$, this is immediate, since $Y_t \leq Y_{t + 1} - C - 1/2$ and $X_t \leq X_{t + 1} + C$. If $t \notin R$ and $X_{t + 1} > 0$, then with probability at least $1 - p$, we have $X_t \leq X_{t + 1} - 1$, and deterministically, $Y_t \leq Y_{t+1} + 1/2$. Finally, if $t \notin R$ and $X_{t + 1} = 0$, then it must be the case that $X'_{t + 1} \leq 0$, since $Y$ is non-positive. Then with probability at least $1 - p$, $X_t$ stays $0$, and thus $X'_t \leq 0$.

Applying \cref{thm:random-walk-cheap}, we see that $\Pr[X_{0}'>0]\le O_{C,\delta,\varepsilon}(p)$, and it suffices to prove that $Y_0= 0$ (i.e., that $X_0= X_0'$) with probability at least $1-p$.

To see this, for each $i=0,1,\ldots,\lfloor\log_2 (N+2)\rfloor$ let $Q_i$ be the number of bad steps $t\in R$ in the range $[2^i-1,2^{i+1}-1)$. Note that if the inequality $(C+1)Q_i< (1/2)2^{i-1}$ holds for all $i$, then $Y_0\ge 0$: for each $i\ge 1$, the ``cost'' incurred in the interval $[2^i-1,2^{i+1}-1)$ is ``repaid'' in the $2^{i-1}$ steps in the interval $[2^{i-1}-1,2^{i}-1)$. The above is guaranteed by our weaker version of \cref{W4}, taking $t=2^{i+1}-2$ for each $i$.
\end{proof}

\section{Boosting the rank}\label{sec:boost-corank}
In this section we prove some general lemmas studying how the rank
of a matrix changes when a random row/column is added to it. The lemmas in this section represent the main difference between parts (A) and (B) of \cref{thm:main-RMT}: there are certain additional dependencies involved when dealing with symmetric random matrices.

First, the following simple lemma will be used for \cref{thm:main-RMT}(B).

\begin{definition}\label{def:balanced}
The \emph{$\lambda$-level set} of a vector is the set
of all entries equal to $\lambda$. Say that a vector $\mathbf{v}\in\mathbb{R}^{n}$
is \emph{$\eta$-balanced} if all of the level sets of $\mathbf{v}$
have size at most $n(1-\eta)$.
\end{definition}

\begin{lemma}\label{lem:corank-boosting-bipartite}
Fix $0<\eta<1/2$, let $d\ge1$, and consider a matrix $B\in\mathbb{R}^{n_1\times n_2}$. Consider a subset $E\subseteq \{1,\ldots,n_1\}$ with size at least $n_1(1-\eta/3)$, and let $\mathbf{x}=(x_{1},\ldots,x_{n_1})$ be a random zero-one vector,
such that the restriction $\mathbf{x}_{E}$ to the entries indexed
by $E$ is a uniformly random zero-one vector with exactly $d$ ones
(and the restriction $\mathbf{x}_{\overline{E}}$ to entries not indexed
by $E$ is deterministic). Add $\mathbf{x}$ as a new column of $B$ to obtain a new matrix $B'$.

If $B^\intercal $ has an $\eta$-balanced kernel vector, then $\rank(B')\ge \rank(B)+1$ with probability $1-O_{\eta}(1/\sqrt{d})$.
\end{lemma}

\cref{lem:corank-boosting-bipartite} concerns the addition of a new random column, but of course it symmetrically applies to the addition of a new random row (we can simply consider the transpose of $B$). In fact, in our proof of \cref{thm:main-RMT}(B) we will use \cref{lem:corank-boosting-bipartite} to show that when a random row and a random column are independently added, the rank increases by 2.

The following more sophisticated lemma will be used for \cref{thm:main-RMT}(A).

\begin{definition}\label{def:unstructured}
Say that a symmetric matrix $A\in\mathbb{R}^{n\times n}$
is \emph{$\eta$-unstructured }if there are at least $\eta n^{2}$ pairs of
distinct entries $(i,i')\in\{1,\ldots,n\}^{2}$ such that if $\on{supp}(A\mathbf{v})=\{i,i'\}$
then $\mathbf{v}$ is $\eta$-balanced.
\end{definition}

\begin{lemma}
\label{lem:corank-boosting-master}Fix $0<\eta<1/2$ and let
$d\ge1$. Consider a symmetric matrix $A\in\mathbb{R}^{n\times n}$
and a subset $E\subseteq\{1,\ldots,n\}$ of size at least $n(1-\eta/3)$.
Let $\mathbf{x}=(x_{1},\ldots,x_{n})$ be a random zero-one vector,
such that the restriction $\mathbf{x}_{E}$ to the entries indexed
by $E$ is a uniformly random zero-one vector with exactly $d$ ones
(and the restriction $\mathbf{x}_{\overline{E}}$ to entries not indexed
by $E$ is deterministic). Add $\mathbf{x}$ as a new row and column
of $A$ (and put a zero in the new diagonal entry) to obtain a new symmetric
matrix $A'$.
\begin{enumerate}
\item[(a)] If $A$ has a $\eta$-balanced kernel vector,
then $\rank(A')=\rank(A)+2$ with probability $1-O_{\eta}(1/\sqrt{d})$.
\item[(b)] Let $S=\on{supp}(\ker A)$. If $A$ is $\eta$-unstructured and $|S|\le \eta^2 n/32$ and $S\cap E=\emptyset$ and $\mbf x_S=\mbf 0$ then $\rank(A')\ge\rank(A)+1$ with
probability at least $1-O_\eta((\log 2d)^{O(1)}/\sqrt{d})$.
\end{enumerate}
\end{lemma}

First, \cref{lem:corank-boosting-bipartite} and \cref{lem:corank-boosting-master}(a) will be simple consequences of an anti-concentration inequality for linear polynomials on the ``Boolean slice'' (i.e., for linear polynomials of uniformly random binary vectors with a prescribed number of 1s). The following lemma is a direct consequence of \cite[Lemma~4.2]{FKS21} (a similar inequality also appears in \cite{LLTTY17}), and is proved using the \emph{Erd\H{o}s--Littlewood--Offord theorem} (see for example \cite[Chapter~7]{TV10}).

\begin{lemma}\label{lem:linear-LO-slice}
Let $1\le d\le n/2$, $\kappa>0$ and let $\mbf{v}=(v_{1},\ldots,v_{n})\in \mb R^n$
be a $\kappa$-balanced vector. Let $\mbf{x}=(x_{1},\ldots,x_{n})\in \{0,1\}^n$
be a random vector, uniformly selected from the zero-one
vectors with exactly $d$ ones, and consider any $y\in \mb R$. Then
\[
\Pr[\mbf{v}^\intercal\mbf{x}=y]=O((\kappa d)^{-1/2}).
\]
\end{lemma}

Now we prove \cref{lem:corank-boosting-bipartite} and \cref{lem:corank-boosting-master}(a).
\begin{proof}[Proof of \cref{lem:corank-boosting-bipartite}]
Let $\mbf{v}\in \mb R^{n_1}$ be an $\eta$-balanced kernel vector of $B^\intercal $ (note that this means that $\mbf v$ lies in the orthogonal complement of the column space of $B$). Since $\mbf v$ is $\eta$-balanced and $E\ge (1-\eta/3)n_1$, every level set of $\mbf v_E$ has size at most $((1-\eta)/(1-\eta/3))|E|$, implying that $\mbf v_E$ is $\Omega(\eta)$-balanced. Therefore it follows by \cref{lem:linear-LO-slice} that $\mb{P}[\mbf v^\intercal \mbf x = 0] =\mb{P}[\mbf v_{E}^\intercal \mbf x_{E} = -\mbf v_{\overline E}^\intercal \mbf x_{\overline E}] = O_{\eta}(1/\sqrt{d})$. But note that if $\mbf v^\intercal \mbf x \ne 0$, then $\mbf x$ does not lie in the column space of $A$, and its addition as a new column increases the rank.
\end{proof}

\begin{proof}[Proof of \cref{lem:corank-boosting-master}(a)]
Let $\mbf{v}$ be an $\eta$-balanced kernel vector of $A$ (or equivalently, of $A^\intercal $), so $\mbf v$ lies in the orthogonal complement of the row space of $A$ (or equivalently, the orthogonal complement of the column space). As above, $\mbf v_E$ is $\Omega(\eta)$-balanced, so by \cref{lem:linear-LO-slice} we have $\mb{P}[\mbf v^\intercal \mbf x = 0] =O_{\eta}(1/\sqrt{d})$. If $\mbf v^\intercal \mbf x \ne 0$, then $\mbf x$ does not lie in the row space or column space of $A$, so adding $\mbf{x}$ as a
new row and column increases the rank twice. The desired result follows.
\end{proof}

For \cref{lem:corank-boosting-master}(b) we need an anti-concentration inequality for \emph{quadratic} polynomials of random vectors $(x_1,\ldots,x_n)\in \{0,1\}^n$ on the Boolean slice. The following lemma appears as \cite[Proposition~3.4]{FKSS21}, and is proved using an inequality of Kane~\cite{Kan14}.

\begin{lemma}\label{prop:quadratic-LO-general}
Let $M=(m_{ij})_{i,j}$ be an $n\times n$ symmetric matrix for
which there are $\Omega(n^{4})$ different 4-tuples $(i,i',j,j')$
with $m_{ij}-m_{i'j}-m_{ij'}+m_{i'j'}\ne0$. Let $\mbf{x}=(x_{1},\ldots,x_{n})$
be a random zero-one vector, uniformly selected from the zero-one
vectors with exactly $d\le n/2$ ones. Then for any vector $\mbf{v}\in\mb R^{n}$
and any $x\in\mb R$ we have 
\[
\Pr[\mbf{x}^\intercal M\mbf{x}+\mbf{v}^\intercal\mbf{x}=x]\le O((\log 2d)^{O(1)}/\sqrt{d}).
\]
\end{lemma}

We also need the following lemma implicit in the work of DeMichele, the first author, and Moreira~\cite{DGM22}, on the existence of a certain kind of ``pseudoinverse''.

\begin{lemma}\label{lem:pseudoinverse}
Consider a symmetric matrix $A\in \mb R^{n\times n}$, let $S=\supp(\ker(A))$ and let $\overline S=\{1,\ldots,n\}\setminus S$. Let $P\in \mb R^{n\times n}$ be the projection matrix that projects onto the coordinates indexed by $\overline S$ (that is, $P$ is a diagonal matrix with ``$1$'' in the diagonal entries indexed by $\overline S$, and ``$0$'' in the entries indexed by $S$). Then there is a matrix $B\in \mb R^{n\times n}$ such that $AB=P$. Further, $PBP$ is symmetric.
\end{lemma}
\begin{proof}
If $i\notin S$, then there is no kernel vector of $A$ with nonzero $i$-coordinate, which means that the $i$th row of $A$ cannot be expressed as a linear combination of the other rows of $A$. This means there is some vector $\mbf b_i$ which is orthogonal to every row of $A$ except the $i$th (and by rescaling we can assume that the inner product of $\mbf b_i$ with the $i$th row of $A$ is exactly 1). 
Let $B$ be the matrix whose $i$th column is $\mbf b_i$, for $i\notin S$, and whose columns indexed by $S$ are all-zero. Writing $\mbf e_i$ for the $i$th standard basis vector, we have $\mbf e_i^\intercal AB=\mbf e_i$ for $i\notin S$, and $\mbf e_i^\intercal AB=\mbf 0$ for $i\in S$, from which it follows that $AB=P$.

Finally, since $A$ is symmetric, transposition yields $B^\intercal A=P$. We then have $PB = (B^\intercal A)B = B^\intercal (AB) = B^\intercal P=(PB)^\intercal $, so $PB=PBP$ is symmetric.
\end{proof}

A key property of this notion of pseudoinverse is that it provides a sufficient condition for the rank to increase when we add a row and column to a matrix, as follows.

\begin{lemma}\label{lem:quadratic-rank-increase}
Consider a matrix $A\in \mb R^{n\times n}$, and let $S=\supp(\ker(A))$, let $S=\supp(\ker(A))$, and let $B$ be a ``pseudoinverse'' as guaranteed in \cref{lem:pseudoinverse}. For $\mbf x,\mbf y\in \mb R^n$ and $c\in \mb R$, let $A'$ be the matrix obtained from $A$ by appending the new column $\mbf x$, the new row $\mbf y$ and the new diagonal entry $c$. If $\mbf x_S=\mbf 0$ and $\mbf y^\intercal B\mbf x\ne c$ then $\rank A'\ge \rank A+1$.
\end{lemma}
\begin{proof}
Let $A''$ be the matrix obtained by appending the column $\mbf x$ (but not appending a new row), and let $\mbf y'$ be the last row of $A'$ (obtained by appending the entry ``$c$'' to the end of the vector $\mbf y$). Also, let $\mbf w\in \mb R^{n+1}$ be the vector obtained by appending the entry ``$-1$'' to the end of the vector $B\mbf x$.

Since $\mbf x_S=\mbf 0$, we have $AB\mbf x=\mbf x$ by the defining property of $B$, which implies that $\mbf w$ is a kernel vector of $A''$ (i.e., it lies in the orthogonal complement of the row space of $A''$). So, if $\mbf w\cdot \mbf y'\ne 0$, then $\mbf y'$ does not lie in the row space of $A''$, meaning that $\rank A'\ge\rank A''+1\ge \rank A+1$. The desired result follows, noting that $\mbf w\cdot \mbf y'=\mbf x^\intercal B^\intercal \mbf y-c$.
\end{proof}
Now we are ready to prove \cref{lem:corank-boosting-master}(b).

\begin{proof}[Proof of \cref{lem:corank-boosting-master}(b)]
Let $B,P$ be as in \cref{lem:pseudoinverse}, and consider the symmetrisation $B' = (B + B^\intercal )/2$ of $B$. Since $PBP$ is symmetric, $B$ and $B'$ have the same entries in positions indexed by $\ol S\times \ol S$. Recall that $\mbf x_S = 0$, so $\mbf x^\intercal B\mbf x=\mbf x^\intercal B'\mbf x$.

We claim that there are at least $\eta^2 n^4/4$ different tuples $(i,i',j,j')$ with $B_{i,j}-B_{i,j'}+B_{i',j}-B_{i',j'}\neq 0$. Since the symmetrisation of $B$ to $B'$ only affects at most $4|S|n^3\le \eta^2 n^4/8$ of these tuples, it will follow from \cref{prop:quadratic-LO-general} that $\mbf x^\intercal B'\mbf x=\mbf x^\intercal B'\mbf x\ne 0$ with probability at least $1-O((\log 2d)^{O_{\eta}(1)}/\sqrt{d})$, in which case $\rank A'\ge \rank A+1$ by \cref{lem:quadratic-rank-increase}, as desired.

To prove the claim, let $\mbf b_i$ be the $i$th row of $B$ and let $\mbf e_i$ be the $i$th standard unit vector. For $i,i'\in \ol S$, let $\mbf w_{i,i'}=\mbf b_i-\mbf b_{i'}$, so $A\mbf w_{i,i'}=\mbf e_i - \mbf e_{i'}$. Recall that $A$ being $\eta$-unstructured means that there are $\eta n^2$ pairs of indices $(i,i')$ for which all $\mbf w$ satisfying $\supp(A\mbf w)=\{i,i'\}$ are $\eta$-balanced. So, there are at least $\eta n^2-2|S|n\ge \eta n^2/2$ pairs $(i,i')\in \ol S^2$ for which $\mbf w_{i,i'}$ is $\eta$-balanced. For each such $(i,i')$, there are at least $\eta(1-\eta)n^2\ge \eta n^2/2$ pairs $(j,j')$ for which the $j$th and $j'$th entry of $\mbf w_{i,i'}$ differ, in which case $B_{i,j}-B_{i,j'}+B_{i',j}-B_{i',j'}\neq 0$.
\end{proof}

\section{Extracting high-degree vertices}\label{sec:corank-init}
As outlined, for the proof of \cref{thm:main-RMT} we need to ``extract'' high degree vertices from our random graph with minimum degree at least 2, without revealing too much about the neighbourhoods of the extracted vertices. We will need certain information about the graph that remains after this extraction; most notably we need control over its corank, and we need to know that most of its vertices still have degree at least 2.

\begin{lemma}\label{lem:initial-rank}
Fix $\varepsilon,\alpha,\Delta > 0$ such that $1/\Delta\ll \alpha\ll \varepsilon$.
\begin{enumerate}
    \item[(A)] Consider sets $S\subseteq V$ and an integer $m$ such that
    \begin{itemize}
        \item $|V|=n$,
        \item $\varepsilon n\le m-n\le n/\varepsilon$, and
        \item $|S|=\left\lfloor\vphantom{x^{x^x}}\alpha n\right\rfloor$.
    \end{itemize}
    Let $G\sim \mc K(V,m,2)$, and let $T=\{v\in S\colon\deg_G(v)\ge \Delta\}$ be the set of vertices in $S$ with degree at least $\Delta$. Then, whp:
    \begin{enumerate}
        \item[(1)] $\corank A(G[V\setminus T])\le |T|/8$,
        \item[(2)] Let $W$ be the set of vertices in $G[V\setminus T]$ which have degree at least 2 (with respect to $G[V\setminus T]$). Then $|(V\setminus T)\setminus W|=O_{\varepsilon}(\alpha n)$.
        \item[(3)] $|T|\le \alpha \exp(-\Omega_{\varepsilon}(\Delta)) n$,
        \item[(4)] Whp all but $|T|/\Delta$ vertices $v\in T$ satisfy $\deg_{W}(v)\ge\sqrt \Delta$.
    \end{enumerate}
    \item[(B)] Consider disjoint sets $S_1\subseteq V_1$ and $S_2\subseteq V_2$, and an integer $m$, such that
    \begin{itemize}
        \item $|V_1|,|V_2|=n+o(n)$,
        \item $\varepsilon n\le m-|V_1|-|V_2|\le n/\varepsilon$, and
        \item $|S_1|=|S_2|=\left\lfloor\vphantom{x^{x^x}}\alpha n\right\rfloor$.
    \end{itemize}
    Let $G\sim \mc K(V_1,V_2,m,2)$, and for $i\in \{1,2\}$ let $T_i=\{v\in S_i\colon\deg_G(v)\ge \Delta\}$ be the set of vertices in $S_i$ with degree at least $\Delta$. Let $V=V_1\cup V_2$ and $T=T_1\cup T_2$. Then, whp:
    \begin{enumerate}
        \item[(1)] $\corank B(G[V\setminus T])\le |T|/16$,
        \item[(2)] Let $W$ be the set of vertices in $G[V\setminus T]$ which have degree at least 2 (with respect to $G[V\setminus T]$). Then $|(V\setminus T)\setminus W|=O_{\varepsilon}(\alpha n)$.
        \item[(3)] $|T_1|=|T_2|+o_{\varepsilon,\alpha,\Delta}(n)\le \alpha \exp(-\Omega_{\varepsilon}(\Delta))n$,
        \item[(4)] Whp all but $|T|/\Delta$ vertices $v\in T$ satisfy $\deg_{W}(v)\ge\sqrt \Delta$.
    \end{enumerate}
\end{enumerate}
\end{lemma}

All parts of \cref{lem:initial-rank} follow from similar calculations to those that were performed in \cite{FKSS21}. In particular, in the setting of (A), \cite[Lemma~6.6]{FKSS21} gives asymptotic formulas (in terms of $m/n$, $\alpha$ and $\Delta$) for $|T|$ and the number of vertices in $G[V\setminus T]$ with a given degree\footnote{The lemmas in \cite[Section~6]{FKSS21} are stated with an assumption $k\ge 3$, but this is completely unnecessary.}, from which (A2) and (A3) follow. Then, (A4) may be proved with a simple configuration-model calculation; such a calculation appears in the proof\footnote{Again, \cite[Lemma~8.1]{FKSS21} is stated with an assumption $k\ge 3$, but this is unnecessary.} of \cite[Lemma 8.1(3)]{FKSS21} (recalling $E$ from the notation of that lemma, we have $E\subseteq W$ when $k=2$.

Essentially the same calculations can be performed in the bipartite setting to prove (B2--4) (in fact, the relevant asymptotic formulas are nearly identical, though one needs to consider separate parameters for each side of our random bipartite graph).

(A1) and (B1) warrant a bit more explanation. \cref{lem:initial-rank}(A1) is basically the same as \cite[Lemma~7.3]{FKSS21} (which concerned random graphs constrained to have minimum degree at least $k\ge 3$, while we need to handle random graphs constrained to have minimum degree at least 2). Roughly speaking, the idea is to show that our random graph $G$ locally weakly converges (in the sense of Aldous--Steele~\cite{AS04} and Benjamini--Schramm~\cite{BS01}) to a certain Galton--Watson tree. Spectral convergence machinery of Bordenave, Lelarge, and Salez~\cite{BLS11} then can be used to bound the corank of $A(G)$ in terms of a certain probability generating function associated with that Galton--Watson tree. The proof of \cite[Lemma~7.3]{FKSS21} does use the assumption $k\ge 3$, but it was written in a slightly inefficient way; it is possible to slightly modify the proof to overcome this assumption (as we sketch momentarily).

It turns out that essentially the same proof can also be used for \cref{lem:initial-rank}(B1), because local weak convergence does not ``see'' whether a graph is bipartite or not.

\begin{proof}[Proof sketch of \cref{lem:initial-rank}(A1) and (B1)]
First, for the reader's convenience, we outline the proof of \cite[Lemma~7.3]{FKSS21}. We then discuss the minor changes that are necessary to prove \cref{lem:initial-rank}(A1).

We may assume that $2m/|V|$ converges to some $g\in[2+2\eps,2/\eps]$. Indeed, if the desired property did not hold whp, then for some $\tau>0$ there would be an infinite sequence of integers $n$ (and accompanying sets $V^{(n)}$) along which for each $n$ the corresponding property fails to hold with probability at least $\tau$. By compactness there would then be an infinite violating subsequence along which $2m/|V^{(n)}|$ converges to a limit.

As proved in \cite[Lemma~6.10]{FKSS21}, 
$G$ locally converges to a Galton--Watson tree with a certain offspring distribution $\mu$, and by \cite[Lemma~6.6]{FKSS21}, whp $|T|=\beta n+o(n)$ for some explicit $\beta>0$ depending on $g$ (and the other parameters). As discussed in \cite[Section~7]{FKSS21}, by results of Bordenave, Lelarge, and Salez (specifically \cite[Theorem~13 and Eq.~(19)]{BLS11}), we have $\corank A(G)\le \max_{x\in [0,1]}M_\mu(x)+o(n)$ for a certain function $M_\mu\colon[0,1]\to \mb R$ depending on $\mu$, so it suffices to prove that $\max_{x\in [0,1]}M_\mu(x)\le\beta/16$. This is essentially what is proved in \cite[Lemma~7.5]{FKSS21}, but there was one point where the assumption $k\ge 3$ was used: namely, for certain $\gamma,\lambda>0$ (which depend only on $g$ and $\alpha$) and a function $\phi\colon[0,(1-\gamma)\lambda]\to \mb R$ defined by
\[\phi(x) = \sum_{t=k-1}^\infty\frac{x^t}{t!}((\mbm{1}_{t\ge k}-\alpha\mbm{1}_{t\ge\Delta}) + \gamma\lambda(\mbm{1}_{t+1\ge k}-\alpha\mbm{1}_{t+1\ge\Delta})),\]
it is necessary to prove that $\phi''$ is log-concave, and the proof in \cite[Lemma~7.5]{FKSS21} uses the assumption $k\ge 3$. We give an alternative proof for the log-concavity of $\phi''$ in the case $k=2$, as follows.

In the case $k=2$, one can compute
\[\phi''(x) = \sum_{t=0}^\infty\frac{x^t}{t!}((1 + \gamma\lambda)-\alpha(\mbm{1}_{t\ge\Delta-2} + \gamma\lambda\mbm{1}_{t\ge\Delta-3})) = (1+\gamma \lambda) (e^{x} - h(x)),\]
where 
\[h(x) = \sum_{t=0}^\infty\frac{x^t}{t!}(1+\gamma \lambda) ^{-1}\alpha(\mbm{1}_{t\ge\Delta-2} + \gamma\lambda\mbm{1}_{t\ge\Delta-3}).\]
To verify that $\phi''(x)$ is log-concave it suffices to verify that 
\[\frac{d^2}{dx^2}\log(e^{x}-h(x)) = -\frac{e^x(h(x)+h''(x)-2h'(x)) + (h'(x)^2-h(x)h''(x))}{(e^x-h(x))^2}\]
is non-positive. First we compute
\begin{align*}
h(x)+h''(x)-2h'(x) &= \sum_{t=0}^\infty\bigg(\frac{x^t}{t!}-\frac{2x^{t-1}}{(t-1)!}+\frac{x^{t-2}}{(t-2)!}\bigg) (1+\gamma \lambda)^{-1}\alpha(\mbm{1}_{t\ge\Delta-2} + \gamma\lambda\mbm{1}_{t\ge\Delta-3})\\
&= \sum_{t=0}^\infty\frac{x^{t-2}(x^2-2tx+t(t-1))}{t!}(1+\gamma \lambda)^{-1}\alpha(\mbm{1}_{t\ge\Delta-2} + \gamma\lambda\mbm{1}_{t\ge\Delta-3}).
\end{align*}
For sufficiently large $\Delta$ in terms of $\lambda$ and $\gamma$, and $x\in [0,(1-\gamma)\lambda]$, each term in this expression is nonnegative, because $x^2-2tx+t(t-1)\ge 0$ for $x\le t-\sqrt{t}$. So, it suffices to verify that $h'(x)^2-h(x)h''(x)\ge 0$, which is equivalent to $h$ being log-concave.

To prove that $h$ is log-concave, we use the well-known fact (see  for example \cite[Lemma~3]{BB05}) that if a nonnegative function $f$ is log-concave on an interval $[a,b]$, then its antiderivative $x\mapsto \int_a^xf(y)dy$ is also log-concave on that same interval. Given this, it suffices to prove that the $(\Delta-3)$-fold derivative $h^{(\Delta-3)}$ is log-concave. We compute
\[h^{(\Delta-3)}(x) = (1+\gamma\lambda)^{-1}\alpha((1+\gamma\lambda)e^{x}-\gamma\lambda)),\]
so we may now finish the proof by direct differentiation. Indeed, for any $r\in [0,1]$ we compute $\frac{d^2}{dx^2}\log(e^{x}-r) = -re^{x}/(e^{x}-r)^2\le 0$, which implies the desired result.

For (B1), essentially the same proofs as in \cite[Section~6]{FKSS21} show that $G$ still locally converges to a Galton--Watson tree with offspring distribution $\mu$, and whp $|T|=\beta(|V_1|+|V_2|)+o(n)$, so the same proof as above shows that $\corank A(G[V\setminus T])\le |T|/8$. So, by \cref{fact:biadjacency-rank} we have $\corank B(G[V\setminus T])\le |T|/16$.
\end{proof}

\section{Analysing a corank-walk}\label{sec:RMT-overview}
In this section we explain how to prove \cref{thm:main-RMT}(A1) using the tools from
\cref{sec:corank-init,sec:boost-corank,sec:random-walk}. The proof of \cref{thm:main-RMT}(B1) is very similar (actually, it is slightly easier), and we briefly sketch the necessary changes for that proof in \cref{sec:bipartite-sketch}.

Fix $\varepsilon>0$, let $1+\varepsilon\le m/n\le n/\varepsilon$, let $G\sim\mathcal{K}(n,m,2)$ and write $V=\{1,\ldots,n\}$ for its vertex set. Our objective is to prove that whp
$\corank A(G)=s(G)$.

Fix $\alpha,\eta,\Delta>0$ such that $1/\Delta\ll\alpha\ll \eta\ll\varepsilon$. At the end of the proof we will take $\Delta\to \infty$ as $n\to \infty$,
but for now we view $\Delta$ as a constant. (In particular, we assume $n$ is large in terms of $\alpha,\eta,\Delta$.)

As in \cref{lem:initial-rank}, consider a set $S$ of
$\lfloor\alpha n\rfloor$ vertices (say $S=\{1,\ldots,\lfloor\alpha n\rfloor\}$),
and let $T=\{v\in S\colon\deg_{G}(v)\ge\Delta\}$ be the set of vertices
with degree at least $\Delta$. When we take $\Delta\to \infty$ at the end of the proof, we will have $|T|=o(n)$ whp, but until then the reader should think of $|T|$ as having order $n$.

Given the information in \cref{lem:initial-rank} about $G[V\setminus T]$, our strategy is to study the evolution of the corank as we add back the vertices in $T$ (in a random order), using \cref{thm:random-walk-master,lem:corank-boosting-master}. Let $N=|T|$, consider a uniform random ordering $v_{N-1},\ldots,v_{0}$ of the elements of $T$, and let $G_{t}=G[V\setminus \{v_{t-1},\ldots,v_{0}\}]$
(so $G_{N}=G[V\setminus T]$ and $G_{0}=G$).

Let $W$ be the set of vertices in $G_{N}$ which have degree at
least 2 (with respect to $G_{N}$). The idea is that the vertices
in $W$ already satisfy their degree constraints, so all the vertices
in $W$ are equally likely to be neighbours of vertices in $T$. The
following lemma makes this precise.
\begin{claim}
\label{claim:neighbourhood-symmetry}Reveal an outcome of $T$, reveal
all the edges of $G$ not between $T$ and $W$, and reveal $\deg_{W}(v)$
for each $v\in T$. Then, conditionally, the neighbourhoods $(N_{W}(v))_{v\in T}$
are independent uniformly random subsets of $W$ with sizes $(\deg_{W}(v))_{v\in T}$.
\end{claim}
\begin{proof}
Let $G[T,W]$ be the bipartite graph of edges between $T$ and $W$. Let $H_1,H_2$ be bipartite graphs with the same bipartition $X\cup Y$, such that every vertex in $X$ has the same degree in $H_1$ as it does in $H_2$. Then, for any outcome of $G$ such that $G[T,W]=H_1$, we can swap $G[T,W]$ with $H_2$ to obtain an outcome of $G$ such that $G[T,W]=H_2$. So, $H_1$ and $H_2$ are equally likely to occur as $G[T,W]$. (It is important that this swap can never change the sets $W$ or $T$, and can never cause the degree of any vertex to drop below $2$.)
\end{proof}

Recall that, by \cref{lem:initial-rank}(A4), almost all vertices $v_1,\ldots,v_N$ have degree at least $\sqrt \Delta$ into $W$. So, 
the upshot of \cref{claim:neighbourhood-symmetry} is that when we add the vertices $v_{1},\ldots,v_{N}$
back to $G_{N}$, at most of these steps we are essentially adding
a new random row and column with many ``1''-entries, which puts
us in a position to apply \cref{lem:corank-boosting-master}.

Recall that \cref{lem:corank-boosting-master}(a) has an assumption that the matrix under consideration
has a balanced kernel vector: a kernel vector which is not dominated
by a single level set. We therefore need some estimates about kernel
vectors of the adjacency matrices $A(G_{t})$.

First, say that a level set of a vector $\mbf v$ is a \emph{nonzero level set} if it is the $\lambda$-level set of $\mbf v$ for some $\lambda\ne 0$. With some crude estimates it is not hard to show that there
are unlikely to ever be kernel vectors which are dominated by a nonzero
level set.
\begin{claim}
\label{claim:nonzero-level-sets}With probability $1-n^{-\omega(1)}$,
no $A(G_{t})$ has a kernel vector with a nonzero level set larger than $(1-\eta)(n-t)$ (provided $\eta\ll \eps$ and $\alpha\le \eta$).
\end{claim}

We defer the simple proof of \cref{claim:nonzero-level-sets} to \cref{sec:stalks-computation}. It is much more delicate to
deal with kernel vectors which are dominated by their zero level set,
i.e., kernel vectors with small \emph{support}. Indeed, the special
cycles counted by $s(G)$ each give rise to a kernel vector with small
support, so we certainly cannot rule these out entirely. For each $t\le N$, define the set of ``small-support'' kernel vectors
\[K^{(\eta)}_t=\{\mathbf{v}\in\ker A(G_{t})\colon|\on{supp}(\mathbf{v})|\le\eta (n-t)\}.\]
The following lemma shows that while there may be some vectors in $K^{(\eta)}_t$, typically these vectors are collectively supported on a small subset of indices.
\begin{claim}
\label{claim:small-kernel-vectors-rough}
For each $t$,
\[
\Pr[t\le N\emph{ and }\dim(K^{(\eta)}_t)\ge t/4]\lesssim_{\eta} \frac{1}{t}+\left(\frac{t}{n}\right)^{1/4}.
\]
\end{claim}
\begin{remark*}
Note that $N$ is random, and we are not conditioning on it at this stage.
\end{remark*}

We also need much more precise control for the last few steps of our random walk, essentially characterising each $K^{(\eta)}_t$ in terms of the special cycles of $G_t$. For an $h$-vertex graph $H$, let $s^{(\eta)}(H)$ be the number of special cycles of length at most $2\eta h$ in $H$, counting isolated special cycles twice. Let $C^{\mr{spec}(\eta)}_t$ be the set of special cycles of length at most $2\eta n$ in $G_t$ and let $V^{\mr{spec}(\eta)}_t$ be the set of degree-2 vertices in these special cycles.

\begin{claim}\label{claim:small-kernel-vectors-precise}
The following hold together with probability at least $1-1/\Delta$.
\begin{enumerate}
    \item[(a)] $\supp(K^{(\eta)}_t)=V_t^{\mr{spec}(\eta)}$ for each $t\le \Delta$.
    \item[(b)] $\dim(K^{(\eta)}_t)=s^{(\eta)}(G_t)$ for each $t\le \Delta$.
    \item[(c)] $(C^{\mr{spec}(\eta)}_\Delta,V^{\mr{spec}(\eta)}_\Delta)=\cdots=(C^{\mr{spec}(\eta)}_0,V^{\mr{spec}(\eta)}_0)$.
    \item[(d)] $|V^{\mr{spec}(\eta)}_t|\le \eta^2 n/50$ for each $t\le \Delta$.
    \item [(e)] $v_t$ has no neighbor in $V_{t+1}^{\mr{spec}(\eta)}$, for each $t< \Delta$.
\end{enumerate}
\end{claim}
(At the end of the proof we will take $\Delta\to \infty$, meaning that \cref{claim:small-kernel-vectors-precise} will become a with-high-probability statement.)

The proofs of \cref{claim:small-kernel-vectors-rough,claim:small-kernel-vectors-precise} are very delicate; they proceed by considering a linear-algebraic notion of \emph{minimal kernel vectors}, and studying the combinatorial consequences of this notion. We defer the proofs to \cref{sec:stalks-computation}.

For the last few steps of our random walk (for $t<\Delta$) we need to use \cref{lem:corank-boosting-master}(b), so we also need to know that $A(G_{t})$ is likely to be $\eta$-unstructured for such $t$,
in the sense of \cref{def:unstructured}.
\begin{claim}
\label{claim:unstructured-whp}For $t\in\{0,\ldots,\Delta\}$,
$A(G_{t})$ is $\eta$-unstructured with probability $1-(\log n)^{-\omega(1)}$.
\end{claim}

It turns out that \cref{claim:unstructured-whp} can be proved within the same general framework as \cref{claim:small-kernel-vectors-rough,claim:small-kernel-vectors-precise}. We defer this proof of \cref{claim:unstructured-whp} to \cref{sec:stalks-computation}.

We are now ready to define our random walk $X_{N},\ldots,X_{0}$. Say an index $t\le N$
is \emph{good} if all of the following hold.
\begin{enumerate}[{\bfseries{G\arabic{enumi}}}]
    \item\label{G1} $\deg_{W}(v_{t})\ge \sqrt \Delta$, and
    \item\label{G2} if $t\ge\Delta/2$, then $\dim(K_{t+1}^{(\eta)})< t/4$, and
    \item\label{G3} $A(G_{t+1})$ has no kernel
vector with a nonzero level set larger than $(1-\eta)(n-t-1)$, and
    \item\label{G4} if $t< \Delta$, then $v_{t}$ has no neighbour in $\supp(K_{t+1}^{(\eta)})$, and
    \item\label{G5} if $t< \Delta$, then $\dim(K_{t+1}^{(\eta)})= \dim(K_{\Delta}^{(\eta)})$, and
    \item\label{G6} if $t< \Delta$, then $A(G_{t+1})$ is $\eta$-unstructured, and
    \item\label{G7} if $t< \Delta$, then $|\supp(K_{t+1}^{(\eta)})| < \eta^2 (n-t-1)/32$.
\end{enumerate}
Let $R$ be the set of indices that are not good (i.e., bad). Then, we have the following consequence of \cref{claim:nonzero-level-sets,claim:small-kernel-vectors-rough,claim:small-kernel-vectors-precise,claim:unstructured-whp} (and the properties in \cref{lem:initial-rank}).
\begin{claim}\label{claim:bad-steps-tail}
With probability $1-o_{\Delta\to \infty}(1)$, for all $t\le N$ we have $|R\cap \{1,\ldots,t\}| \le t/100$.
\end{claim}
\begin{remark*}
Here we use notation of the form $f=o_{\Delta\to\infty}(n)$ to mean that $f/n$ can be made arbitrarily small by taking sufficiently large $\Delta$. Recall that for the moment we are viewing $\Delta$ as a constant, but later on we will take $\Delta\to\infty$ as $n\to\infty$.
\end{remark*}
\begin{proof}
By \cref{lem:initial-rank}(A3), whp $N\le \beta n$ for some $\beta=o_{\Delta\to \infty}(n)$. For each $i=0,1,\ldots,\lfloor\log_2 (\beta n+2)\rfloor$, let $Q_i$ be the number of bad steps $t\in R$ with $t\le \beta n$ in the range $[2^i-1,2^{i+1}-1)$. It suffices to show that with probability $1-o_{\Delta\to \infty}(1)$, for each $i$ we have $Q_i\le 2^i/800$.

Let $Q_i^{\text{\cref{G1}}},\ldots,Q_i^{\text{\cref{G7}}}$ be the contribution to $Q_i$ from failure of each of \cref{G1} to \cref{G7} (so $Q_i\le Q_i^{\text{\cref{G1}}}+\cdots+Q_i^{\text{\cref{G7}}}$). We will show that with probability $1-o_{\Delta\to \infty}(1)$ we have $Q_i^{j}\le 2^i/5600$ for each $j\in \{\text{\cref{G1}},\ldots,\text{\cref{G7}}\}$.

First, the cases $j\in\{\text{\cref{G3}}, \ldots ,\text{\cref{G7}}\}$ are easy to handle with \cref{claim:nonzero-level-sets,claim:unstructured-whp,claim:small-kernel-vectors-precise} and the union bound. For $j=\text{\cref{G1}}$, say that a vertex $v\in T$ is ``degree-bad'' if $\deg_{W}(v)<\sqrt \Delta$. By \cref{lem:initial-rank}(A4), whp the fraction of degree-bad vertices in $T$ is at most $1/\Delta$. We can condition on such an outcome of these degree-bad vertices without revealing any information about the ordering $v_1,\ldots,v_N$ of the vertices in $T$. So, by a Chernoff bound for the hypergeometric distribution (\cref{lem:chernoff}), we have $\Pr[Q_i^{\text{\cref{G1}}}> 2^i/5600]\le e^{-\Omega(2^i)}$, and by Markov's inequality we have $\Pr[Q_i^{\text{\cref{G1}}}> 2^i/5600]\le O(1/\Delta)$. Using the former inequality for say $i\ge \log \Delta$ and the latter inequality for $i<\log \Delta$, the desired result follows by a union bound.

It remains to consider the case $j=\text{\cref{G2}}$. By Markov's inequality and \cref{claim:small-kernel-vectors-rough}, we have
\begin{align*}
\sum_{i=0}^{\lfloor\log_2 (\beta n+2)\rfloor} \Pr[Q_i^{\text{\cref{G2}}}> \frac{2^i}{5600}]\le \sum_{i=0}^{\lfloor\log_2 (\beta n+2)\rfloor} \frac{\mb E Q_i^{\text{\cref{G2}}}}{\Omega(2^i)}&=\sum_{i=0}^\infty\;\sum_{\substack{t\in[2^i-1,2^{i+1}-1)\\\Delta/2\le t\le \beta n}}\!\!\frac{\Pr[t\le N\text{ and }\dim(K_{t+1}^{(\eta)})\ge t/4]}{\Omega(2^i)}\\
&\lesssim_{\eta} \sum_{t=\lceil\Delta/2\rceil}^{\beta n} \frac{1/t+(t/n)^{1/4}}{t} \lesssim_{\eta} \frac{1}{\Delta}+\beta^{1/4}= o_{\Delta\to \infty}(1).\qedhere
\end{align*}
\end{proof}
Now, we are ready to complete the proof of \cref{thm:main-RMT}(A1) (using the notation and claims from throughout this section).
\begin{proof}[Proof of \cref{thm:main-RMT}(A1)]
Let $\mc I=(G_N,(\deg_W(v))_{v\in T})$. That is to say, $\mc I$ specifies $G_N$ (which determines $T$ and $W$), and the degrees from $T$ into $W$. Let $\mc E$ be the event that $|R\cap \{1,\ldots,t\}|> t/100$ for some $t\le N$, so $\Pr[\mc E]\le h(\Delta)$ fir sine $h$ satisfying $h(\Delta)\to0$ as $\Delta\to\infty$ by \cref{claim:bad-steps-tail}.

Note that $\Pr [\mc E]=\mb E[\Pr[\mc E\,|\,\mc I]]$, so applying Markov's inequality to $\Pr[\mc E\,|\,\mc I]$, and applying \cref{lem:initial-rank} for $\alpha\ll \eps$, we see that $\mc I$ satisfies
\begin{enumerate}
    \item[(i)] $\corank A(G_N)\le|T|/8$,
    \item[(ii)] $G[V_N]$ has at most $\eta n/10$ vertices with degree less than 2,
    \item[(iii)] $\Pr[\mc E\,|\,\mc I]\ge 1-h(\Delta)^{1/2}$.
\end{enumerate}
with probability at least $1-2h(\Delta)^{1/2}$. For the rest of the proof, we condition on such an outcome of $\mc I$ (so, for example, we treat $T$, $N$, and $G_N$ as deterministic objects).

Now, for $t\le N$, let
\[
X_{t}=\corank A(G_{t})-\mbm 1_{t< \Delta} \dim{K_\Delta^{(\eta)}}-t/4.
\]
We claim that the sequence $X_N,\ldots,X_0$ and the ``bad set'' $R$ satisfy the conditions of \cref{thm:random-walk-master} (with $\varepsilon=1/2$, $C=5/4$, $\delta=3/4$ and $p=o_{\Delta\to \infty}(1)$). Conditions \cref{W0} and \cref{W2} are immediate, condition \cref{W1} follows from (i) above, and condition \cref{W4} follows from (iii). So, we just need to verify \cref{W3}.

To this end, condition on any outcome of $G_{t+1}$ (which determines $X_{t+1}$). We will study how $X_t$ differs from $X_{t+1}$, in this conditional probability space.

In addition to information revealed so far, reveal $v_t$ and its degree $\deg_W(v_t)$ into $W$. If $t< \Delta$, also reveal the neighbourhood of $v_t$ in $\supp(K_{t+1}^{(\eta)})$ (this is enough information to see whether step $t$ is bad). Condition on an outcome of the revealed information; we need to show that if step $t$ is not bad, then with probability at least $1-o_{\Delta\to \infty}(1)$: if $X_{t+1}> 0$ then $X_{t}\le X_{t+1}-3/4$, and if $X_{t+1}\le 0$ then $X_t\le 0$. We assume that the revealed information is such that step $t$ is not bad (otherwise there is nothing to prove).

Let $d=\deg_W(v_t)\ge \sqrt \Delta$ (by \cref{G1}). If $t\ge \Delta$ let $E=W$ and if $t< \Delta$ let $E=W\setminus \supp(K_{t+1}^{(\eta)})$. In either case, we have $E\ge (1-\alpha)n-\eta n/10-\Delta\ge (1-\eta/3)n$ (using (ii) above and \cref{G2,G5}). By \cref{claim:neighbourhood-symmetry} and \cref{G4}, the neighbourhood $N_E(\mbf v_t)$ of $v_t$ in $E$ is a uniformly random size-$d$ subset of $E$. Now, we use one of the two parts of \cref{lem:corank-boosting-master}, as follows.
\begin{itemize}
    \item \emph{Case 1: $\corank(A(G_{t+1}))>\dim(K_{t+1}^{(\eta)})$.} In this case, there is some kernel vector $\mbf v$ of $A(G_{t+1})$ with $|\supp(\mbf v)|\ge \eta (n-t-1)$. By \cref{G3}, this kernel vector is $\eta$-balanced, so \cref{lem:corank-boosting-master}(a) yields that $\rank(A(G_t))\ge \rank(A(G_{t+1}))+2$ (and hence $X_t\le X_{t+1}-3/4$) with probability at least $1-O_\eta(\Delta^{-1/4})\ge 1-\Delta^{-1/8}$.
    \item \emph{Case 2: $\corank(A(G_{t+1}))=\dim(K_{t+1}^{(\eta)})$.} By \cref{G2} and \cref{G5}, this case can only happen if $X_{t+1} < 0$. Thus if $\corank(A(G_t))=\corank(A(G_{t+1}))$, we have $X_t \leq 0$ since the walk is quarter-integral and increases by at most $1/4$. By \cref{G6} and \cref{G7}, the conditions for \cref{lem:corank-boosting-master}(b) are met, and thus with probability at least $1-\Delta^{-1/8}$, we have $\corank(A(G_t))=\corank(A(G_{t+1}))$. The desired result follows.
\end{itemize}

Now, having verified conditions \cref{W0,W1,W2,W3,W4}, the conclusion of \cref{thm:random-walk-master} is that $X_0=\corank A(G_{0})-\dim{K_\Delta^{(\eta)}}\le 0$ with probability at least $1-O(\min(h(\Delta)^{1/2},\Delta^{-1/8}))$. By \cref{claim:small-kernel-vectors-precise}, it follows that $\corank A(G_{0})=\dim{K_0^{(\eta)}}=s^{(\eta)}(G)$ with probability at least $1-O(\min(h(\Delta)^{1/2},\Delta^{-1/8}))$. We deduce that $\corank A(G_{0})=\dim{K_0^{(\eta)}}=s^{(\eta)}(G)$ whp, taking $\Delta\to\infty$ sufficiently slowly. It now just suffices to observe that when $\corank A(G_{0})=\dim{K_0^{(\eta)}}$ there are no special cycles longer than $2\eta n$. Indeed, such a special cycle would give rise to a kernel vector with support larger than $\eta n$, by \cref{fact:special-kernel}.
\end{proof}

\section{Kernel vectors and stalks}\label{sec:stalks-computation}
In this section we prove \cref{claim:unstructured-whp,claim:nonzero-level-sets,claim:small-kernel-vectors-rough,claim:small-kernel-vectors-precise}, which are the remaining ingredients in our proof of \cref{thm:main-RMT}(A1).

\cref{claim:nonzero-level-sets,claim:small-kernel-vectors-rough,claim:small-kernel-vectors-precise} concern kernel vectors; note that \cref{claim:unstructured-whp} can be interpreted as a claim about \emph{almost} kernel vectors. Indeed, say that a vector $\mbf v\in \mb R^{n}$ is an $\ell$-almost kernel vector of a matrix $A\in \mb R^{n\times n}$ if $|\on{supp} (A\mbf v)|=\ell$ (so a kernel vector is a $0$-almost kernel vector, and the definition of $\eta$-unstructuredness in \cref{def:unstructured} concerns 2-almost kernel vectors). We will therefore be able to prove each of \cref{claim:unstructured-whp,claim:nonzero-level-sets,claim:small-kernel-vectors-rough,claim:small-kernel-vectors-precise} by carefully studying almost-kernel vectors in degree-constrained random graphs.

First, it is easy to show that kernel vectors which are dominated by a nonzero level set are unlikely: the following lemma immediately implies \cref{claim:nonzero-level-sets}, and is a simple consequence of \cref{lem:degree-sequence-poisson}(A4).

\begin{lemma}\label{lem:nonzero-level-set}
Let $G=G_0,\ldots,G_N$ be as in \cref{sec:RMT-overview}. Suppose $\eta\ll \eps$ and $\alpha\le \eta$. Then with probability $1-n^{-\omega(1)}$ no $A(G_t)$ has an $\ell$-almost kernel vector with a nonzero level set larger than $(1-\eta)(n-t)$, for any $\ell\in\{0,1,2\}$.
\end{lemma}
\begin{proof}
We show that the desired result follows whenever $G$ satisfies the conclusion of \cref{lem:degree-sequence-poisson}(A4).

Consider any vector $\mbf x\in \mb R^{V(G_t)}$ with $\lambda$-level set $U\subseteq V(G_t)\subseteq V(G)$ larger than $(1-\eta)(n-t)\ge (1-\eta)(n-\alpha n)\ge (1-2\eta)n$, for some $\lambda\ne 0$, and suppose without loss of generality that $\lambda>0$. We will show that $\mbf x$ cannot be an $\ell$-almost kernel vector of $A(G_t)$, for any $\ell\le 2$.

Let $\ol U$ be the complement of $U$ in $V(G)$, and recall that we defined $G=G_0,\ldots,G_N$ by deleting some vertices from a special vertex subset $S$ with $|S|=\lfloor\alpha n\rfloor$. Note that $\ol U\cup S$ has at most $O_\eps((\eta n)\log(1/\eta))\le n-t-3$ neighbors in $G$. This implies that there are at least $3$ vertices $v\in V(G_t)$ which have at least $2$ neighbours in $G_t$, all of which are in $U$. So, the $v$-coordinate of $A(G_t) \mbf x$ is $\deg_{G_t}(v)\lambda\ge 2\lambda>0$. The desired result follows.
\end{proof}

The above lemma handles almost-kernel vectors that are dominated by a nonzero level set, but we also need to handle almost-kernel vectors that are dominated by their zero level set (i.e., almost-kernel vectors with small support). To this end we need the notion of a \emph{minimal} vector (previously appearing in work of DeMichele, the first author, and Moreira~\cite{DGM22}).

\begin{definition}
Say that a vector $\mbf v$ is \emph{minimal} if for any $\mbf w\in \mb{R}^{n}\setminus\{\mbf{0}\}$ with $\on{supp}(\mbf w)\subsetneq \on{supp}(\mbf v)$, we have $\on{supp}(A\mbf w)\nsubseteq\on{supp}(A\mbf v)$.
\end{definition}

It is clear from the above definition that for any vector $\mbf v$, there is a minimal vector whose support is contained in $\mbf v$. In fact more is true: every $\ell$-almost kernel vector can be written as a sum of such minimal almost kernel vectors.

\begin{lemma}\label{lem:minimal-decompose}
For any matrix $A\in \mb R^{n\times n}$ and vector $\mbf v\in \mb R^n$, we can write $\mbf v$ as a sum of minimal vectors $\mbf w$ satisfying $\on{supp}(A\mbf w)\subseteq \on{supp}(A\mbf v)$.
\end{lemma}
\begin{proof}
Suppose for the purpose of contradiction that the lemma statement is false, and let $\mbf v=(v_1,\ldots,v_n)$ be a minimal-support counterexample (i.e., with $|\on{supp}(\mbf v)|$ as small as possible). We are assuming there is no way to represent $\mbf v$ as a sum of minimal vectors $\mbf w$ which satisfy $\on{supp}(A\mbf w)\subseteq \on{supp}(A\mbf v)$. We say a vector $\mbf w$ is \emph{properly contained} in $\mbf v$ if $\on{supp}(\mbf w)\subsetneq \on{supp}(\mbf v)$ and $\on{supp}(A\mbf w)\subseteq \on{supp}(A\mbf v)$.

By assumption, $\mbf v$ is not itself minimal, meaning that there is a vector $\mbf w=(w_1,\ldots,w_n)\in \mb{R}^{n}\setminus\{\mbf{0}\}$ that is properly contained in $\mbf v$. Fix any $i\in \on{supp}(\mbf w)$, let $\lambda=v_i/w_i$ and let $\mbf w'=\mbf v-\lambda \mbf w$. Then $\mbf w'$ is properly contained in $\mbf v$ as well (note that $i\notin \on{supp}(\mbf w')$).

Since $\mbf v$ is a minimal-support counterexample, we can write $\lambda \mbf w$ (respectively $\mbf w'$) as a sum of vectors that are properly contained in $\lambda \mbf w$ (respectively, properly contained in $\mbf w'$). Note that proper containment is transitive; since $\mbf v=\mbf w'+\lambda \mbf w$, we can now write $\mbf v$ as a sum of vectors that are properly contained in $\mbf v$, which is a contradiction.
\end{proof}

Minimal almost-kernel vectors enjoy certain combinatorial properties, which we capture in the notion of a \emph{stalk}. For a graph $G$ and a vertex set $R$, write $N(R)$ for the union of neighbourhoods of vertices in $R$ (so $N(R)$ may intersect $R$).

\begin{definition}\label{def:stalk}
Given a graph $G$, call a set of vertices $R\subseteq V(G)$ an $(r, s, \ell)$-\emph{stalk} for $G$ if:
\begin{enumerate}[{\bfseries{S\arabic{enumi}}}]
\setcounter{enumi}{-1}
    \item\label{S0} $|R| = r$ and $|N(R)| = s$,
    \item\label{S1} $s\ge r - 1 + \ell$,
    \item\label{S2} $R$ cannot be split into two nonempty sets with disjoint neighborhoods, and
    \item\label{S3} All but exactly $\ell$ vertices $v\in N(R)$ have at least two neighbours in $R$.
\end{enumerate}
The $\ell$ vertices $v\in N(R)$ with $|N(\{v\})\cap R|\le 1$ are called the \emph{exceptional vertices} for the stalk $R$. Also, we use the shorthand ``$(r,\ell)$-stalk'' to describe an $(r,s,\ell)$ stalk for any $s\ge r-1+\ell$, and the shorthand ``$(\le\!q,\ell)$-stalk'' to describe an $(r,\ell)$-stalk for any $r\le q$.
\end{definition}

\begin{lemma}\label{lem:minimal-stalk}
Consider an $n$-vertex graph $G$ and a minimal $\ell$-almost kernel vector $\mbf v$ of its adjacency matrix $A(G)$. Then $R=\on{supp}(\mbf v)$ is a $(|R|,\ell')$-stalk, for some $\ell'\le \ell$.
\end{lemma}

\begin{proof}
Let $Q=\supp(A(G_t)\mbf v)$ (so $|Q|=\ell$). For \cref{S1}, suppose for the purpose of contradiction that $|N(R)|<|R|+\ell-1$. Let $R'$ be obtained by removing an arbitrary vertex of $R$, and let $A'$ be the $R'\times (N(R)\setminus Q)$ submatrix of $A(G)$. Then $A'$ has $|R|-1$ rows and $|N(R)|-\ell<|R|-1$ columns, so has a nonzero left kernel vector $\mbf w$. Padding this vector with zeroes gives a nonzero vector $\mbf w'\in \ker(A(G))$ with $\supp(\mbf w')\subsetneq \supp(\mbf v)$ and $\supp(A(G)\mbf w')\subseteq \supp(A(G)\mbf v)$, contradicting the minimality of $\mbf v$.

For \cref{S2}, suppose for the purpose of contradiction that $R$ can be split into two nonempty sets $R_1,R_2$ with distinct neighbourhoods. Let $\mbf v_1\in \mb R^n$ be the vector obtained from $\mbf v_1$ by setting all entries not indexed by $R_1$ to zero. Then $\mbf v_1$ contradicts the minimality of $\mbf v$.

For \cref{S3}, suppose for the purpose of contradiction that some $u\in N(R)\setminus Q$ has exactly one neighbour in $R$ (call that neighbour $w$). But then the $u$-entry of $A(G)\mbf v$ is the same as the $w$-entry of $\mbf v$, which is impossible (recall that $w\in R=\supp(\mbf v)$ while $u\notin Q=\supp(A(G)\mbf v)$).
\end{proof}

We also need the following lemma deducing the precise corank of the adjacency matrix from information about its stalks and special cycles. Recall that a special cycle in $G$ is an induced cycle with length divisible by 4, such that every second vertex has degree 2 in $G$. Recall that $s(G)$ is the number of special cycles, counting isolated special cycles twice.

\begin{lemma}\label{lem:corank-s}Let $G$ be an $n$-vertex graph, and let
\[K^{(\eta)}=\{\mathbf{v}\in\ker A(G)\colon|\on{supp}(x)|\le\eta n.\}\]
Suppose that every $(\le\!\eta n,0)$-stalk $R$ is an $(|R|,|R|,0)$-stalk which satisfies $e(R\cup N_G(R))\le 2|R|$ and $\deg_G(v)\ge 2$ for all $v\in R$. Also, suppose that all special cycles in $G$ have length at most $\eta n$ and are vertex-disjoint from each other.
\begin{enumerate}
    \item $\supp K^{(\eta)}$ is precisely the set of degree-2 vertices in special cycles of $G$.
    \item $\dim K^{(\eta)}=s(G)$.
\end{enumerate}
\end{lemma}
We emphasise that \cref{lem:corank-s} is a non-probabilistic statement about general graphs $G$ (though we will eventually apply it to the random graphs $G_t$ defined in \cref{sec:RMT-overview}).
\begin{proof}
Let $V^{\mr{spec}}$ be the set of degree-2 vertices in special cycles of $G$, and suppose the vertex set of $G$ is $\{1,\ldots,n\}$. Recall from \cref{fact:special-kernel} that special cycles give rise to kernel vectors, so $\supp K^{(\eta)}\supseteq V^{\mr{spec}}$. Also, since the special cycles in $G$ are vertex-disjoint,
each of the $s(G)$ kernel vectors obtained in this way have disjoint supports, so are linearly independent. This shows that $\dim K^{(\eta)}\ge s(G)$.

Recalling \cref{lem:minimal-decompose}, to prove that $\dim K^{(\eta)}= s(G)$ it now suffices to show that every nonzero minimal kernel vector of $A(G)$ whose support size is at most $\eta n$ is a multiple of one of the explicit kernel vectors arising from special cycles via \cref{fact:special-kernel}. To this end, consider a nonzero minimal kernel vector $\mbf v=(v_1,\ldots,v_n)$, let $R=\on{supp}(\mbf v)$, and suppose $|R|\le \eta n$. Then, by \cref{lem:minimal-stalk} and the assumption in the lemma, $R=\on{supp}(\mbf v)$ is an $(|R|,|R|,0)$-stalk consisting of vertices with degree at least 2, such that $e(R\cup N_G(R))\le 2|R|$. This is only possible if $R\cup N_G(R)$ is an induced cycle in which the vertices in $R$ have degree exactly 2. Write $u_2,u_4,\ldots,u_{2q}$ (in order) for the vertices of this cycle, where $u_2,u_4,\ldots,u_{2q}\in R$.

For each $u_{i}\in N_G(R)$ (with $i$ odd), the $u_i$-coordinate of $A\mbf v$ is precisely $v_{u_{i-1}}+v_{u_{i+1}}$. Since $\mbf v$ is a kernel vector, each of $u_2,u_4,\ldots,u_{2q}$ must have the same absolute value, and cyclically alternate their signs. This is only possible if $q$ is even (i.e., if $R\cup N_G(R)$ induces a special cycle), and implies that $\mbf v$ is a multiple of one of the explicit kernel vectors arising from \cref{fact:special-kernel}.
\end{proof}

Next, the following lemma shows how to establish the $\eta$-unstructuredness property in \cref{def:unstructured} using information about stalks.
\begin{lemma}\label{lem:unstructured-stalk}
For $0<\eta<1/2$, consider a graph $G$ with at most $n/10$ different $(\le\!\eta n,1)$-stalks and at most $n^2/10$ different $(\le\!\eta n,2)$-stalks, where additionally, $A(G)$ has no 2-almost kernel vector with a nonzero level set of size at least $(1-\eta)n$. Then the adjacency matrix $A(G)$ is $\eta$-unstructured.
\end{lemma}
\begin{proof}
Since we have assumed $A(G)$ has no 2-almost kernel vector with a nonzero level set of size at least $(1-\eta)n$, we only need to consider unbalanced almost-kernel vectors with a large zero level set.

We first claim that whenever there is a vector $\mbf v\in \mb R^n$ with $\supp(A\mbf v)=\{i,i'\}$, there is a \emph{minimal} vector $\mbf w\in \mb R^n$ with $\supp(A\mbf w)\in \{\{i\},\{i'\},\{i,i'\}\}$ and $\supp(\mbf w) \subseteq \supp(\mbf v)$. Indeed, consider the decomposition into minimal vectors given by \cref{lem:minimal-decompose}. For all of these vectors $\mbf w$ we have $\supp(A\mbf w)\subseteq\{i,i'\}$, and it cannot be the case that all of these vectors are kernel vectors of $A$ (otherwise $\mbf v$ would be a kernel vector as well). Since $\supp(\mbf w) \subseteq \supp(\mbf v)$, if $\mbf v$ is non-$\eta$-balanced with a large zero level set, then $\mbf w$ is also  non-$\eta$-balanced.

Then, using \cref{lem:minimal-stalk}, the assumptions in the lemma imply that there are at most $n/10$ different $i$ for which there is a non-$\eta$-balanced vector $\mbf v\in \mb R^n$ with $\supp(A\mbf v)=\{i\}$, and there are at most $n^2/10$ different pairs $\{i,i'\}$ for which there is a non-$\eta$-balanced vector $\mbf v\in \mb R^n$ with $\supp(A\mbf v)=\{i,i'\}$. It follows that there are at least $n(n-1)-2n(n/10)-2(n^2/10)\ge \eta n^2$ pairs of distinct indices $(i,i')$ for which every $\mbf v$ with $\supp(A\mbf v)=\{i,i'\}$ is $\eta$-balanced, meaning that $A(G)$ is $\eta$-unstructured.
\end{proof}

Essentially all that remains is to carefully analyse the stalks that exist in the random graphs $G_N,\ldots,G_0$ defined in \cref{sec:RMT-overview}.

\subsection{Estimates on stalks}
Recall the definitions of the random graphs $G_N,\ldots,G_0$ from \cref{sec:RMT-overview}: to obtain $G_N$ we looked at the degrees of the first $\lfloor\alpha n\rfloor$ vertices of a random graph $G\sim \mc K(n,m,2)$ and deleted the vertices with degree at least $\Delta$, then we added back these vertices in a random order to obtain $G_{N-1},\ldots,G_0$.

Crucially, a similar proof as for \cref{lem:rotate-sequence-simple} shows that for each $t\le\alpha n$, if we condition on the degree sequence of $G_t$ (more precisely, we condition on the event $t\le N$ and then further condition on the degree sequence), then $G_t$ is distributed like a uniformly random graph with that degree sequence. 
So, we perform various calculations after conditioning on properties of the degree sequence of $G_t$. Specifically, the properties we need are as follows.
\begin{definition}\label{def:good-degree}
Consider integers $m>n,t$ and some $\kappa>0$. Choose $\lambda>0$ such that if $Z\sim \on{Poisson}(\lambda)$, then $2m/n=\mb E[Z|Z\ge 2]$. A sequence $\mbf d=(d_1,\ldots,d_{n-t})\in \mb R^{n-t}$ is \emph{$(n,m,t,\kappa)$-typical} if it satisfies the following properties.
\begin{enumerate}[{\bfseries{T\arabic{enumi}}}]
    \item\label{T1} $d_v=0$ for at most $(t/n)^{3/2}n$ different $v$.

    \item\label{T2} $d_v=1$ for at most $\kappa^{-1}t\log(n/t)$ different $v$. (Here we use the convention $0\log\infty=0$ for the case $t=0$.)
    \item\label{T3} $d_v=2$ for at most $(\Pr[Z=2|Z\ge 2] + \kappa)n$ different $v$.
    \item\label{T4} $2(1-\kappa)m\le d_1+\cdots+d_{n-t}\le 2m$.
    \item\label{T5} For any $U\subseteq\{1,\ldots,n-t\}$ with $|U| = u$, we have $\sum_{v\in U}d_v\le \kappa^{-1}u\log(2n/u)$.
    \item\label{T6}$\sum_{v}\binom{d_v}2\le \left(E\left[\binom{Z}2\middle|Z\ge2\right]+\kappa\right)n$, and $\sum_{v}\binom{d_v}j\le \kappa^{-j} n$ for $j\ge 3$.
\end{enumerate}
\end{definition}
We remark that most of these bounds are essentially sharp, for a typical outcome of the degree sequence of $G_t$. The exception is \cref{T1}: the number of isolated vertices is typically about $(t\log(n/t)/n)^2n$ (but we will not need such a strong estimate).

\begin{lemma}\label{lem:good-deg-seq}
Fix $\eps,\alpha,\Delta,\kappa>0$ such that $1/\Delta\ll\alpha\ll\kappa\ll\varepsilon$. 
Recall the definitions of $G_N,\ldots,G_0$ (in terms of $\alpha,\Delta$) from \cref{sec:RMT-overview}.
Then, for each $t\le \alpha n$: with probability at least $1 - (t/n)^{1/4}$, if $t\le N$ then the degree sequence of $G_t$ is $(n,m,t,\kappa)$-typical.
\end{lemma}
Recall that $N=o_{\Delta\to \infty}(n)$ whp, so when we take $\Delta\to \infty$, \cref{lem:good-deg-seq} becomes a with-high-probability statement.
\begin{proof}
First, note that by \cref{lem:degree-sequence-poisson}(A4), with probability at least $1-n^{-\omega(1)}$
we have $\sum_{v\in S}\deg_G(v)=O_\eps(s\log(2n/s))$ for each size-$s$ subset $S\subseteq V(G)$. This directly yields \cref{T5} since the degrees in $G_t$ are at most those in $G$. Further, \cref{lem:degree-sequence-poisson}(A4) yields that the sum of degrees of the vertices in $V(G)\setminus V(G_t)$ (i.e., the vertices $v_{t-1},\ldots,v_0$) is $O_\kappa(t\log(2n/t)) \le \kappa n\le \kappa m$ for $\alpha\ll \kappa$ (recalling that $t\le \alpha n$). This yields \cref{T4} and additionally \cref{T2}, since the number of degree $1$ vertices is at most the number of edges from $V(G) \setminus V(G_t)$. Similarly \cref{T3,T6} hold with probability $1-n^{-\omega(1)}$, by \cref{lem:degree-sequence-poisson}(A1) and (A2,A3) respectively.

For \cref{T1}, we need a simple calculation in the configuration model. Condition on an outcome of the degree sequence $\mbf d$ of $G$, satisfying the conclusions of \cref{lem:degree-sequence-poisson}. This determines $T$; also condition on an outcome of $V(G)\setminus V(G_t)=\{v_{t-1},\ldots,v_0\}$, and let $T_t=T\cap(V(G)\setminus V(G_t))$ (so $|T_t|=t$). By \cref{lem:rotate-sequence-simple}, after our conditioning, we have $G\sim \mb G(\mbf d)$. By \cref{lem:simplicity}(A), it suffices to prove the desired result for $G\sim\mb G^\ast(\mbf d)$ (i.e., we may work in the configuration model).

A vertex $v$ can only be isolated in $G_t$ if it has at least two neighbours in $T_t$. The number of stubs corresponding to the vertices in $T_t$ is $O_{\kappa}(t\log(2n/t))$, so the probability that this happens is $O_{\kappa}(t\log(2n/t)/n)^2\le (t/n)^{1.9}$ for $\alpha\ll \kappa$. That is to say, the expected number of isolated vertices is at most $(t/n)^{1.9}n$, so \cref{T1} holds with probability at least $1-(t/n)^{1/4}$ by Markov's inequality.
\end{proof}

Now, the following definition captures the stalks which are not handled by \cref{lem:dense-subset}.

\begin{definition}\label{def:sparse-stalk}
Say a stalk $R$ is \emph{$\varepsilon$-sparse} if there is no subset of vertices in $R\cup N(R)$ of any size $u$ which spans more than $u+\lfloor C_{\ref{lem:dense-subset}}(\varepsilon)u/\sqrt{\log(2n/u)}\rfloor$ edges.
\end{definition}

The following lemma encapsulates a careful analysis of small sparse stalks in random graphs with a given typical degree sequence.

\begin{lemma}\label{lem:structure-master}
Fix $\eta,\eps,\kappa, \alpha$ such that
$\alpha\ll \eta \ll \kappa \ll \eps \ll 1$. Choose $n,m,t$ with $1+\varepsilon\le m/n\le 1/\varepsilon$ and $t\le \alpha n$, let $\mbf d\in \mb R^{n-t}$ be a $(n,m,t,\kappa)$-typical sequence, and let $G_t\sim\mb G(\mbf d)$.

For $s\ge r-1+\ell$ and $1\le r\le \eta n$, let $X_{r,s,\ell}$ be the number of $\eps$-sparse $(r,s,\ell)$-stalks in $G_t$, and let $X_{r,s,\ell}'$ be the number of such stalks $R$ for which there is a vertex $v\in R$ with $d_v=1$. Then we have the following estimates.
\begin{enumerate}
    \item If $s\ge r+1$, then $\mb EX_{r,s,0}\lesssim_\kappa e^{-\Omega_\eps(r)}/n$.
    \item $\mb E X_{r,r,0}\lesssim_\kappa e^{-\Omega_\eps(r)}$.
    \item If $t\le n^{1/8}$ then $\mb E X_{r,r,0}'\lesssim_\kappa e^{-\Omega_\eps(r)}n^{-3/4}$.
    \item If $s = r-1$, then $\mb E X_{r,s,0}\lesssim_\kappa e^{-\Omega_\eps(r)}(t/n)^{3/2}n$.
    \item For any $\ell\le 2$, we have $\mb E X_{r,s,\ell}\lesssim_\kappa e^{-\Omega_\eps(r)}n$.
    \item If $t \le \log n$, then $\mb E X_{r,s,1}\lesssim_\kappa e^{-\Omega_\eps(r)}(\log n)^{2}$.
\end{enumerate}
\end{lemma}
We emphasise that the above estimates are only for $r\le \eta n$ (i.e., for those stalks that correspond to small-support kernel vectors of $A(G_t)$).

We remark that our notion of a stalk has some resemblance to the notion of a \emph{flipper} in \cite[Section~8]{CCKLRb}, and \cite[Lemma~8.1]{CCKLRb} plays a similar role to \cref{lem:structure-master}. However, in our setting we need much more precision, and the details are much more involved.

For our proof of \cref{lem:structure-master} we collect some elementary estimates. First, we will need to estimate products of factorials.

\begin{lemma}\label{lem:combine-bound}
If $(k_i)_{i=1}^r$ is a sequence of nonnegative integers with $\sum_{i=1}^r k_i = a$ and $\sum_{i=1}^r ik_i = b$ then 
\[\prod_{i=1}^r k_i!\ge e^{-b}a!.\]
\end{lemma}
\begin{proof}
By the multinomial theorem we have
\[\binom{a}{k_1,\ldots,k_r}k_1^{k_1}\cdots k_r^{k_r}\le (k_1+\cdots+k_r)^a=a^a.\]
So,
\[\frac{a!}{\prod_{j=1}^rk_j!}=\binom{a}{k_1,\ldots,k_r}\le \prod_j\bigg(\frac{a}{k_j}\bigg)^{k_j}=\exp\bigg(a\sum_j\frac{k_j}{a}\log\bigg(\frac{a}{k_j}\bigg)\bigg).\]
We can interpret the right-hand side as $e^{a H(Y)}$, where $H(Y)$ is the (base-$e$) entropy of a random variable $Y$ satisfying $\Pr[Y=i] = k_i/a$ for each  $i\in\{1,\ldots,r\}$.
Note that $\mb EY=b/a$; among positive integer random variables with this mean, the maximum possible entropy is attained by a geometric random variable with parameter $p:=a/b$ (see for example \cite{LV72}). The entropy of such a geometric random variable is $(-p\log p-(1-p)\log(1-p))/p\le 1-\log p$. So,
\[\frac{a!}{\prod_{j=1}^rk_j!}\le e^{a H(Y)}\le e^{a(1-\log p)}=(e/p)^{pb}\le e^b,\]
using the inequality $(e/p)^p\le e$ (which holds for all $0\le p\le 1$).
\end{proof}

We also need the following general-purpose inequality to bound various binomial coefficients.
\begin{lemma}\label{lem:binom}
For any  $z>0$, and any $a,b\in \mb N$, we have $\binom{a}{b}\le(1+z)^a e^{O_z(b)}$.
\end{lemma}
\begin{proof}
If $4b\ge z^2a$, we have $\binom{a}{b}\le (ae/b)^b\le (4e/z^2)^b=e^{O_z(b)}$. Otherwise, if $4b<z^2a$,  writing $x=b/a$, we have $\binom a b\le (ae/b)^b=((e/x)^x)^a\le (1+2\sqrt x)^a\le (1+z)^a$. (Here we used the inequality $(e/x)^x\le 1+2\sqrt x$, which holds for all $x\ge 0$.)
\end{proof}

We are now ready to prove \cref{lem:structure-master}.

\begin{proof}[Proof of \cref{lem:structure-master}]
In this proof we think of $\varepsilon$ as being a constant (without explicitly writing $\eps$ as a subscript on asymptotic notation), and we simply write ``sparse'' instead of ``$\eps$-sparse''. Also, throughout this proof we let $f(s) = \lfloor C_{\ref{lem:dense-subset}}(\varepsilon)s/\sqrt{\log(2n/s)}\rfloor$. Note that $f$ is essentially sub-linear, in the sense that $f(a+b)\le f(a)+f(b)+1$.

First, we briefly note that isolated vertices are $(1,0,0)$-stalks (and by \cref{S2}, isolated vertices are not contained in any other types of stalks). By \cref{T1}, the number of isolated vertices is at most $(t/n)^{3/2} n\lesssim n$, which handles the $(r,s)=(1,0)$ cases of (4) and (5). For the rest of the proof we can restrict our attention to $(r,s,\ell)$-stalks which do not contain any isolated vertices.

The reader may find it helpful to think of two basic examples of sparse stalks that may occur in graphs with minimum degree at least 2. First, for any even cycle in which every second vertex has degree 2, we can take those degree-2 vertices as a $(r,r,0)$-stalk. Second, for any odd cycle in which every vertex has degree 2, we can take the entire vertex set of the cycle as an $(r,r,0)$-stalk. It is not hard to estimate the expected number of these types of cycles using the configuration model.

Roughly speaking, the proof strategy is as follows. First, we prove a sequence of inequalities (\cref{claim:structure}) showing that \emph{every} stalk approximately resembles a union of copies of these two examples. Then, we do an explicit configuration-model calculation that parallels the cycle calculation mentioned above.

For this entire proof we will work with the configuration model $G_t\sim \mb G^\ast(\mbf d)$, taking $X_{r,s,\ell}=X_{r,s,\ell}'=0$ whenever $G_t$ is not simple. (By \cref{lem:simplicity}, it suffices to prove the desired estimates in this setting, noting that \cref{T6} implies that $\sum_v d_v^2=O(n)$.)

\medskip
\noindent\textbf{Step 1: Parameters of stalks.}
Fix $r,s,\ell$ with $s\ge r-1+\ell$ and $\ell\le 2$. We define a number of parameters of a sparse $(r,s,\ell)$-stalk $R$ in $G_t$. We will later study the contribution to $\mb E X_{r,s,\ell}$ and $\mb E X_{r,s,\ell}'$ from each choice of these parameters.
\begin{itemize}
    \item Let $S=N(R)$, let $S_1=R_1=S\cap R$, let $S_2=S\setminus R$ and let $R_2=R\setminus S$. For each $i\in \{1,2\}$, let $r_i=|R_i|$ and $s_i=|S_i|$. (So, $s=s_1+s_2$ and $r=r_1+r_2$ and $r_1=s_1$.)
    \item For each $i\in \{1,2\}$, let $\ell_i$ be the number of exceptional vertices in $S_i$. (So, $\ell=\ell_1+\ell_2$.)
    \item Let $x$ be the number of $v\in R_2$ which have $\deg_{G_t}(v)=1$.
    \item Let $m_1$ be the number of edges in $R_1$, let $m_{1,2}$ be the number of edges between $R_1$ and $S_2$, and let $m_{2,2}$ be the number of edges between $R_2$ and $S_2$.
    \item For $i\ge 1$, let $k_i$ be the number of vertices in $S_2$ which have exactly $i$ neighbours in $R$ (so in particular $k_1=\ell_2$).
\end{itemize}
There are a number of simple inequalities that must hold between our parameters. First, by \cref{S1}, we have \begin{equation}
    s_2-r_2=s-r\ge \ell-1.\label{eq:S1-consequence}
\end{equation}
Second, by $\eps$-sparsity, we have
\begin{equation}
    m_1\le r_1+f(r_1),\quad m_{2,2}\le r_2+s_2+f(r_2+s_2),\quad m_1+m_{1,2}+m_{2,2}\le r_1+r_2+s_2+f(r_1+r_2+s_2)\label{eq:sparsity-consequence}
\end{equation}
Third, recall from \cref{S3} that all non-exceptional vertices in $S$ have at least two neighbours in $R$. By the considerations at the start of the proof, we are assuming $R$ contains no isolated vertices, and by definition $R$ has exactly $x$ vertices with degree $1$ into $S$. So, summing over degrees in $R_1,R_2,S_2$, we obtain
\begin{equation}
m_1\ge\lceil(2r_1-\ell_1)/2\rceil=r_1-\lfloor\ell_1/2\rfloor,\quad m_{1,2}+m_{2,2}\ge 2s_2-\ell_2,\quad m_{2,2}\ge 2r_2-x.
\label{eq:S-consequence}
\end{equation}
Finally, using \cref{T5}, we have
\begin{equation}m_{1,2}\le O_{\kappa}(s_2\log(2n/s_2)).\label{eq:T5-consequence}\end{equation}
(Other similar inequalities can also be obtained via \cref{T5}, but we will not need them.)

\medskip
\noindent\textbf{Step 2: The structure of sparse stalks.}
We now combine the above inequalities, to prove the following claim about the parameters of a sparse $(r,s,\ell)$-stalk $R$. Roughly speaking, the claim says that if we consider two disjoint copies of $R$ and $S$ and a bipartite graph of the edges between the two, then almost all vertices have degree 2, and almost all edges are inside $R_1=S_1$ or between $R_2$ and $S_2$. (Note that there are no edges between $R_2$ and $S_1=R_1$, by the definition of $R_1$.)

\begin{claim}\label{claim:structure}
Consider a sparse $(r,s,\ell)$-stalk $R$, with parameters as defined as in Step 1.
\begin{enumerate}
    \item $s_2=r_2+O(1+f(r))$ (i.e., $S$ and $R$ have roughly the same size).
    \item All but $O(1+x+f(r))$ vertices in $R$ have degree exactly 2.
    \item $k_2=s_2+O(1+x+f(r))$ (i.e., almost all vertices in $S_2$ have degree exactly 2 into $R$).
    \item $m_{1,2}=O(1+x+f(r))$ (i.e., there are few edges between $R_1$ and $S_2$).
    \item $2s_2+O(1+x+f(r))\le m_{2,2}\le 2s_2+O(1+f(s_2))$ (i.e., the number of edges between $S_2$ and $R_2$ is not much more than $2s_2$, which by (1) is roughly the same as $2r_2$).
    \item $m_1=s_1+O(1+f(r_1))$ (i.e., the number of edges inside $S_1=R_1$, which is half its degree sum, is not much more than $s_1=(2s_1)/2$).
\end{enumerate}
\end{claim}
In light of \cref{rem:f} and since $r/n\le \eta$, if $\eta\ll \kappa$ we have $f(q)\le \kappa q$ for $q\in \{s_1,s_2,r\}$ (we are also using \cref{claim:structure}(1) here to show $s/n$ is small). In particular, throughout the rest of the proof, terms of the form $f(q)$ can be viewed as being ``lower order'' than $q$.
\begin{proof}
First, (6) follows from the first inequalities in \cref{eq:sparsity-consequence} and \cref{eq:S-consequence}.

Next, by combining (6), the last inequality in \cref{eq:sparsity-consequence}, and the second inequality in \cref{eq:S-consequence}, we obtain $s_2\le r_2+O(1+f(r+s_2))$. Together with \cref{eq:S1-consequence}, this nearly gives us (1), but we need to do a little more work to replace the error term ``$O(1+f(r+s_2))$'' with the desired error term ``$O(1+f(r))$''. Specifically, to show that these error terms are equivalent, we need to prove that $s_2\lesssim r$. By \cref{T5} we have $s\le\sum_{v\in R}d_v\le\kappa^{-1}r\log(2n/r)\lesssim\eta^{1/2}n$ (assuming $\eta\ll\kappa$), so $f(r+s_2)\lesssim(r+s_2)/\sqrt{\log(1/\eta)}$ and thus our initial inequality implies $s_2\lesssim r$, as desired.

Then, (5) follows from \cref{eq:S1-consequence}, the second inequality in \cref{eq:sparsity-consequence}, (1), and the last inequality in \cref{eq:S-consequence}. After this, we can deduce (4) from (1), (5), (6), and the last inequality in \cref{eq:sparsity-consequence}.

Finally, by (1,4,5,6), note that
\[\sum_{v\in R}(\deg(v)-2)=2m_1+m_{1,2}+m_{2,2}-2r\lesssim 1+x+f(r),\quad \sum_{v\in S_2}(\deg_R(v)-2)=m_{1,2}+m_{2,2}-2s_2\lesssim 1+x+f(r).\]
Recall that at most $x\lesssim 1+x+f(r)$ vertices in $R$ have degree less than 2, and at most $\ell\lesssim 1+x+f(r)$ vertices in $S$ have fewer than $2$ neighbours in $R$. So, (2) and (3) follow.
\end{proof}

\medskip
\noindent\textbf{Step 3: Breaking down the expectation.}
For a vector of parameters
\[\mbf p=(r_1,r_2,s_1,s_2,\ell_1,\ell_2,x,m_1,m_{1,2},m_{2,2},(k_i)_{i=1}^r),\]
we now consider the contribution to $\mb EX_{r,s,\ell}$ from sparse stalks with these parameters. We will eventually sum over all possible $\mbf p$. (For $\mb EX_{r,r,0}'$, we simply sum over all $\mbf p$ with $x>0$.)

Recall that we are working in the configuration model $\mb G^\ast(\mbf d)$, for a particular $(n,m,t,\kappa)$-typical degree sequence $\mbf d=(d_1,\ldots,d_{n-t})$ (so we have $n-t$ buckets corresponding to vertices, and within the bucket corresponding to a vertex $v$, there are $d_v$ stubs).

First, we define $N_{\mbf p}$ to be ``the number of possible places that a stalk may appear''. Specifically, $N_{\mbf p}$ is the number of ways to choose disjoint vertex sets $R_1,R_2,S_2$, and to colour all the stubs from $R_1$ blue and yellow, and to colour some stubs from $S_2$ red, and all the stubs from $R_2$ green, such that the following hold.

\begin{itemize}
    \item $|R_1|= r_1$, $|R_2| = r_2$, $|S_2| = s$.
    \item There are $2m_1+m_{1,2}$ stubs coming from $R_1$. Exactly $2m_1$ are blue and exactly $m_{1,2}$ are yellow.
    \item There are exactly $m_{2,2}$ stubs coming from $R_2$, all coloured green.
    \item Among the stubs from $S_2$, exactly $m_{1,2} + m_{2,2}$ are red.
    \item Each vertex in $S_2$ has at least two red stubs, except exactly $\ell_2$ which have one red stub.
    \item Exactly $x$ of the vertices in $R_2$ have degree 1.
    \item Exactly $\ell_1$ of the vertices in $R_1$ have degree 1, and none have degree 0.
\end{itemize}

Then, for each of the choices of $R_1,R_2,S_2$ and red/blue/yellow colourings as above, we consider the probability that
\begin{itemize}
    \item the $2m_1$ blue stubs (from $R_1$) pair with each other, and
    \item the $m_{1,2}$ yellow stubs (from $R_1$) pair with red stubs (from $S_2$), and
    \item the $m_{2, 2}$ stubs from $R_2$ pair with red stubs (from $S_2$).
\end{itemize}
This probability only depends on $\mbf p$; denote it by $P_{\mbf p}$. Observe that $\mb E X_{r,s,\ell}\le \sum_{\mbf p} N_{\mbf p}P_{\mbf p}$.

\medskip
\noindent\textbf{Step 4: Estimating combinatorial quantities.}
Let $Q$ be the number of possibilities for $\mbf p$.
We now give upper bounds for $Q$, $N_{\mbf p}$, and $P_{\mbf p}$. We will very often want to use the expression ``$(1+O(\kappa))^re^{O_{\kappa}(1+x+f(r))}$'' as a multiplicative error term, so we introduce the shorthand ``$O^\ast(1)$'' for a term of this form.

\begin{claim}\label{claim:num-p}
$Q=O^\ast(1)$.
\end{claim}
\begin{proof}
Recall the definitions of the various parameters from Step 1, and recall from \cref{claim:structure} that $s,m_1,m_{1,2},m_{2,2}\lesssim r$. It is easy to see that there are at most $(\ell+1)(r+1)^2(s+1)\lesssim r^3$ choices for $r_1,r_2,s_1,s_2,\ell_1,\ell_2$, at most $O(r^3)$ choices for $m_1,m_{1,2},m_{2,2}$,
and at most $r+1$ choices for $x$. Then, note that $\sum_{i=1}^r i k_i=m_{1,2} + m_{2,2}$, so $(k_i)_{i=1}^r$ encodes an integer partition of $m_{1,2} + m_{2,2}$ ($k_i$ is the number of parts of size $i$). For each $m_{1,2},m_{2,2}$, the number of such partitions is $e^{O(\sqrt{m_{1,2} + m_{2,2}})}=e^{O(\sqrt r)}$. 
\end{proof}

\begin{claim}\label{claim:Pp}
For any $\mbf p$ we have
\[P_{\mbf p}\le O^\ast(1)\frac{r_1^{m_1}s_2^{m_{1,2}+m_{2,2}}}{e^{r_1+2r_2}m^{m_1+m_{1,2}+m_{2,2}}}.\]
\end{claim}
\begin{proof}
Let $d_\Sigma=d_1+\cdots+d_{n-t}=(2+O(\kappa))m$. First note that
\[P_{\mbf p}\le \frac{(2m_1)!!(d_\Sigma-2m_1)!!}{d_\Sigma!!}\cdot \binom{m_{1,2} + m_{2,2}}{m_{1,2}}\cdot \frac{m_{1,2}!m_{2,2}!}{(d_\Sigma-2m_1)_{m_{1,2}+m_{2,2}}}.\]
Indeed, the first term accounts for the probability that the $2m_1$ blue stubs pair with each other, the second term is the number of ways to choose which of the $m_{1,2}+m_{2,2}$ red stubs will pair with yellow stubs and which will pair with green stubs, and the last two terms bound the probability that the red stubs do indeed pair with the yellow and green stubs in this way.

We now just need to manipulate the above expression using the inequalities in \cref{claim:structure}. Throughout, we will use the crude bounds that $m_1,m_{1,2},m_{2,2}\lesssim r$ without further remark.

First, we have $\binom{m_{1,2} + m_{2,2}}{m_{1,2}}\le (1+\kappa)^{m_{1,2}+m_{2,2}}e^{O_\kappa(m_{1,2})}$ by \cref{lem:binom},
and
\[\frac{(2m_1)!!(d_\Sigma-2m_1)!!}{d_\Sigma!!}=\frac{\binom{d_{\Sigma}/2}{m_1}}{\binom{d_{\Sigma}}{2m_1}}\le (1+O(\kappa))^{m_1}\frac{\binom{m}{m_1}}{\binom{2m}{2m_1}}\le (1+O(\kappa))^{r} \bigg(\frac{m_1}{em}\bigg)^{m_1},\]
provided $\eta\ll \kappa$ (recall that $m_1\lesssim r\le \eta n$, while $m\ge n$). Then (again with $\eta\ll\kappa$) we have $(d_\Sigma-2m_1)_{m_{1,2}}=(d_\Sigma-O(r))^{m_{1,2}+m_{2,2}}=(1+O(\kappa))^{O(r)}(d_\Sigma)^{m_{1,2}+m_{2,2}}$. Using Stirling's formula, we therefore have
\[P_{\mbf p}\le (1+O(\kappa))^{O(r)}e^{O_\kappa(m_{1,2})}\sqrt{m_{1,2} m_{2,2}} \bigg(\frac{m_1}{em}\bigg)^{m_1}\bigg(\frac{m_{1,2}}{2em}\bigg)^{m_{1,2}}\bigg(\frac{m_{2,2}}{2em}\bigg)^{m_{2,2}}.
\]
Next, using \cref{claim:structure}(4,5,6), we have $m_1=r_1+O(1+f(r_1))=(1+O(\kappa))r_1+O(1)$, $m_{1,2}=O(1+x+f(r))$, and $m_{2,2}\le 2s_2+O(1+f(s_2))=(1+O(\kappa))2s_2+O(1)$. So, we deduce
\[P_{\mbf p}\le O^\ast(1)\frac{r_1^{m_1}m_{1,2}^{m_{1,2}}s_2^{m_{2,2}}}{e^{r_1+2r_2}m^{m_1+m_{1,2}+m_{2,2}}}.
\]
Now, to finish the proof of the claim it suffices to show that $(m_{1,2}/s_2)^{m_{1,2}}\le O^\ast(1)$. We distinguish cases.

\begin{itemize}
    \item \textit{Case 1: $f(r)+x\le 1$.} We have $m_{1,2}\lesssim 1$ by \cref{claim:structure}(4), so $(m_{1,2}/s_2)^{m_{1,2}}\lesssim 1$.
    \item \textit{Case 2: $x\ge f(r)+2$}. Recall that $x$ counts degree-1 vertices in $R_2$, so $x\le r_2$, and recall from \cref{eq:S1-consequence} that $s_2\ge r_2-1$. Also, \cref{claim:structure}(4) implies that $m_{1,2}\lesssim x$. So, $s_2\ge r_2-1\ge x-1\ge x/2\gtrsim m_{1,2}$, meaning that $(m_{1,2}/s_2)^{m_{1,2}}\le e^{O(x)}$.
    \item \textit{Case 3: $f(r)+x\ge 2$ and $x\le f(r)+1$}. In this remaining case, note that $f(r) \gtrsim 1+x$. Since $s_2/n\le 1/2$, from \cref{eq:T5-consequence} we have $m_{1,2}\le O_{\kappa}(s_2\log (n/s_2))$, which implies that    
    \[\frac{m_{1,2}}{\log(n/m_{1,2})}\le \frac{O_{\kappa}(s_2\log(n/s_2))}{\log(n/O_{\kappa}(s_2\log(n/s_2)))}=s_2\frac{O_{\kappa}(\log(n/s_2))}{\log((n/s_2)/O_{\kappa}(\log(n/s_2)))}\le O_{\kappa}(s_2).\] 
    Here in the last inequality we have used the fact that $s_2/n \lesssim r/n \leq \eta$ is small relative to $\kappa$ (meaning that the denominator $\log((n/s_2)/O_\kappa(\log (n/s_2)))$ is at least say $\log(n/s_2)/2$).
    Now
    \begin{align*}
    \log\left(\left(\frac{m_{1, 2}}{s_2} \right)^{m_{1, 2}}\right) &= m_{1, 2}\log\left(O_{\kappa}(\log(n/m_{1,2}))\right) \lesssim m_{1, 2}\log \log\left(\frac{n}{m_{1,2}}\right)\\ &\lesssim f(r)\log \log\left(\frac{n}{f(r)}\right)\lesssim r \left(\frac{\log\log  (n/r)}{\sqrt{\log(n/r)}}\right)\le r\log(1+\kappa),
    \end{align*}
    provided $\eta\ll\kappa$. Exponentiating yields the desired result.\qedhere
\end{itemize}
\end{proof}

\begin{claim}\label{claim:Np}
For any $\mbf p$ we have
\[N_{\mbf p}\le O^\ast(1)\frac{n^{r_1+r_2+s_2}}{r_1^{r_1}s_2^{r_2+s_2}}\cdot e^{r_1+2r_2}\rho_2^{r_1+r_2}\bigg(\frac{\log(n/t)}{n/t}\bigg)^{\!x} E_2^{r_2},\]
where $\rho_2=\Pr[Z=t|Z\ge 2]$ and $E_2=\mb E\left[\binom Z2\middle|Z\ge 2\right]$ for $Z\sim\on{Poisson}(\lambda)$, where $\lambda$ is such that $2m/n=\mb E[Z|Z\ge 2]$.
\end{claim}
\begin{proof}
Recall that our degree sequence is $(n,m,t,\kappa)$-typical. First, we bound the number of choices of $R_1$. Recall from \cref{claim:structure}(2) that there is some $i=O(1+x+f(r))$ such that at least $r_1-i$ vertices in $R_1$ have degree exactly 2. The number of ways to choose a sequence of $r_1$ vertices, for which the first $r_1-i$ have degree exactly 2, is at most $(\rho_2+\kappa)^{r_1-i}n^{r_1}=(\rho_2+\kappa)^{r_1}(O(1))^i n^{r_1}$. Also, there are $\binom{2m_1+m_{1,2}}{m_{1,2}}$ ways to choose a blue/yellow colouring of the stubs from $R_1$. In total, the number of choices of $R_1$ and a suitable blue/yellow colouring of its stubs is at most
\begin{equation}\label{eq:R1-count}
    \binom{r_1}i\cdot \frac{(\rho_2+\kappa)^{r_1}(O(1))^{i}n^{r_1}}{r_1!}\cdot \binom{2m_1+m_{1,2}}{m_{1,2}}\le O^\ast(1)\left(\frac{en\rho_2}{r_1}\right)^{r_1}
\end{equation}
where we used \cref{lem:binom} twice (with $z=\kappa$) and Stirling's inequality.

Second, we bound the number of choices for $R_2$. Recall that $x$ of the vertices in $R_2$ have degree exactly 1, and there is some $j=O(1+x+f(r))$ such that at least $r_2-x-j$ of the other vertices in $R_2$ have degree exactly 2. By \cref{T2} and \cref{T3}, the number of ways to choose a sequence of $r_2$ vertices, of which the first $x$ have degree exactly 1, and the next $r_2-x-j$ have degree exactly 2, is at most  $O^\ast(1)((t\log(n/t)/n)^x(\rho_2 + \kappa)^{r_2-x-j}n^{r_2}$. The number of choices of $R_2$ is therefore at most
\begin{equation}\label{eq:R2-count}
    \binom{r_2}x\binom{r_2}j\cdot \frac{((t\log(n/t)/n)^x(\rho_2 + \kappa)^{r_2-x}(O(1))^jn^{r_2}}{r_2!}=O^\ast(1)\left(\frac{e\rho_2n}{r_2}\right)^{r_2}\left(\frac{t\log(n/t)}n\right)^x,
\end{equation}
where we used \cref{lem:binom} twice and Stirling's inequality.

Third, the number of ways to choose $S_2$, and to choose which of its stubs are red, is at most
\[\prod_{i=1}^r\frac{1}{k_i!}\left(\sum_{v\in V(G_t)}\binom{d_v}{i}\right)^{k_i}\le \frac{1}{k_2!}(E_2+\kappa)^{k_2}\frac{e^{m_{1,2}+m_{2,2}-2k_2-\ell_2}}{(s_2-k_2-\ell_2)!}(1/\kappa)^{m_{1,2}+m_{2,2}-2k_2-\ell_2}n^{s_2}.\]
Here we used that $\sum_v \binom {d_v}1\le m\lesssim n$ by \cref{T4}, we used \cref{T6}, and we used \cref{lem:combine-bound} applied to $(k_i)_{i\ge 2}$, noting that $\sum_{i=1}^r k_i=s_2$, $\sum_{i=1}^r ik_i=m_{1,2}+m_{2,2}$ and $k_1=\ell_2$.

Now, by \cref{claim:structure}(1,2,5), we have $m_{1,2}+m_{2,2}-2k_2-\ell_2=O(1+x+f(r))$, so the above expression is bounded by 
\begin{equation}\label{eq:S2-count}
    O^\ast(1)\frac{(E_2 n)^{s_2}}{k_2!(s_2-k_2-\ell_2)!}= O^\ast(1)\binom {s_2}{k_2}\frac{(E_2 n)^{s_2}}{s_2!}\le O^\ast(1)\left(\frac{eE_2n}{s_2}\right)^{s_2}.
\end{equation}
(we have used the fact that $s_2 = k_2 + O(1+x + r)$, \cref{lem:binom}, and Stirling's inequality).

Multiplying the expressions in \cref{eq:S2-count,eq:R2-count,eq:R1-count} (counting the number of ways to choose $R_1,R_2,S_2$, and their stub-colourings) shows that
\[N_{\mbf p}\le O^\ast(1)\frac{n^{r_1+r_2+s_2}}{r_1^{r_1}r_2^{r_2}s_2^{s_2}}\cdot e^{s_2}(e\rho_2)^{r_1+r_2}\bigg(\frac{\log(n/t)}{n/t}\bigg)^{\!x} E_2^{s_2}.\]
The desired result follows, noting that $s_2=r_2+O(1+x+f(r))$ by \cref{claim:structure}(1) and hence $(s_2/r_2)^{r_2}=(1+O(1+x+f(r))/r_2)^{r_2}=e^{O(1+x+f(r))}=O^\ast(1)$.
\end{proof}

\medskip
\noindent\textbf{Step 5: Putting everything together.}
Let $E_{\mbf p}=Q N_{\mbf p} P_{\mbf p}$, so
\[\mb EX_{r,s,\ell}\le \max_{\mbf p} E_{\mbf{p}},\quad\mb EX'_{r,s,\ell}\le \max_{\mbf p:x>0} E_{\mbf{p}}.\]
Combining \cref{claim:num-p,claim:Pp,claim:Np}, we have
\[E_{\mbf p}\le O^\ast(1)\bigg(\frac{\log(n/t)}{n/t}\bigg)^{\!x}\rho_2^{r_1}(\rho_2E_2)^{r_2}\frac{n^{r_1+r_2+s_2}r_1^{m_1-r_1}s_2^{m_{1,2}+m_{2,2}-r_2-s_2}}{m^{m_1+m_{1,2}+m_{2,2}}}.
\]
Now, $r_1,s_1\le 2r$, and $m_1-r_1,m_{1,2}+m_{2,2}-r_2-s_2\lesssim 1+x+f(r)$, by \cref{claim:structure}(1,4,5,6). So, if $m_1-r_1$ (respectively, $m_{1,2}+m_{2,2}-r_2-s_2$) is nonnegative, then $\left(\frac{r_1}{r}\right)^{m_1-r_1}=O^\ast(1)$ (respectively, $\left(\frac{s_2}{r}\right)^{m_{1,2}+m_{2,2}-r_2-s_2}=O^\ast(1)$). By \cref{eq:S1-consequence} and \cref{eq:S-consequence}, $m_1-r_1$ and $m_{1,2}+m_{2,2}-r_2-s_2$ can only be very slightly negative (i.e., if either is negative, it is $O(1)$). In such a case, we again have $\left(\frac{r_1}{r}\right)^{m_1-r_1}=O^\ast(1)$ or $\left(\frac{s_2}{r}\right)^{m_{1,2}+m_{2,2}-r_2-s_2}=O^\ast(1)$, respectively. Also, $(m/n)^{s_2-r_2}=e^{O(1+x+f(r))}$. Putting all this together, we further bound
\[E_{\mbf p}\le O^\ast(1)\bigg(\frac{\log(n/t)}{n/t}\bigg)^{\!x}\left(\frac{\rho_2n}{m}\right)^{r_1}\left(\frac{\rho_2E_2n^2}{m^2}\right)^{r_2}\left(\frac{r}{m}\right)^{m_1+m_{1,2}+m_{2,2}-r_1-r_2-s_2}.\]
Recalling the definitions of $\rho_2$ and $E_2$ in terms of a Poisson random variable $Z$, and recalling the choice of the Poisson parameter $\lambda$, we compute
\[\frac{\rho_2n}{m}=\frac{2\Pr[Z=2|Z\ge 2]}{\mb E[Z|Z\ge 2]}=\frac{2\Pr[Z=2]}{\mb E[Z\mbm 1_{Z\ge 2}]}=\frac{2(\lambda^2e^{-\lambda}/2)}{\lambda-\lambda e^{-\lambda}}=\frac{\lambda e^{-\lambda}}{1-e^{-\lambda}}=1-\Omega_\eps(1),\]
and similarly
\[\frac{\rho_2E_2n^2}{m^2}=\frac{4\Pr[Z=2]\,\mb E\binom Z 2}{\mb E[Z\mbm 1_{Z\ge 2}]^2}=\frac{4(\lambda^2 e^{-\lambda}/2)(\lambda^2/2)}{(\lambda-\lambda e^{-\lambda})^2}=\frac{\lambda^2e^{-\lambda}}{(1-e^{-\lambda})^2}= 1-\Omega_\eps(1).\]
Also, combining the inequalities in \cref{eq:S-consequence} in different ways, we can obtain $m_1+m_{1,2}+m_{2,2}\ge s_1+2s_2-\ell=r_1+r_2+s_2-\ell+(s-r)$ and $m_1+m_{1,2}+m_{2,2}\ge r_1+2r_2-\lfloor\ell_1/2\rfloor-x=r_1+r_2+s_2-\lfloor\ell_1/2\rfloor-x+(r-s)$. We deduce
\[E_{\mbf p}\lesssim_{\kappa} e^{-\Omega_\eps(r)}\bigg(O_{\kappa}(1)\frac{\log(n/t)}{n/t}\bigg)^{\!x}\left(\frac{r}{m}\right)^{\max(s-r-\ell,\;r-s-\lfloor\ell_1/2\rfloor-x)}\]
for $\kappa\ll \eps$, recalling the definition $O^\ast(1):= (1+O(\kappa))^re^{O_{\kappa}(1+x+f(r))}\le e^{O(\kappa r)}O_{\kappa}(1)O_{\kappa}(1)^x$ (for the inequality, we are using that $\eta\ll \kappa$, so the ``$f(r)$'' in the exponent is sufficiently small compared to $r$).

We finally break into cases to prove the six different parts of \cref{lem:structure-master}. Observe that since we are assuming $t \leq \alpha n$ for $\alpha\ll\kappa$, we have that $\big(O_{\kappa}(1)\log(n/t)/(n/t)\big)^{\!x} \leq 1$.
\begin{enumerate}
    \item If $s\ge r+1$ and $\ell=0$, then taking the first term in the ``max'' in the exponent yields $\mb EX_{r,s,0}\lesssim_\kappa e^{-\Omega_\eps(r)}(r/m)=e^{-\Omega_\eps(r)}/n$.
    \item If $s = r$ and $\ell=0$, taking the first term in the ``max'', we see $\mb EX_{r,r,0}\lesssim_\kappa e^{-\Omega_\eps(r)}$.
    \item When $s = r$, $\ell=0$, $x > 0$, and $t\le n^{1/8}$, notice that we have $O_{\kappa}(1)t\log(n/t)/n\le n^{-3/4}$. Thus taking the first term in the ``max'' yields $\mb E X_{r,r,0}'\lesssim_\kappa e^{-\Omega_\eps(r)}n^{-3/4}$.
    \item Suppose $s=r-1$ and $\ell = 0$. If $x\ge 2$ then take the first term in the ``max'', and if $x<2$ take the second term. Thus for $x \ge 2$, we have $E_{\mbf p} \lesssim_\kappa e^{-\Omega_\eps(r)}(t\log(n/t)/n)^2\cdot n\lesssim_\kappa e^{-\Omega_\eps(r)}(t/n)^{3/2} n$.
    For $x = 1$, we have $E_{\mbf p} \lesssim_\kappa e^{-\Omega_\eps(r)}(t\log(n/t)/n)\lesssim_\kappa e^{-\Omega_\eps(r)}(t/n)^{3/2} n$ and for $x = 0$, we have $E_{\mbf p} \lesssim_\kappa e^{-\Omega_\eps(r)}\frac{r}{m} \lesssim_{\kappa} e^{-\Omega_\eps(r)}/m \lesssim e^{-\Omega_\eps(r)}(t/n)^{3/2} n$. 
    \item For any $\ell$, by \cref{S1} we have $s-r-\ell\ge -1$, so (taking the first term in the ``max'') we have $\mb E X_{r,s,\ell}\lesssim_\kappa e^{-\Omega_\eps(r)}n$.
    \item If $\ell = 1$ and $t \leq \log n$, notice that $t\log (n/t)/n\le (\log n)^2/n$. If $x > 0$ then take the first term in the ``max'', which is at most $-1$ by \cref{S1}. If $x = 0$, then $\lfloor\ell_1/2\rfloor=0$ means the ``max'' term must evaluate to at least $0$. Together these two cases yield $\mb E X_{r,s,1}\lesssim_\kappa e^{-\Omega_\eps(r)}(\log n)^{2}$. \qedhere
\end{enumerate}
\end{proof}

\subsection{Deductions}
We now deduce \cref{claim:unstructured-whp,claim:small-kernel-vectors-rough,claim:small-kernel-vectors-precise}.

\begin{proof}[Proof of \cref{claim:unstructured-whp}]
By \cref{lem:unstructured-stalk} and \cref{lem:nonzero-level-set}, for $t\le \Delta \leq \log n$ it suffices to prove that $G_t$ has at most $(n-t)/10$ different $(\le\!\eta (n-t),1)$-stalks and at most $(n-t)^2/10$ different $(\le\!\eta (n-t),2)$-stalks, with probability $1-(\log n)^{-\omega(1)}$. By \cref{lem:dense-subset}, we only need to consider $\eps$-sparse stalks. Also, by \cref{lem:good-deg-seq}, with probability $1 - (t/n)^{1/4} \geq 1-(\log n)^{-\omega(1)}$, the degree sequence of $G_t$ is $(n,m,t,\kappa)$-typical. Thus it suffices to prove the result conditional on a particular such degree sequence.

Conditioning on a typical degree sequence, by \cref{lem:structure-master}(6) and (5) respectively, we have the expected numbers of $\eps$-sparse $(\le\!\eta (n-t),1)$-stalks and $(\le\!\eta (n-t),2)$-stalks are at most 
\[\sum_{r=1}^\infty re^{-\Omega(r)}O_{\kappa}(\log n)^2\lesssim_\kappa(\log n)^2\quad\text{and}\quad\sum_{r=1}^\infty re^{-\Omega(r)}O_{\kappa}(n)\lesssim_\kappa n,\]
respectively. The desired result follows from Markov's inequality.
\end{proof}

\begin{proof}[Proof of \cref{claim:small-kernel-vectors-rough}]
Fix $t$. We would like to prove that with probability at least $1-O_{\kappa}(1/t + (n/t)^{1/4})$, we either have $t> N$ (i.e., $t$ is outside our range of consideration), or $\dim(K^{(\eta)}_t) < t/4$.

Notice that (if $t\le N$) we have $\dim(K^{(\eta)}_t)\le |\supp(K^{(\eta)}_t)|$, which by \cref{lem:minimal-stalk} is at most the number of vertices in $(\le\!\eta (n-t),0)$-stalks. By \cref{lem:dense-subset} we only need to worry about $\eps$-sparse stalks, and by \cref{lem:good-deg-seq}, it suffices to prove the result conditioned on a particular $(n,m,t,\kappa)$-typical degree sequence for $G_t$ (note that the degree sequence of $G_t$ determines whether $t\le N$).

Conditioning on a typical degree sequence, by \cref{lem:structure-master}(1,2,4), the expected number of vertices in $\eps$-sparse $(\le\!\eta (n-t),0)$-stalks is at most
\[\sum_{r=1}^\infty O_{\kappa}(re^{-\Omega(r)}+re^{-\Omega(r)}(t/n)^{3/2}n)=O_\kappa(1+(t/n)^{3/2}n).\]
By Markov's inequality, the probability this number is greater than $t/4$ is $O_\kappa((t/n)^{1/2}+1/t)$.
\end{proof}

\begin{proof}[Proof of \cref{claim:small-kernel-vectors-precise}]
In this proof we only consider $t\le \Delta$ (so, for example, ``all $t$'' should be read as ``all $t\le \Delta$''). We prove that each of (a,b,c,d,e) hold with probability at least $1-1/(5\Delta)$. Say a ``special stalk'' is an $(r,r,0)$-stalk for some $r\le \eta (n-t)$.

Let $V_t^\ast$ be the set of vertices contained in an $\eps$-sparse special stalk. By \cref{lem:structure-master}(2), in the setting of \cref{lem:structure-master} (conditioning on a particular typical degree sequence for $G_t$), we have
\[\mb E|V_t^\ast|\lesssim_{\kappa} \sum_{r=1}^\infty re^{-\Omega(r)}=O_\kappa(1).\]
By Markov's inequality and \cref{lem:good-deg-seq}, and a union bound over $t$, with probability at least say $1-1/(10\Delta)$ each $|V_t^\ast|\le O_{\kappa,\Delta}(1)$. So, by \cref{lem:dense-subset}, with probability at least say $1-1/(9\Delta)$, for each $t$ there are at most $O_{\kappa,\Delta}(1)$ vertices in special stalks. For any special cycle of length $4k\le 2\eps(n-t)$, there is a $(2k,2k,0)$-stalk containing half its vertices (i.e., a special stalk), so this takes care of (d).

Similarly, by \cref{lem:structure-master}(1,2,3,4) together with Markov's inequality and \cref{lem:good-deg-seq}, with probability $1/(10\Delta)$ the only $0$-stalks in any $G_t$ are $(r,r,0)$-stalks which do not contain any degree-1 vertices, for some $|R|=O_{\kappa,\Delta}(1)$. The union of any two non-disjoint cycles has strictly more edges than vertices (since in such a union every vertex has degree at least 2, and some vertex has degree strictly greater than 2). So, given the above event, if two of the special cycles in $G_t$ were not vertex disjoint, they would provide a set of $O_{\kappa,\Delta}(1)=o_{\kappa,\Delta}(\sqrt{\log n})$ vertices contradicting \cref{lem:dense-subset} (which holds with probability $1-n^{-\omega(1)}$). So, \cref{lem:corank-s} takes care of (a,b).

For (c,e), note that by \cref{lem:dense-subset}(A2) and \cref{lem:degree-sequence-poisson}(A4), there are at most say $\exp((\log\log n)^4)$ edges (and thus, vertices) in $G$ which are in a cycle of length at most $\log \log n$ or adjacent to such a cycle. We can reveal these ``dangerous'' vertices in $G$ without revealing the random ordering $v_{N-1},\ldots,v_0$ of the vertices in $T$ (recall that the vertices of $T$ are deleted then added back in some random order to form our sequence of graphs $G_N,\ldots,G_0$). With probability $1 - 1/(10\Delta)$, none of the vertices $v_{\Delta},\ldots,v_0$ is dangerous (indeed, the expected number of such dangerous vertices is $\Delta n^{-1+o(1)} \leq 1/(10\Delta^2)$, so Markov's inequality yields this). This handles (e). If (e) holds, the special cycles of length at most $\log \log n$ are completely unaffected by the vertex additions defining the sequence $G_{\Delta},\ldots,G_0$; this takes care of (c), recalling that with probability at least $1-1/(10\Delta)$ each $|V_t^{\mr{spec}}|\le O_{\kappa,\Delta}(1)\le \log \log n$.
\end{proof}

\section{The bipartite case}\label{sec:bipartite-sketch}
Having just proved \cref{thm:main-RMT}(A1), we now sketch the changes that must be made for a proof of (B1). The proof strategy is extremely similar, but there are some minor simplifications and complications. The most notable simplification is that we can use \cref{lem:corank-boosting-bipartite} instead of the more sophisticated \cref{lem:corank-boosting-master}, and the primary complication is that a small amount of extra notation and bookkeeping becomes necessary, due to the fact that we need to pay attention to both right and left kernels (i.e., the kernels of our matrix $B$ and its transpose $B^\intercal $).

Recall that in the setting of \cref{thm:main-RMT}(A1), we had a set $T$ of high-degree vertices (coming from \cref{lem:initial-rank}(A)). We ``extracted'' this set (and used \cref{lem:initial-rank}(A1) to control the rank of the resulting matrix), then added back these vertices one-by-one in a random order (each such addition corresponds to the addition of a new row and column), studying how the rank changes during this process.

In the setting of \cref{thm:main-RMT}(B1), after sampling $G\sim\mc{K}(n_1,n_2,m,2)$ and using the setup of \cref{lem:initial-rank}(B) we will now have \emph{two} sets $T_1,T_2$ of high-degree vertices (whose sizes are almost the same). We let $\mbf T\subseteq T_1\times T_2$ be a set of $\min(|T_1|,|T_2|)$ disjoint \emph{pairs} of vertices from $T_1\times T_2$, which will play the role of $T$ above. Indeed, let $G[V\setminus \mbf T]$ be the (balanced) bipartite graph obtained from $G$ by removing the vertices in the pairs in $\mbf T$, so by \cref{lem:initial-rank}(B1,B3) we have $\corank G[V\setminus \mbf T]\le |T|/15$. The plan is then to add back the pairs in $\mbf T$ pair-by-pair in a random order (each such addition corresponds to the addition of a new row for the $T_1$-vertex and the addition of a new column for the $T_2$-vertex).

Similarly to \cref{sec:RMT-overview}, we define $N=|\mbf T|$, and let $G_t$ be the graph that results after $t$ of the pairs in $\mbf T$ have been added back. Instead of just defining the ``small-support kernel''  $K^{(\eta)}_t$, we now need both a right and left version:
\[K^{(\eta)}_t=\{\mathbf{v}\in\ker B(G_{t}):|\on{supp}(x)|\le\eta (n_2-t)\},\quad Q^{(\eta)}_t=\{\mathbf{v}\in\ker B(G_{t})^\intercal :|\on{supp}(x)|\le\eta (n_1-t)\}.\]

Then, it is straightforward to prove bipartite analogues to \cref{claim:neighbourhood-symmetry,claim:nonzero-level-sets,claim:small-kernel-vectors-rough,claim:small-kernel-vectors-precise} (we have no need for an analogue of \cref{claim:unstructured-whp}, because the corank-boosting part is now simpler). Specifically, our analogue of \cref{claim:neighbourhood-symmetry} should say that after appropriate revelations the neighbourhoods of vertices in $T_1$ are uniformly random subsets of $W\cap V_2$, and the neighbourhoods of the vertices in $T_2$ are uniformly random subsets of $W\cap V_1$ (of the appropriate sizes), all independent of each other. Our analogue of \cref{claim:nonzero-level-sets} should hold for both $B(G_t)$ and $B(G_t)^\intercal $, and our analogues of \cref{claim:small-kernel-vectors-rough,claim:small-kernel-vectors-precise} should hold for both $K^{(\eta)}_t$ and $Q^{(\eta)}_t$ (for \cref{claim:small-kernel-vectors-precise}, $K^{(\eta)}_t$ should be described in terms of 2-special cycles, and $Q^{(\eta)}_t$ should be described in terms of 1-special cycles). There are no additional difficulties in the proofs of any of these claims. Actually, things are slightly simpler: we remark that the bipartite analogue of a \emph{stalk} in \cref{sec:stalks-computation} should be defined to be a set of vertices $S$ contained on just one side of our bipartite graph, so there can be no intersection between $S$ and its neighbourhood $N(S)$; this simplifies the calculations in \cref{claim:structure}.

Now, recall that in the proof of \cref{thm:main-RMT}(A1) we considered a random walk defined by random variables of the form $\dim\ker A(G_{t})-\mbm 1_{t\le \Delta} \dim{K_\Delta^{(\eta)}}-t/4$.
For \cref{thm:main-RMT}(B1) we need a similar definition that takes both sides of our bipartite graph into account: let
\[
X_{t}=\min\left(\dim\ker B(G_{t})-\mbm 1_{t\le \Delta} \dim{K_\Delta^{(\eta)}},\;  \dim\ker B(G_{t})^\intercal -\mbm 1_{t\le \Delta} \dim{Q_\Delta^{(\eta)}}\right)-t/4.
\]
Actually, it turns out that only the first term of the ``min'' is really necessary: recall that we are assuming $n_1-n_2\to \infty$, and note that $n_2-\dim\ker B(G_{t})=\rank B(G_t)=n_1-\dim\ker B(G_{t})^\intercal $.
Also, by (a bipartite analogue of) \cref{claim:small-kernel-vectors-rough}, whp $\dim{K_\Delta^{(\eta)}}$ and $\dim{Q_\Delta^{(\eta)}}$ are of the form $o(n_1-n_2)$. So whp we actually have
\[
X_{t}=\dim\ker B(G_{t})-\mbm 1_{t\le \Delta} \dim{K_\Delta^{(\eta)}}-t/4.
\]
for all $t$. The above reasoning also shows that whp for all $t\le \Delta$, we have $\dim \ker B(G_t)^\intercal -\dim{Q_\Delta^{(\eta)}}>0$, i.e., $B(G_t)^\intercal $ has a kernel vector $\mbf v\in \mb R^{n_2}$ with $|\supp (\mbf v)|\ge \eta n_1$.

Now, we apply \cref{thm:random-walk-master} in basically the same way as for the proof of \cref{thm:main-RMT}(A1). We say an index is \emph{good} if it satisfies the natural analogues of \cref{G1}, \cref{G2}, \cref{G3}, \cref{G4}, \cref{G5}, \cref{G7} (where \cref{G3} needs to hold for both $B(G_t)$ and $B(G_t)^\intercal $, and \cref{G2,G4,G5,G7} need to hold for both $K^{(\eta)}_t$ and $Q^{(\eta)}_t$), and if the following property holds (c.f., the discussion in the previous paragraph):
\begin{enumerate}[{\bfseries{G\arabic{enumi}}}]
    \setcounter{enumi}{7}
    \item \label{Gspec} If $t\le \Delta$, then $B(G_t)^\intercal $ has a kernel vector $\mbf v\in \mb R^{n_2}$ with $|\supp (\mbf v)|\ge \eta n_1$.
\end{enumerate}

Now, the rest of the proof of \cref{thm:main-RMT}(A1) basically translates directly into a proof of \cref{thm:main-RMT}(A1), with the exception that we need to replace the applications of \cref{lem:corank-boosting-master} with applications of \cref{lem:corank-boosting-bipartite}. Specifically:
\begin{itemize}
    \item In ``Case 1'': instead of \cref{lem:corank-boosting-master}(a) we apply \cref{lem:corank-boosting-bipartite} twice to increase the rank by 2 (first we add our new column to obtain a matrix $B'$, then we view our new-row-addition as adding a column to $(B')^\intercal $).
    \item In ``Case 2'': instead of \cref{lem:corank-boosting-master}(b) we add our new column and apply \cref{lem:corank-boosting-bipartite} (using \cref{Gspec}), then note that adding an additional row cannot decrease the rank.
\end{itemize}

\section{Asymptotic distributions}\label{sec:distributions}
In this section we prove the central limit theorem in \cref{cor:CLT}, and discuss the (Poisson-type) asymptotic distributions of various quantities in \cref{thm:main-RMT,thm:matching-vs-rank}.

First, we prove \cref{cor:CLT}.
\begin{proof}[Proof of \cref{cor:CLT}]
For $p=c/n$ with $c<1$ or $c>e$, the matching number $\nu(G)$ of a random graph $G\sim\mathbb{G}(n,c/n)$ is known to satisfy a central limit theorem: there are $\mu=\mu(c,n)$ and $\sigma=\sigma(c,n)$ (where $\mu$ and $\sigma^2$ both have order of magnitude $n$) such that $(\nu(G)-\mu)/\sigma\overset d \to \mc N(0,1)$. For $c<1$ this is due to Pittel~\cite{Pit90}, and for $c>e$ this is due to Krea\v{c}i\'c~\cite[Theorem~19]{Kre17}.

By \cref{thm:matching-vs-rank}(A1), we have $2\nu(G)-\rank A(G)=o(\sqrt n)$ whp, which implies that $X=\rank A(G)$ satisfies the same central limit theorem as $\nu(G)$.

Strictly speaking, it remains to show that we also have $(\rank A(G)-\mb E X)/\sqrt{\on{Var} X}\to \mc N(0,1)$. Indeed, \emph{a priori}, there may be no connection between $\mu$ and $\mb E X$ or between $\sigma^2$ and $\on{Var}X$, if the mean or variance of $X$ is dominated by the effect of outliers. To rule out such pathological behaviour, we need the well-known observation (easily proved with the Azuma--Hoeffding martingale concentration inequality; see for example the appendix of \cite{BLS11}) that the rank of a random matrix is subgaussian with ``variance proxy'' $O(n)$ (and thus the tails have negligible contribution to the mean and variance).
\end{proof}

Next we prove \cref{thm:main-RMT}(A2). We omit the proof of \cref{thm:main-RMT}(B2), as it follows from an easier version of the same argument.

First, we need expressions for certain infinite sums, which can both be obtained by manipulating the Taylor series $\log(1-x)=-\sum_{k=1}^\infty x^k/k$.
\begin{lemma}\label{lem:sum-formulas}
Let $Z\sim\operatorname{Poisson}(\lambda)$.
\begin{enumerate}
\item[(A)] If $Q\in \mb R$ satisfies $|Q|\lambda<e^{\lambda/2}-e^{-\lambda/2}$ then
\[
\sum_{k=1}^{\infty}\frac{1}{4k}\cdot \frac{(2Q^{2}\Pr[Z=2|Z\ge2]\mb E[Z(Z-1)|Z\ge2])^{2k}}{\mb E[Z|Z\ge2]^{4k}}=-\frac{1}{4}\log\left(1-\left(\frac{Q\lambda}{e^{\lambda/2}-e^{-\lambda/2}}\right)^{4}\right).
\]
Also, for all $\lambda>0$ we have $\lambda<e^{\lambda/2}-e^{-\lambda/2}$ (i.e., the above holds for $Q$ sufficiently close to 1).
\item[(B)] If $Q\in \mb R$ satisfies $|Q|\lambda<e^{\lambda}-1$ then
\[
\sum_{k=1}^{\infty}\frac{1}{8k}\left(\frac{2Q\Pr[Z=2|Z\ge2])}{\mb E[Z|Z\ge2]}\right)^{4k}=-\frac{1}{8}\log\left(1-\left(\frac{Q\lambda}{e^{\lambda}-1}\right)^{4}\right).
\]
Also, for all $\lambda>0$ we have $\lambda<e^{\lambda}-1$.
\end{enumerate}
\end{lemma}

\begin{lemma}\label{lem:large-special-cycles}
In the setting of \cref{thm:main-RMT}(A), for any $M\to\infty$, whp there are no special cycles of length at least $4M$.
\end{lemma}

\begin{proof}
First, we need to separately rule out extremely long special cycles. One could perform a configuration model calculation, but it is convenient to borrow from the proof of \cref{thm:main-RMT}: right at the end of \cref{sec:RMT-overview} (at the end of the proof of \cref{thm:main-RMT}(A1)): we proved that, for an arbitrarily small constant $\eta$, whp there is no special cycle longer than $2\eta n$ (in the notation of that section, $s(G)=s^{(\eta)}(G)$). Taking $\eta\to0$ sufficiently slowly, it now suffices to consider special cycles of length $o(n)$.

The remaining long cycles of length $o(n)$ can actually also be handled by borrowing from the proof of \cref{thm:main-RMT} (specifically, from \cref{lem:good-deg-seq,lem:structure-master}(2)). However, as a warm-up to more involved calculations that will appear later in the proof of \cref{thm:main-RMT}(A2), we perform an explicit configuration model calculation.

Condition on a degree sequence satisfying the properties in \cref{lem:degree-sequence-poisson}(A). We compute with the
configuration model (which suffices, by \cref{lem:simplicity}). Let $n'=\Pr[Z=2|Z\ge2]n+o(n)$ be the number
of degree-2 vertices (using \cref{lem:degree-sequence-poisson}(A1)). Using \cref{lem:degree-sequence-poisson}(A2), the number of configurations of $4k$ pairs that correspond to a special cycle of length $4k$ is
\begin{align*}
\frac{(n')_{2k}2^{2k}}{4k}\sum_{\substack{v_{1},\ldots,v_{2k}\\\text{ distinct}}
}\prod_{i=1}^{2k}d_{v_{i}}(d_{v_{i}}-1)
&\le \frac{1}{4k}\left( 4n'\sum_v \binom {d_v}2\right)^{2k}\\
&\le \frac{1}{4k}\left(\vphantom{\sum}(2+o(1))n\Pr[Z=2|Z\ge2]\,\mb E[Z(Z-1)|Z\ge2]\right)^{2k}
\end{align*}
(this counts isolated cycles twice). For $k=o(n)$, the probability such a configuration
actually appears is
\[
\frac{(2m-8k-1)!!}{(2m-1)!!}=\frac{(m)_{4k}2^{4k}}{(2m)_{8k}}=\left(\frac{1+o(1)}{\mb E[Z|Z\ge 2]}\right)^{4k}.
\]
We compute that the expected number of special cycles with length at least $4M$ (and at most $o(n)$) is at most
\[
\sum_{k=M}^{o(n)}\frac{\left(\vphantom{\sum}(2+o(1))\Pr[Z=2|Z\ge2]\,\mb E[Z(Z-1)|Z\ge2]\right)^{2k}}{4k(\mb E[Z|Z\ge2])^{4k}}.
\]
By \cref{lem:sum-formulas}, this tends to zero as $M\to\infty$, so the desired result follows from Markov's inequality.
\end{proof}

We are now ready to prove \cref{thm:main-RMT}(A2).

\begin{proof}
[Proof of \cref{thm:main-RMT}(A2)]Recall that an induced cycle is special if its length is
divisible by 4 and every second vertex has degree 2. Say that the
cycle is \emph{weakly special} if it is not isolated (i.e., if it has at least one vertex whose degree is not 2).

We work in the configuration model, conditioning on a degree sequence satisfying the properties in \cref{lem:degree-sequence-poisson}. Again, let $n'=\Pr[Z=2|Z\ge2]n+o(n)$ be the number
of degree-2 vertices. Let $N_{k}$ be the number of
weakly special cycles of length $4k$, and let $N_{k}^{\dagger}$ be the number
of isolated cycles of length $4k$. Let $A_{1}$ and $A_{2}$ be the
numbers of loops and 2-cycles (so the random multigraph produced by the configuration model is simple if and only if $A_1=A_2=0$). Also, let
\[
\gamma_k^\dagger=\frac{1}{8k}\left(\frac{2\Pr[Z=2|Z\ge2])}{\mb E[Z|Z\ge2]}\right)^{4k},\quad\gamma_{k}=\frac{1}{4k}\cdot\frac{(2\Pr[Z=2|Z\ge2]\mb E[Z(Z-1)|Z\ge2])^{2k}}{\mb E[Z|Z\ge2]^{4k}}-2\gamma_k^\dagger,
\]
\[
\eta_{1}=\frac{1}{2}\cdot\frac{\mb E[Z(Z-1)|Z\ge2]}{\mb E[Z|Z\ge2]},\quad\eta_{2}=\frac{1}{4}\left(\frac{\mb E[Z(Z-1)|Z\ge2]}{\mb E[Z|Z\ge2]}\right)^{2}.
\]
Let $(W_{1},W_{2})$, $(Y_{k})_{k=1}^{\infty}$ and $(Y_{k}^{\dagger})_{k=1}^{\infty}$
be independent sequences of independent Poisson random variables with
$\mb E W_{i}=\eta_{i}$, $\mb E Y_{k}=\gamma_{k}$, and $\mb E Y_{k}^{\dagger}=\gamma_k^\dagger$. 

Abusing notation, we write $\cup$ for concatenation of sequences.
We claim that for any $M$, we have
\[
(A_{1},A_{2})\cup(N_{k})_{k=1}^{M}\cup(N_{k}^{\dagger})_{k=1}^{M}\overset{d}{\to}(W_{1},W_{2})\cup(Y_{k})_{k=1}^{M}\cup(Y_{k}^{\dagger})_{k=1}^{M}.
\]
(where here $n$ goes to infinity while $M$ is fixed). In combination with \cref{lem:large-special-cycles} and \cref{lem:simplicity}, this suffices to prove \cref{thm:main-RMT}(A2). Indeed, \cref{lem:sum-formulas} then shows that the number of weakly special cycles converges in distribution to a Poisson random variable with parameter $\gamma(c)$, and the number of isolated cycles converges to a Poisson random variable with parameter $\gamma^\dagger(c)$.

To this end, by the method of moments (see for example \cite[Lemma~2.8]{Wor99}) it suffices to
prove that for any $(s_{1},s_{2})\cup(r_{k})_{k=1}^{M}\cup(r_{k}^{\dagger})_{k=1}^{M}$ (which we treat as fixed, while $n\to \infty$),
we have
\[
\mb E\left[(A_{1})_{s_{1}}(A_{2})_{s_{2}}\prod_{k=1}^{M}(N_{k})_{r_{k}}(N_{k}^{\dagger})_{r_{k}^{\dagger}}\right]\to\eta_1^{s_1}\eta_2^{s_2}\prod_{k=1}^{M}\gamma_{k}^{r_{k}}(\gamma_k^\dagger)^{r_{k}^{\dagger}}.
\]

Note that $(A_{1})_{s_{1}}(A_{2})_{s_{2}}\prod_{k=1}^{M}(N_{k})_{r_{k}}(N_{k}^{\dagger})_{r_{k}^{\dagger}}$
is the number of (ordered) collections of distinct cycles, containing
$s_{1}$ loops, $s_{2}$ 2-cycles, $r_{k}$ weakly special cycles of
each length $4k$ and $r_k^\dagger$, isolated cycles of each length $4k$.

If such a collection does not consist of vertex-disjoint cycles, then
the union of this collection has strictly more edges than vertices
(because every vertex has degree at least 2, and some vertex has degree
strictly greater than 2). The expected number of such collections
is therefore $O(1/n)$. So, it suffices to consider the contribution
from collections of vertex-disjoint cycles.

In such a collection of disjoint cycles, let $x=s_{1}+2s_{2}$ be the number of vertices that should be in loops and 2-cycles (which can have any degree), let $y=\sum_{k=1}^{M}(2kr_{k}+4kr_{k}^{\dagger})$ be the number of vertices that should have degree 2 in special cycles, and let $t=x+y+\sum_{k=1}^{M}2kr_{k}$ be the total number of vertices (recall that in weakly special cycles, not all vertices have degree 2, so there is a mild restriction on the degrees of the vertices not counted by $x$ and $y$). The number of ways to choose configurations forming an appropriate collection
of disjoint cycles is
\[
\frac{(n')_{y}2^{y}}{2^{s_{1}}4^{s_{2}}\prod_{k=1}^{M}(4k)^{r_{k}}(8k)^{r_{k}^{\dagger}}}\left(\;{\sideset{}{^\ast}\sum_{v_{1},\ldots,v_{t-y}}}\prod_{i=1}^{t-y}d_{v_{i}}(d_{v_{i}}-1)\right)
\]
where we fix a set $V^\ast$ of $y$ degree-2 vertices, and the sum with a ``$\ast$'' is over all tuples of distinct
vertices $v_{1},\ldots,v_{t-y}\notin V^\ast$, satisfying the following condition.
After the first $x$ vertices, if we group the vertices into consecutive
blocks of lengths $2,\ldots,2,4,\ldots,4,6,\ldots,6,\ldots,2M,\ldots,2M$
(where there are $r_{k}$ blocks of each length $2k$), then each block
has at least one non-degree-2 vertex (and therefore its corresponding
special cycle is only weakly special). Using \cref{lem:degree-sequence-poisson}(A4) with $s=1$, we see that the distinctness restriction on the $v_{1},\ldots,v_{t-y}$ makes essentially no difference (and similarly with the condition that they do not lie in $V^\ast$). Thus, observe that 
\[
{\sideset{}{^\ast}\sum_{v_{1},\ldots,v_{t-y}}}\prod_{i=1}^{t-y}d_v(d_v-1)=\left(\sum_{v\in V(G)}d_v(d_v-1)\right)^{x}\prod_{k=1}^{M}\left(\left(\sum_{v\in V(G)}d_v(d_v-1)\right)^{2k}-(n')^{2k}2^{2k}\right)^{r_{k}}+o(n^{t-y}).
\]
We may estimate the above sum using the property in \cref{lem:degree-sequence-poisson}(A2). Now, for any of our configurations of disjoint cycles, the probability such a configuration actually appears is
\[
\frac{(2m-t)!!}{(2m)!!}=\frac{1+o(1)}{(n\mb E[Z|Z\ge2])^t}.
\]
The desired result follows, using our expressions for $n'$ and $\sum_{v\in V(G)}d_v(d_v-1)$ and simplifying.
\end{proof}

\subsection{Further comments on asymptotic distributions}\label{sub:further-comments}
We finish this section with some discussion of the asymptotic distributions of $\nu(G)-\rank B(G)$ in the setting of \cref{thm:matching-vs-rank}(B), and $2\nu(G)-\rank A(G)$, $\sigma(G)-\rank A(G)$ in the setting of \cref{thm:matching-vs-rank}(A).

First, we believe that in the setting of \cref{thm:matching-vs-rank}(B), whp $\nu(G)$ attains the Karp--Sipser bound in \cref{cor:KS-bounds}(B) exactly (meaning that $\nu(G)=n-\max(i_1(G),i_2(G))$). To prove this in the same way as \cref{thm:matching-characterisation}, one needs to prove that in the setting of \cref{thm:main-RMT}(B), we have $\nu(G)=n_2$ whp (i.e., there is a matching saturating the smaller side of the bipartite graph). A very similar statement was proved in a difficult paper of Frieze~\cite{Fri05}, and we believe that the ideas in his paper are also applicable to our setting. If this were true, then the asymptotic distribution of $\nu(G)-\rank B(G)$ would be precisely as described in \cref{thm:rank-characterisation-distribution}(B).

Similarly, we believe that in the setting of \cref{thm:matching-vs-rank}(A), whp $\sigma(G)$ attains the Karp--Sipser bound in \cref{cor:KS-bounds}(A2) exactly (meaning that $\sigma(G)=n-i(G)$). To prove this in the same way as \cref{thm:matching-characterisation}, one needs to prove that in the setting of \cref{thm:main-RMT}(A), we have $\sigma (G)=n$ (i.e., there is a collection of vertex-disjoint cycles and edges covering the entire graph). We believe that such a collection can almost entirely consist of edges (i.e., it is essentially a matching), but odd cycles must be included if $G$ has isolated odd cycles, and an additional odd cycle may be necessary for parity reasons. We believe that the ideas by Frieze and Pittel~\cite{FP04} on matchings in degree-constrained random graphs should be suitable to prove this (in fact, it may be possible to deduce the desired statement from the main result of \cite{FP04} in a black-box manner). If this were true, then the asymptotic distribution of $\sigma(G)-\rank A(G)$ would be precisely as described in \cref{thm:rank-characterisation-distribution}(A).

In the setting of \cref{thm:matching-vs-rank}(A), we already have a characterisation of $\nu(G)$ from \cref{thm:matching-characterisation}. A routine calculation in the configuration model (similar to the one used to prove \cref{thm:main-RMT}(A2) earlier in this section) shows that the number of isolated odd cycles in $\on{core}_\mr{KS}(G)$ is asymptotically independent from $s(G)$. If $c>e$, it is not hard to see that the parity of the size of the giant component of $\on{core}_\mr{KS}(G)$ is asymptotically independent from these two quantities. We believe that the ideas in \cite[Section~7]{CCKLRb} can be used to prove that this parity is asymptotically equidistributed, in which case the asymptotic distribution of $2\nu(G)-\rank A(G)$ would be $Y+2Y^{\dagger}-W-\mbm 1_{c>e} U$, where $Y,Y^{\dagger}$ are as in \cref{thm:rank-characterisation-distribution}(A), and, independently, $U$ is uniform on $\{0,1\}$ and $W$ is Poisson with mean
\[\sum_{k=1}^\infty \frac{1}{2(2k-1)}\cdot \left(\frac{2\Pr[Z=2|Z\ge2])}{\mb E[Z|Z\ge2]}\right)^{2k-1}=\frac{1}{4}\left(\log\left(1 + \frac\lambda{e^\lambda-1}\right) - \log\left(1 - \frac\lambda{e^\lambda-1}\right)\right).\]

\bibliographystyle{amsplain0.bst}
\bibliography{main.bib}

\appendix
\section{Analysing the subcritical Karp--Sipser process}\label{sec:AFP}
Here we briefly sketch the analysis in \cite{AFP98} used to prove \cref{lem:subcritical-cycles}(A): if $c<e$, then the Karp--Sipser core whp consists of a collection of vertex-disjoint cycles, and the numbers of cycles of each length are asymptotically jointly Poisson distributed. As will become clear, one can prove \cref{lem:subcritical-cycles}(B) (i.e., the bipartite case of the same fact) with essentially the same analysis.

The main part of the proof is an analysis of the Karp--Sipser leaf-removal process. This analysis is slightly simpler on a random multigraph than a random graph: instead of $\mb G(n,c/n)$, we consider a random multigraph whose edges correspond to a sequence of exactly $\lfloor cn/2\rfloor$ pairs of vertices, sampled uniformly at random with replacement (results about such random multigraphs can be transferred to random graphs, as observed in \cite[Lemma~1]{AFP98}). At each step of the leaf-removal process, we consider the number $v_0$ of isolated vertices, the number $v_1$ of degree-1 vertices, the number $v$ of vertices of degree at least 2, and the number $m$ of edges remaining. These statistics $(v_0,v_1,v,m)$ can be shown to evolve as a Markov chain (\cite[Lemma~3]{AFP98}).

The authors study the typical trajectory of these statistics as the process evolves, using the \emph{differential equations method}. Namely, they first study the expected change in each of $v_0,v_1,v,m$ after a single step of the leaf-removal process, in terms of the statistics $v_0,v_1,v,m$ themselves (\cite[Lemmas~6 and~7]{AFP98}). These expected one-step changes approximately correspond to a system of differential equations (solved in \cite[Lemma~8]{AFP98}), and it can then be shown that whp the trajectories of the evolving statistics $v_0,v_1,v,m$ are well-approximated by the solution to this system of differential equations (\cite[Lemma~11]{AFP98}).

Specifically, to study the expected change after a single step of leaf-removal, the authors use the fact that at any time $t$, the distribution of the remaining multigraph is uniform among all multigraphs with statistics $(v_0,v_1,v,m)$ (\cite[Lemma~2]{AFP98}). Apart from the $v_0+v_1$ vertices of degree 0 and 1, the degrees of the remaining vertices are then shown to be well-approximated by a sequence of truncated Poisson random variables (with a particular Poisson parameter $z$ defined in terms of $v_0,v_1,v,m$; see \cite[Lemmas~4 and~5]{AFP98}), and this degree information can be used to estimate the expected 1-step changes in the various statistics (in the leaf-removal process, if we delete a leaf $x$ with neighbour $y$, then the change to $v_0,v_1,v_2,m$ can be described in terms of the degrees of the neighbours of $y$).

As the process continues, the Poisson parameter $z$ evolves with $v_0,v_1,v,m$. Differential equation heuristics suggest that if $c<e$ then $z$ converges to zero as the process reaches completion. Actually, it is convenient to parameterise the process by $z$: \cite[Lemma~11]{AFP98} allows one to control the trajectories of $v_0,v_1,v,m$ (showing that they are well-approximated by differential equation heuristics) until say $z< n^{-0.1}$. At this point, almost all of the $v$ vertices with degree at least 2 in fact have degree exactly 2. The number of degree-1 vertices $v_1$ is about $n(1-\eta)z^2/c$ and the number of degree-2 vertices is about $n\eta z^2/(2c)$, where $\eta$ is the solution to the equation $c=\eta e^\eta$ (see \cite[Eqs.~(79)--(90)]{AFP98}).

Although we are still some way from the end of the process (there are still a lot of degree-1 vertices remaining), the key observation is that it is already possible to see what the final Karp--Sipser core will end up looking like. Indeed, since there are so few vertices with degree 3 or greater, by a configuration-model calculation, it is easy to see that whp there are no ``heavy cycles'' containing a degree-3 vertex (see \cite[Eqs.~(91)--(93)]{AFP98} and the following discussion). So, at this stage, the connected components are trees and isolated cycles, meaning that the Karp--Sipser core will consist precisely of the (disjoint) cycles which still exist at this stage. In the rest of \cite[Section~5.1]{AFP98}, the authors then use the method of moments (in a similar way to the calculations in the proof of \cref{thm:main-RMT}(A2) in \cref{sec:distributions}) to show that the number of cycles of length $k$ is asymptotically Poisson with mean $\eta^k/(2k)$ (independently for each $k$). Roughly speaking, given a typical outcome of the degree sequence (with $N\approx z^2 n/c$ stubs in the corresponding configuration model), the number of possible sets of $k$ configuration-edges corresponding to a $k$-cycle is about $(\eta N)^k/(2k)$, and the probability a given $k$-cycle is present is about $1/N^k$.

In the bipartite case, we can perform essentially exactly the same differential-equations-method calculation to track the evolution of the number of edges $m$, the numbers of isolated vertices $v_0^{(1)},v_0^{(2)}$ \emph{on each side}, the numbers of degree-1 vertices $v_1^{(1)},v_1^{(2)}$ on each side, and the numbers $v^{(1)},v^{(2)}$ of vertices of degree at least 2 on each side. Actually, the differential equations are exactly the same (where $v_0^{(1)}, v_0^{(2)}$ both take essentially the same value $v_0$, and $v_1^{(1)}, v_1^{(2)}$ take essentially the same value $v_1$, and $v^{(1)},v^{(2)}$ take essentially the same value $v$), because $\mb G(n,c/n)$ and $\mb G(n,n,c/n)$ locally ``look the same'' (locally, they both look like a $\on{Poisson}(c)$ Galton--Watson tree). Near the end of the process, the number of degree-1 vertices on each side is about $n(1-\eta)z^2/c$ and the number of degree-2 vertices on each side is about $n\eta z^2/(2c)$. Given such a degree sequence, for even $k$ the number of possible $k$-cycles is about $(\eta N)^k/k$, where $N\approx z^2 n/c$ is the approximate number of stubs on each side in the corresponding configuration model, and the probability a given $k$-cycle is present is about $1/N^k$. So, the expected number of $k$-cycles is $\eta^k/k$.

\end{document}